\documentclass[11pt]{article}

\usepackage[a4paper,margin=1in]{geometry}
\usepackage{amsmath,amssymb,amsthm,mathtools,mathrsfs}
\usepackage{enumitem}
\usepackage{microtype}
\usepackage[hidelinks]{hyperref}
\usepackage[nameinlink,capitalise]{cleveref}
\usepackage{setspace}

\numberwithin{equation}{section}
\newtheorem{theorem}{Theorem}[section]
\newtheorem{proposition}[theorem]{Proposition}
\newtheorem{lemma}[theorem]{Lemma}
\newtheorem{corollary}[theorem]{Corollary}

\theoremstyle{definition}
\newtheorem{definition}[theorem]{Definition}

\theoremstyle{remark}
\newtheorem{remark}[theorem]{Remark}

\newcommand{\dd}{\,d}

\newcommand{\starprod}{\mathbin{\odot}}
\newcommand{\cH}{\mathcal H}
\newcommand{\cC}{\mathcal C}

\providecommand{\FiniteChainHardyProofLocation}%
  {Appendix~\ref{app:finite-chain-hardy}}
\providecommand{\PartitionApproximationProofLocation}%
  {Appendix~\ref{app:partition-approximation}}
\providecommand{\ExtendedRegularisationProofLocation}%
  {Appendix~\ref{app:extended-regularisation}}
\providecommand{\CompactLinearTransferProofLocation}%
  {Appendix~\ref{app:compact-linear-transfer}}
\providecommand{\StrictCopsonReversalProofLocation}%
  {Appendix~\ref{app:copson-reversal}}
\providecommand{\DirectedMeasureExtensionLocation}%
  {Appendix~\ref{app:partition-approximation}, Subsection~\ref{subsec:appC-directed-extension}}

\providecommand{\HardyCopsonSectionReference}%
  {Section~\ref{sec:hardy-copson}}
\providecommand{\LinearHardySectionReference}%
  {Section~\ref{sec:linear-hardy}}
\providecommand{\LocalCharacteristicsSectionReference}%
  {Section~\ref{sec:local-characteristics}}
\providecommand{\CompactMeasureHardyCopsonTheoremReference}%
  {Theorem~\ref{thm:compact-measure-hardy-copson}}

\title{Endpoint-Safe Direct Characterisations of Weighted
Bilinear Hardy Inequalities on Arbitrary Intervals}

\author{Saikat Kanjilal\\
\small Department of Mathematics, Faculty of Science, JIS University\\
\small Kolkata--700109, West Bengal, India\\
\small Email: \href{mailto: saikat.kanjilal.07@gmail.com}%
{\texttt{saikat.kanjilal.07@gmail.com}} \\
\small ORCID: \href{https://orcid.org/0000-0002-4359-8343}%
{0000-0002-4359-8343}}

\date{}

\hypersetup{
  pdftitle={Endpoint-Safe Direct Characterisations of Weighted Bilinear Hardy Inequalities on Arbitrary Intervals},
  pdfauthor={Saikat Kanjilal},
  pdfsubject={Weighted bilinear Hardy inequalities},
  pdflang={en-GB},
  pdfkeywords={weighted bilinear Hardy inequality, two-weight Hardy inequality, Hardy-Copson operator, Lebesgue-Stieltjes measure, endpoint exponent, compact exhaustion}
}

\begin{document}

\singlespacing

\maketitle
\begin{abstract}
We study the weighted bilinear Hardy inequality for the same-direction product
\(H_If\,H_Ig\) on an arbitrary real interval \(I\), with
\(0<q<\infty\), \(1\le p_1,p_2\le\infty\), and measurable weights allowed
to take the values \(0\) and \(+\infty\). The characterisation is formulated
through directed compact restrictions \(J\Subset I\). On each compact
interval, endpoint-safe input profiles retain the \(p_i=1\) boundary
contribution as a Stieltjes atom, while the lower-triangle reduction separates
a closed Hardy part from a strict Copson part so that common atoms are counted
exactly once. Upper and mixed regimes follow from exact freezing and compact
linear Hardy estimates. The lower regimes are obtained by power lifting and a
measure-valued Hardy--Copson theorem. After simultaneous regularisation of the
weights, the compact optimal constant is equivalent, with exponent-only
constants, to an operational local characteristic, and the global constant
satisfies
\(C_I\asymp\sup_{J\Subset I}\mathfrak A_J^{\mathrm{op}}\). In the regular
interior range, the resulting conditions recover the classical
\(A_1,\ldots,A_7\) characterisations.
\end{abstract}


\medskip


\tableofcontents
\newpage

\section{Introduction}
\label{sec:introduction}

\subsection{The problem and its scope}

Let \(I\subset\mathbb R\) be an interval with nonempty interior. For a
nonnegative measurable function \(h\) on \(I\), write
\[
 H_Ih(x):=\int_{I\cap(-\infty,x]}h(t)\,\dd t,
 \qquad x\in I,
\]
and define the same-direction bilinear Hardy operator by
\begin{equation}
 \mathcal B_I(f,g)(x):=H_If(x)\starprod H_Ig(x),
 \label{eq:bilinear-hardy-operator}
\end{equation}
where the extended product \(\starprod\) is defined in
\eqref{eq:extended-product}. In particular,
\(0\starprod\infty=\infty\starprod0=0\).
Define \(C_I\) as the infimum of all finite constants \(C\ge0\) for which
\begin{equation}
 \|\mathcal B_I(f,g)\|_{L^q(u;I)}
 \le C
 \|f\|_{L^{p_1}(v_1;I)}
 \|g\|_{L^{p_2}(v_2;I)},
 \qquad f,g\ge0,
 \label{eq:global-bilinear-hardy-inequality}
\end{equation}
holds for every pair with finite input functionals, with
\(\inf\varnothing:=+\infty\). Thus \(C_I=+\infty\) means that no finite
constant is admissible and does not require evaluating a scalar product
\(+\infty\cdot0\). The parameter range throughout is
\begin{equation}
 0<q<\infty,
 \qquad 1\le p_1,p_2\le\infty,
 \label{eq:global-parameter-range}
\end{equation}
for measurable weights \(u,v_1,v_2:I\to[0,\infty]\).

Although \(\mathcal B_I\) is a product of two linear Hardy operators, the
boundedness problem does not in general split into two independent linear
inequalities. The upper region \(p_1,p_2\le q\) is governed by one product
condition, while a mixed region such as \(p_1\le q<p_2\) contains both a
frozen boundary term and a strict Stieltjes tail. When
\(q<\min\{p_1,p_2\}\), the lower triangle divides further according as
\(1/q\le1/p_1+1/p_2\) or \(1/q>1/p_1+1/p_2\). The latter branch introduces
an additional integrability exponent and naturally leads to
measure-valued Hardy--Copson operators.

Our aim is therefore not merely to list the separate regime conditions. We
seek one proof architecture that remains valid at \(p_i=1\) and
\(p_i=\infty\), in the quasi-Banach output range \(0<q<1\), on intervals
with arbitrary endpoint type, and for weights that may vanish or be
infinite on sets of positive measure.

\subsection{Position within the existing theory}

Weighted bilinear Hardy inequalities of the form
\eqref{eq:global-bilinear-hardy-inequality} are already well established.
Aguilar Cañestro, Ortega Salvador and Ramírez Torreblanca obtained a
complete characterisation in the classical interior range
\(1<p_1,p_2,q<\infty\), together with a multidimensional version, by a
discretisation method \cite{AguilarCanestroOrtegaRamirez2012}. K\v{r}epela
later gave a substantially shorter one-dimensional treatment on
\((a,b)\), with finite or infinite endpoints, based on freezing one input
and iterating linear two-weight Hardy inequalities \cite{Krepela2017}. The
freezing identity used below is an endpoint-safe implementation of this
iterative principle, not a priority claim for the principle itself.

Several subsequent works developed the classical theory further. In
particular, Kanjilal, Persson and Shambilova state the full five-region
\(A_1,\ldots,A_7\) theorem for \(0<q<\infty\) and
\(1<p_1,p_2<\infty\), and then develop large parameterised families of
equivalent integral conditions \cite{KanjilalPerssonShambilova2019}.
Stepanov and Shambilova treat the same-direction product
\(R_1f\,R_2g\) for forward Oinarov-kernel Hardy operators. The ordinary
Hardy product is contained in the case of unit kernels, and their lower
region is handled by direct freezing and integral reductions
\cite{StepanovShambilova2019}. Gogatishvili, Jain and Kanjilal obtained
further equivalent criteria in the convex region and applications to a
bilinear geometric-mean inequality \cite{GogatishviliJainKanjilal2022}.
Endpoint exponents and the range \(0<q<1\) are likewise not new in
isolation. Bilgiçli, Mustafayev and Ünver treated multidimensional bilinear
Hardy inequalities for \(0<q\le\infty\) and
\(1\le p_1,p_2\le\infty\), reducing the difficult regions to weighted
iterated Hardy-type inequalities \cite{BilgicliMustafayevUnver2020}. More
recent adjacent developments include bilinear Hardy inequalities on
metric-measure spaces \cite{RuzhanskyShriwastawaVerma2024}, weighted
weak-type bilinear Hardy inequalities
\cite{GarciaGarciaOrtegaSalvador2023}, and the mixed-orientation product
\(Hf\,H^*g\) in the interior exponent range \cite{MohantyJainJain2025}.
These results delimit the prior same-direction, endpoint, geometric and
target-space theory against which the present construction is compared.

Temirkhanova, Zhangabergenova and Oinarov
\cite{TemirkhanovaZhangabergenovaOinarov2026} characterise the product of
an ordinary Hardy operator and a Hardy--Volterra operator with kernels from
the classes \(\mathcal O_1^\pm\) on \((0,\infty)\) for
\(1<p_1,p_2,q<\infty\). For the unit kernel, their criteria reduce to the
classical five-region interior conditions recovered in
Section~\ref{sec:classical-comparison}. Their main theorem scope does not
include the input endpoints or the range \(0<q\le1\). The present
contribution concerns the combined endpoint and atomic formulation
described below, together with directed compact restart and the operational
treatment of extended-valued weights.

The linear and measure-valued inputs also belong to a substantial existing
theory. The two-weight linear Hardy inequality is classically described by
a supremum condition in the convex branch and an integral or
Lebesgue--Stieltjes condition in the non-convex branch. An elementary
half-line proof covering \(1\le p<\infty\), \(0<q<\infty\) is due to
Gogatishvili and Pick \cite{GogatishviliPick2025}. The critical \(p=1\)
case has also been developed in abstract/general-measure and metric-measure
settings \cite{SantacruzHidalgo2024,RuzhanskyShriwastawaTiwari2024}, while
a recent interval treatment is given by Gogatishvili, Pick, Turčinová and
Ünver \cite{GogatishviliPickTurcinovaUnver2025}. General-measure Hardy
inequalities and scales of equivalent conditions were developed, among
others, by Okpoti, Persson and Sinnamon
\cite{OkpotiPerssonSinnamon2007,OkpotiPerssonSinnamon2008}, and Sinnamon's
normal-form theory supplies a broad abstract framework
\cite{SinnamonNormalForm}. Related four-weight Hardy--Copson inequalities
have also been characterised \cite{GogatishviliPickUnver2022}. Our compact
linear and compact measure-valued results are therefore used as
endpoint-safe interfaces for the bilinear proof, not as priority claims for
the classical linear or general Hardy--Copson theories.

Against this background, we do not claim a first boundedness
characterisation, a first exponent range, or a first direct reduction. To the
best of our knowledge, however, we are not aware of an earlier continuous
same-direction strong-type theorem assembling, in one formulation,
arbitrary real-interval endpoint types, exact directed compact restart,
explicit \(p_i=1\) endpoint and atomic bookkeeping, the
closed-Hardy/strict-Copson allocation, all finite \(q>0\) with
\(p_i\in[1,\infty]\), and the present operational treatment of weights taking
values in \([0,\infty]\). The claim made here is restricted to this
endpoint-safe synthesis and the proof architecture supporting it. In the
regular interior range the resulting criteria agree with the established
conditions.

\subsection{Endpoint-safe localisation and atomic geometry}

The principal obstruction to a naive global Stieltjes formulation appears
at a non-attained left endpoint. On a compact restriction
\(J=[c,d]\Subset I\), Section~\ref{sec:linear-hardy} associates each input
with an endpoint-safe increasing profile \(\Phi_{p,v;J}\). At \(p=1\),
the correct right-continuous representative may carry an initial Stieltjes
mass. For every relevant power \(r>0\),
\begin{equation}
 \dd\!\left(\Phi_{1,v;J}^{\,r}\right)(\{c\})
 =\Phi_{1,v;J}(c)^r.
 \label{eq:left-boundary-profile-atom}
\end{equation}
If the ambient interval has no smallest point, a single global profile can
lose this mass. A constant \(p=1\) profile on an open interval has no
interior Stieltjes variation although the associated Hardy operator is
nonzero. Section~\ref{sec:examples} gives the corresponding explicit
failure mechanism.

We therefore restart the operator, profiles and measures on every compact
\(J=[c,d]\Subset I\). The global norm is then recovered exactly from
\begin{equation}
 C_I=\sup_{J\Subset I}C_J.
 \label{eq:directed-norm-identity}
\end{equation}
This identity is valid without a Banach assumption on the output. Moreover,
whenever a local criterion is a sum of structurally different components,
the complete expression is directed as a whole. No replacement by
independently directed summands is built into the definition.

A second endpoint issue arises in the lower-triangle reduction. The
resulting measure operator splits into a closed Hardy part and a strict
Copson part. At a common atom, the diagonal contribution belongs to
\(\int_{[c,x]}h\,\dd\mu\), whereas the Copson component is
\(\int_{(x,d]}h\,\dd\mu\). Consequently the associated Hardy tails are
closed and the Copson tails are strict. This distinction disappears for
nonatomic data but is essential in the endpoint and atomic configurations
covered here.

\subsection{Contributions and main theorem architecture}

The paper has five principal contributions.
\begin{enumerate}
\item It gives a directed compact-restriction formulation on every real
interval, with the \(p=1\) left-boundary atom retained locally and the global
constant recovered by \eqref{eq:directed-norm-identity}.

\item It develops the two endpoint-safe analytic inputs required by the
bilinear argument. These are a compact linear Hardy theorem and a direct compact
measure-valued Hardy--Copson theorem with the closed-Hardy/strict-Copson
atomic allocation made explicit. The detailed finite-chain, reversal,
partition and compact-transfer arguments are separated from the
main proof flow but remain fully available in the technical appendices.

\item It proves the compact bilinear characterisation by one architecture
for the whole range \eqref{eq:global-parameter-range}. The upper and mixed
regions use exact freezing and repeated linear reduction. In the lower
triangle, power lifting first moves the relevant quantity into a genuine
Banach space. Only then is positive duality applied, followed by Tonelli,
a second linear reduction and the measure-valued Hardy--Copson split. Thus
no duality argument is applied directly in \(L^q\) when \(q<1\).

\item It preserves the complete mixed and atomic structure. In particular,
a strict mixed tail alone is not the full criterion. The frozen boundary
contribution is also necessary. The lower-triangle characteristics likewise
retain the closed/strict allocation forced by atoms. Section~\ref{sec:examples}
contains short examples isolating both mechanisms.

\item It treats extended-valued weights operationally by simultaneous
regularisation \(u_m=u\wedge m\), \(v_{i,m}=v_i+1/m\), directing the
complete local characteristic through \(m\). No convergence of raw
extended-profile Stieltjes measures is required. In the regular interior
range, Section~\ref{sec:classical-comparison} identifies the resulting
conditions term by term with the classical \(A_1,\ldots,A_7\) criteria.
\end{enumerate}

These contributions should be read as a unified formulation and proof
architecture rather than as separate priority claims for classical Hardy theory,
freezing, endpoint exponent ranges or the known interior bilinear criteria.
The present paper is restricted to finite output exponent \(q\). The
\(L^\infty\) output problem is outside its scope.

For a compact interval \(J=[c,d]\Subset I\), let \(C_J\) be the optimal
local bilinear constant and let \(\mathfrak A_J^{\mathrm{op}}\) denote the
complete operational characteristic defined in
Section~\ref{sec:local-characteristics}. It incorporates the upper, mixed,
shallow-lower, deep-lower and infinite-input branches without requiring the
regime table to be repeated here. The compact theorem of
Section~\ref{sec:compact-bilinear} is
\begin{equation}
 C_J\asymp_{p_1,p_2,q}\mathfrak A_J^{\mathrm{op}},
 \label{eq:compact-master-introduction}
\end{equation}
with constants independent of \(J\), the weights and the regularisation
index. When \(p_1=p_2=\infty\), the local relation is exact. Combining this
with \eqref{eq:directed-norm-identity} gives the global theorem
\begin{equation}
 C_I
 \asymp_{p_1,p_2,q}
 \sup_{J\Subset I}\mathfrak A_J^{\mathrm{op}},
 \label{eq:global-directed-characterisation}
\end{equation}
again with equality in the doubly-infinite input case.

\subsection{Organisation of the paper}

Section~\ref{sec:weighted-spaces} fixes the extended weighted functionals,
compact restart and simultaneous regularisation. Section~\ref{sec:linear-hardy}
introduces the endpoint-safe profiles and states the compact linear theorem.
Its complete transfer proof is in \CompactLinearTransferProofLocation.
Section~\ref{sec:hardy-copson} gives the compact measure-valued principle,
with the finite-chain, strict-reversal and partition arguments in
\FiniteChainHardyProofLocation, \StrictCopsonReversalProofLocation\ and
\PartitionApproximationProofLocation. Sections~\ref{sec:local-characteristics}--\ref{sec:global-theorem}
define the local characteristics, prove the compact bilinear theorem and
pass to arbitrary intervals. Section~\ref{sec:classical-comparison} verifies
agreement with the classical \(A_1,\ldots,A_7\) conditions,
Section~\ref{sec:examples} records the endpoint and structural examples, and
Section~\ref{sec:discussion} concludes. The full extended-weight limiting
arguments are collected in \ExtendedRegularisationProofLocation. The
locally finite auxiliary extension is recorded in
\DirectedMeasureExtensionLocation.

\section{Weighted spaces and directed compact restrictions}
\label{sec:weighted-spaces}

Throughout, \(I\subset\mathbb R\) is an interval with nonempty interior,
\(u,v_1,v_2:I\to[0,\infty]\) are measurable, and
\[
 0<q<\infty,
 \qquad
 1\le p_1,p_2\le\infty.
\]
All functions are nonnegative and measurable, and operator constants may take
the value \(+\infty\).

\subsection{Extended weighted Lebesgue functionals}
\label{subsec:extended-weighted-spaces}

For \(a,b\in[0,\infty]\), define
\begin{equation}
 a\starprod b
 :=
 \begin{cases}
 0,&a=0\ \text{or}\ b=0,\\[1mm]
 ab,&0<a<\infty\ \text{and}\ 0<b<\infty,\\[1mm]
 \infty,&\text{otherwise},
 \end{cases}
 \label{eq:extended-product}
\end{equation}
so in particular \(0\starprod\infty=\infty\starprod0=0\). For a
measurable set \(E\subset I\), a measurable weight
\(w:E\to[0,\infty]\), and \(0<s<\infty\), set
\begin{equation}
 \|h\|_{L^s(w;E)}
 :=
 \left(
   \int_E h(t)^s\starprod w(t)\,dt
 \right)^{1/s},
 \label{eq:extended-Ls}
\end{equation}
with the usual quasi-norm interpretation when \(0<s<1\). For
\(s=\infty\), define
\begin{equation}
 \|h\|_{L^\infty(w;E)}
 :=
 \operatorname*{ess\,sup}_{t\in E}
 \bigl(h(t)\starprod w(t)\bigr).
 \label{eq:extended-Linfty}
\end{equation}
Thus every finite-norm input vanishes almost everywhere on
\(\{w=\infty\}\), whereas \(\{w=0\}\) carries no weighted input cost. In
particular, if a zero-norm input and a finite-norm partner generate a nonzero
bilinear output, the corresponding optimal constant is \(+\infty\). The
optimal constants below are understood with these zero-norm tests included,
so no quotient by a zero denominator is used.

The positive-floor regularisation needed at an infinite input exponent is
controlled by the following elementary estimate.

\begin{lemma}[Positive-floor continuity]
\label{lem:positive-floor}
Let \(E\subset\mathbb R\) have finite measure, let
\(v:E\to[0,\infty]\) be measurable, and put
\(v_m=v+m^{-1}\), with \(\infty+m^{-1}=\infty\). If \(h\) is bounded and
\(\|h\|_{L^\infty(v;E)}<\infty\), then
\begin{equation}
 0
 \le
 \|h\|_{L^\infty(v_m;E)}
 -
 \|h\|_{L^\infty(v;E)}
 \le
 \frac{\|h\|_{L^\infty(E)}}{m},
 \qquad
 \|h\|_{L^\infty(v_m;E)}\downarrow
 \|h\|_{L^\infty(v;E)}.
 \label{eq:positive-floor-limit}
\end{equation}
\end{lemma}

\begin{proof}
On \(\{v<\infty\}\), one has
\(h\starprod v_m=h\starprod v+h/m\), while finite
\(L^\infty(v;E)\)-norm forces \(h=0\) almost everywhere on
\(\{v=\infty\}\). Taking essential suprema gives the estimate. The bounded
truncation argument used for arbitrary inputs is recorded in
\ExtendedRegularisationProofLocation.
\end{proof}

\subsection{Compact restrictions}
\label{subsec:compact-restrictions}

Write \(J\Subset I\) when \(J=[c,d]\) is a nondegenerate compact interval
contained in \(I\). Define
\begin{equation}
 H_Jh(x)
 :=
 \int_c^x h(t)\,dt,
 \qquad
 \mathcal B_J(f,g)(x)
 :=
 H_Jf(x)\starprod H_Jg(x),
 \qquad x\in J,
 \label{eq:local-operators}
\end{equation}
and define \(C_J\) as the infimum of all finite constants \(C\ge0\) for
which
\begin{equation}
 \|\mathcal B_J(f,g)\|_{L^q(u;J)}
 \le
 C
 \|f\|_{L^{p_1}(v_1;J)}
 \|g\|_{L^{p_2}(v_2;J)},
 \qquad f,g\ge0,
 \label{eq:local-bilinear-inequality}
\end{equation}
holds for every pair with finite input functionals, with
\(\inf\varnothing:=+\infty\).
The local Hardy operator is restarted at the attained left endpoint \(c\), a
point that becomes essential for the endpoint-complete profiles in
Section~\ref{sec:linear-hardy}.

\begin{proposition}[Directed compact exhaustion]
\label{prop:directed-exhaustion}
If \(J\subset K\Subset I\), then
\begin{equation}
 C_J\le C_K\le C_I.
 \label{eq:compact-monotonicity}
\end{equation}
Moreover,
\begin{equation}
 C_I
 =
 \sup_{J\Subset I}C_J.
 \label{eq:directed-exhaustion}
\end{equation}
\end{proposition}

\begin{proof}
Zero extension from \(J\) to \(K\) preserves both input norms and the Hardy
outputs on \(J\), which gives \(C_J\le C_K\le C_I\). For the reverse global
inequality, choose an increasing compact exhaustion
\begin{equation}
 J_n=[c_n,d_n]\Subset I,
 \qquad
 J_n\subset J_{n+1},
 \qquad
 \bigcup_{n\ge1}J_n=I.
 \label{eq:compact-exhaustion}
\end{equation}
For fixed \(f,g\ge0\), the restarted outputs on \(J_n\), extended by zero
outside \(J_n\), increase pointwise to \(H_If\,H_Ig\). Monotone convergence
therefore yields convergence of the output \(L^q(u)\)-functionals, also for
\(0<q<1\). The corresponding restricted input norms increase to their global
values (with essential suprema when \(p_i=\infty\)). Applying
\eqref{eq:local-bilinear-inequality} with
\(S=\sup_{J\Subset I}C_J\) and passing to the limit gives \(C_I\le S\).
\end{proof}

\begin{corollary}
\label{cor:exhaustion-independence}
For every increasing compact exhaustion satisfying
\eqref{eq:compact-exhaustion},
\begin{equation}
 C_I
 =
 \sup_{n\ge1}C_{J_n}
 =
 \lim_{n\to\infty}C_{J_n}.
 \label{eq:exhaustion-sequence}
\end{equation}
Hence the directed norm is independent of the chosen exhaustion.
\end{corollary}

\begin{proof}
The sequence \(C_{J_n}\) is nondecreasing by
\eqref{eq:compact-monotonicity}, and every \(J\Subset I\) is contained in
some sufficiently large \(J_n\). Proposition~\ref{prop:directed-exhaustion}
then gives the claim.
\end{proof}

\subsection{Simultaneous regularisation of extended weights}
\label{subsec:extended-regularisation}

Fix \(J=[c,d]\Subset I\). For \(m\ge1\), define
\begin{equation}
 u_m:=u\wedge m,
 \qquad
 v_{i,m}:=v_i+\frac1m,
 \qquad i=1,2,
 \label{eq:simultaneous-regularisation}
\end{equation}
with \(\infty+m^{-1}=\infty\), and let \(C_{J,m}\) denote the corresponding
local bilinear constant. On the finite part of \(v_i\), the regularised input
weight is bounded below by \(m^{-1}\).
On \(\{v_i=\infty\}\), finite-norm
inputs vanish and reciprocal factors are assigned the value zero. Hence the
reciprocal functions and reciprocal-power densities used later are bounded
on the compact interval at each fixed regularisation level.

\begin{proposition}[Convergence of the regularised constants]
\label{prop:regularised-constants}
For every \(J\Subset I\),
\begin{equation}
 C_{J,m}\uparrow C_J
 \qquad
 \text{as }m\to\infty.
 \label{eq:regularised-constant-limit}
\end{equation}
\end{proposition}

\begin{proof}
Increasing \(m\) increases \(u_m\) and decreases each \(v_{i,m}\), so
\(C_{J,m}\) is nondecreasing and bounded above by \(C_J\). For bounded inputs with positive finite original norms, monotone convergence
on the output and positive-floor convergence on the inputs show that the
regularised test ratios converge to the original one. Arbitrary positive
finite-norm inputs follow by bounded truncation.
If a zero-norm input produces a nonzero output with a finite-norm partner,
the same bounded-truncation argument forces \(C_{J,m}\to\infty\). Otherwise
such directions add no restriction. Thus the limiting regularised constant
is exactly \(C_J\). Full truncation details, including the infinite-input
endpoint, are given in \ExtendedRegularisationProofLocation.
\end{proof}

\begin{corollary}[Operational extended characteristics]
\label{cor:operational-characteristics}
Suppose that for every \(m\ge1\),
\begin{equation}
 c\,\mathfrak A_{J,m}
 \le
 C_{J,m}
 \le
 C\,\mathfrak A_{J,m},
 \label{eq:uniform-regularised-equivalence}
\end{equation}
where \(c,C>0\) depend only on the exponents and are independent of
\(J,m\), and the weights. Define
\begin{equation}
 \mathfrak A_J^{\mathrm{op}}
 :=
 \sup_{m\ge1}\mathfrak A_{J,m}.
 \label{eq:operational-characteristic}
\end{equation}
Then
\begin{equation}
 c\,\mathfrak A_J^{\mathrm{op}}
 \le
 C_J
 \le
 C\,\mathfrak A_J^{\mathrm{op}}.
 \label{eq:extended-characteristic-equivalence}
\end{equation}
\end{corollary}

\begin{proof}
Take the supremum over \(m\) in
\eqref{eq:uniform-regularised-equivalence} and use
Proposition~\ref{prop:regularised-constants}.
\end{proof}

\begin{remark}
\label{rem:no-raw-stieltjes-convergence}
The operational definition \eqref{eq:operational-characteristic} makes no
claim that the raw profiles or their Lebesgue--Stieltjes measures converge
under \eqref{eq:simultaneous-regularisation}. No such convergence is needed.
The regularised compact characterisations are proved with constants uniform
in \(m\), and Corollary~\ref{cor:operational-characteristics} transfers them
to the extended-weight problem. The complete technical formulation is given
in \ExtendedRegularisationProofLocation.
\end{remark}

Proposition~\ref{prop:directed-exhaustion} and
Corollary~\ref{cor:operational-characteristics} are the two limiting
principles used below.
We first remove the regularisation on a fixed compact
interval and then take the directed compact supremum.

\section{Endpoint-safe profiles and the compact linear Hardy theorem}
\label{sec:linear-hardy}

The bilinear reductions below repeatedly use a linear Hardy inequality on a
compact interval.  Fix
\[
 J=[c,d]\Subset I,
 \qquad c<d,
\]
and, for an integrable output weight \(w\ge0\), set
\begin{equation}
 W_J(x):=\int_x^d w(t)\,\dd t,
 \qquad x\in J.
 \label{eq:local-output-tail}
\end{equation}
At a fixed regularisation level an input weight \(v\) is called
\emph{regular on \(J\)} if, for some \(\delta>0\),
\begin{equation}
 v(t)\in[\delta,\infty]
 \qquad\text{for almost every }t\in J.
 \label{eq:regular-input-weight}
\end{equation}
The value \(v=\infty\) is allowed on a forbidden input set.
Finite-norm
inputs vanish there and reciprocal expressions are assigned the value zero.
The regularised compact problems of Section~\ref{subsec:extended-regularisation}
satisfy these hypotheses.  As usual, \(A\asymp_{\alpha}B\) means equivalence
with positive constants depending only on the displayed parameters \(\alpha\).

\subsection{Local integration profiles}
\label{subsec:local-profiles}

Let \(1\le p\le\infty\).  We use
\begin{equation}
 \frac1\infty=0,
 \qquad
 \frac1{v(t)}
 =
 \begin{cases}
 1/v(t),&v(t)<\infty,\\
 0,&v(t)=\infty.
 \end{cases}
 \label{eq:reciprocal-weight-convention}
\end{equation}
For \(1<p<\infty\), with \(p'=p/(p-1)\), define
\begin{equation}
 \Phi_{p,v;J}(x)
 :=
 \left(
   \int_c^x v(t)^{1-p'}\,\dd t
 \right)^{1/p'},
 \qquad c\le x\le d,
 \label{eq:finite-p-profile}
\end{equation}
where \(v^{1-p'}=0\) on \(\{v=\infty\}\).  This profile is finite,
continuous and nondecreasing, and
\begin{equation}
 \Phi_{p,v;J}(c)=0.
 \label{eq:finite-p-profile-left}
\end{equation}
For \(p=\infty\), put
\begin{equation}
 \Phi_{\infty,v;J}(x)
 :=
 \int_c^x\frac{\dd t}{v(t)},
 \qquad c\le x\le d.
 \label{eq:infinite-p-profile}
\end{equation}

The endpoint \(p=1\) requires two one-sided representatives to be
kept distinct.  For \(c<x\le d\), define the exact functional profile
\begin{equation}
 F_{v,J}(x)
 :=
 \operatorname*{ess\,sup}_{c<t<x}\frac1{v(t)}.
 \label{eq:p-one-open-profile}
\end{equation}
Then \(F_{v,J}\) is nondecreasing and left-continuous on \((c,d]\).
For the Stieltjes formulation we use its right-continuous completion
\begin{equation}
 \Phi_{1,v;J}(x)
 :=
 \begin{cases}
 \displaystyle\lim_{y\downarrow c}F_{v,J}(y),&x=c,\\[2mm]
 \displaystyle\lim_{y\downarrow x}F_{v,J}(y),&c<x<d,\\[2mm]
 F_{v,J}(d),&x=d.
 \end{cases}
 \label{eq:p-one-left-value}
\end{equation}
Thus \(F_{v,J}\) gives the exact norm of integration up to the open upper
endpoint, whereas \(\Phi_{1,v;J}\) is the right-continuous generator used
for the endpoint-complete Stieltjes measure.  In particular,
\(\Phi_{1,v;J}(c)\) records the limiting functional norm immediately to the
right of \(c\), not the value of the Hardy operator at the singleton
\(\{c\}\).

\begin{lemma}[Norm of the local integration functional]
\label{lem:local-integration-functional}
Let \(v\) be regular on \(J\).  If \(1<p<\infty\), then for every
\(x\in J\),
\begin{equation}
 \sup_{\|h\|_{L^p(v;J)}\le1}
 \int_c^x h(t)\,\dd t
 =
 \Phi_{p,v;J}(x).
 \label{eq:finite-p-functional-norm}
\end{equation}
If \(p=1\) and \(c<x\le d\), then
\begin{equation}
 \sup_{\|h\|_{L^1(v;J)}\le1}
 \int_c^x h(t)\,\dd t
 =
 F_{v,J}(x).
 \label{eq:p-one-functional-norm}
\end{equation}
\end{lemma}

\begin{proof}
For \(1<p<\infty\), Hölder's inequality gives the upper bound in
\eqref{eq:finite-p-functional-norm}, since \(-p'/p=1-p'\), and the usual
truncated Hölder extremisers give sharpness.  For \(p=1\),
\[
 \int_c^x h(t)\,\dd t
 \le
 F_{v,J}(x)\int_c^x h(t)v(t)\,\dd t.
\]
Conversely, if \(0<\lambda<F_{v,J}(x)\), choose a positive-measure
set \(E\subset(c,x)\) on which \(v^{-1}>\lambda\) and normalise
\(v^{-1}\mathbf1_E\) in \(L^1(v;J)\).  Letting
\(\lambda\uparrow F_{v,J}(x)\) proves the reverse inequality.  Full
functional details, including the infinite-input endpoint, are recorded in
\CompactLinearTransferProofLocation.
\end{proof}

\begin{lemma}[Invariance under the \(p=1\) Stieltjes completion]
\label{lem:p-one-completion-invariance}
Let \(F=F_{v,J}\), \(\Phi=\Phi_{1,v;J}\), and let
\(Z:J\to[0,\infty)\) be continuous. Then
\begin{equation}
 \sup_{c<x\le d}F(x)Z(x)
 =
 \sup_{x\in J}\Phi(x)Z(x).
 \label{eq:p-one-supremum-invariance}
\end{equation}
For every \(r>0\), the left-continuous Lebesgue--Stieltjes measure generated
by \(F^r\), with the initial value zero immediately to the left of \(c\),
is exactly the Borel measure generated by the right-continuous completion
\(\Phi^r\). Equivalently, both measures satisfy
\begin{equation}
 \mu([c,x])=\Phi(x)^r,
 \qquad c\le x\le d.
 \label{eq:p-one-stieltjes-invariance}
\end{equation}
\end{lemma}

\begin{proof}
Since \(F\le\Phi\), one inequality in
\eqref{eq:p-one-supremum-invariance} is immediate. For \(x<d\), choose
\(y_n\downarrow x\). Then \(F(y_n)\to\Phi(x)\) and, by continuity,
\(Z(y_n)\to Z(x)\). The endpoint \(x=c\) is treated in the same way.
At \(x=d\), \(F(d)=\Phi(d)\). Hence the two suprema agree.

For the measure statement, denote by \(\lambda_F\) the left-continuous
Stieltjes measure, so that
\(\lambda_F([c,b))=F(b)^r\) for \(c<b\le d\). If \(x<d\), continuity
from above and \(b\downarrow x\) give
\[
 \lambda_F([c,x])
 =\lim_{b\downarrow x}F(b)^r
 =\Phi(x)^r.
\]
For \(x=d\) the same identity follows from \(\Phi(d)=F(d)\). These closed
initial intervals determine the Borel measure on \(J\), proving
\eqref{eq:p-one-stieltjes-invariance}.
\end{proof}

For the two inputs of the bilinear problem we write
\begin{equation}
 \Phi_{i,J}:=\Phi_{p_i,v_i;J},
 \qquad i=1,2.
 \label{eq:bilinear-local-profiles}
\end{equation}

\subsection{Endpoint-safe Lebesgue--Stieltjes measures}
\label{subsec:endpoint-stieltjes}

Let \(G:[c,d]\to[0,\infty)\) be nondecreasing and right-continuous.  We
associate with \(G\) the Lebesgue--Stieltjes measure obtained by extending
\(G\) by zero immediately to the left of \(c\).  Thus
\begin{equation}
 dG([c,x])=G(x),
 \qquad c\le x\le d,
 \label{eq:stieltjes-cumulative}
\end{equation}
and, in particular,
\begin{equation}
 dG(\{c\})=G(c).
 \label{eq:stieltjes-left-atom}
\end{equation}

For \(1\le p<\infty\) and \(0<q<p\), define \(r\in(0,\infty)\) by
\begin{equation}
 \frac1r=\frac1q-\frac1p.
 \label{eq:linear-r-exponent}
\end{equation}
and set
\begin{equation}
 \mu_{p,v;J}
 :=
 d\!\left(\Phi_{p,v;J}^{\,r}\right).
 \label{eq:profile-stieltjes-measure}
\end{equation}
Then
\begin{equation}
 \mu_{p,v;J}([c,x])=\Phi_{p,v;J}(x)^r,
 \qquad c\le x\le d,
 \label{eq:profile-measure-cumulative}
\end{equation}
and
\begin{equation}
 \mu_{p,v;J}(\{c\})=\Phi_{p,v;J}(c)^r.
 \label{eq:profile-measure-left-atom}
\end{equation}
For \(1<p<\infty\) this atom vanishes, whereas for \(p=1\) it may be
positive and must be retained.

\begin{remark}[Necessity of the boundary atom]
\label{rem:necessity-boundary-atom}
If \(v\equiv1\) and \(p=1\), then \(\Phi_{1,v;J}\equiv1\) and hence
\(d(\Phi_{1,v;J}^{\,r})=\delta_c\). Variation on \((c,d]\) alone would
therefore lose the entire contribution.  The corresponding bilinear endpoint
mechanism is exhibited explicitly in Subsection~\ref{subsec:nonattained-endpoint-example}.
\end{remark}

\subsection{The compact endpoint-safe linear theorem}
\label{subsec:compact-linear-hardy}

For \(1\le p<\infty\), let \(L_{p,q}(v,w;J)\) denote the optimal constant
in
\begin{equation}
 \|H_Jh\|_{L^q(w;J)}
 \le
 L_{p,q}(v,w;J)\|h\|_{L^p(v;J)},
 \qquad h\ge0.
 \label{eq:compact-linear-hardy-inequality}
\end{equation}

\begin{theorem}[Compact endpoint-safe linear Hardy theorem]
\label{thm:compact-linear-hardy}
Let \(J=[c,d]\), let \(1\le p<\infty\) and \(0<q<\infty\), and let
\(v\) be regular on \(J\) in the sense above and let \(w\ge0\) be
integrable on \(J\).  If \(p\le q\), then
\begin{equation}
 L_{p,q}(v,w;J)
 \asymp_{p,q}
 \sup_{x\in J}\Phi_{p,v;J}(x)W_J(x)^{1/q}.
 \label{eq:compact-linear-convex}
\end{equation}
If \(0<q<p<\infty\) and \(r\) is defined by
\eqref{eq:linear-r-exponent}, then
\begin{equation}
 L_{p,q}(v,w;J)
 \asymp_{p,q}
 \left(
   \int_J W_J(x)^{r/q}
   \,\dd\!\left(\Phi_{p,v;J}(x)^r\right)
 \right)^{1/r}.
 \label{eq:compact-linear-nonconvex}
\end{equation}
Both equivalences hold in the extended sense.  In particular, for
\(p=1>q\), the integral in \eqref{eq:compact-linear-nonconvex} contains
\begin{equation}
 W_J(c)^{r/q}\Phi_{1,v;J}(c)^r.
 \label{eq:p-one-boundary-contribution}
\end{equation}
The comparison constants depend only on \(p\) and \(q\).
\end{theorem}

\begin{proof}
The underlying half-line characterisation, including the exponent-only
necessity and sufficiency estimates used here, is given by
\cite[Theorem~2.1 and its proof]{GogatishviliPick2025}. For a recent interval treatment
including the critical case \(p=1\), see
\cite{GogatishviliPickTurcinovaUnver2025}. The complete compact transport is
proved in \CompactLinearTransferProofLocation.  We record the interface needed
later.  Put \(\ell=d-c\), translate \(J\) onto \([0,\ell]\), extend the
output weight by zero on \([\ell,\infty)\), extend each input function by
zero there, and continue the input weight by any positive regular value.  Deleting
the part of any half-line input beyond \(\ell\) cannot
increase its input norm and does not change the output on the support of the
translated output weight.  Hence the compact optimal constant equals the
corresponding half-line constant.

On \([0,\ell]\), translation carries the output tail to
\(W_J(c+s)\) and, for \(1<p<\infty\), the input profile to
\(\Phi_{p,v;J}(c+s)\). For \(p=1\), the half-line theorem uses the
left-continuous exact profile \(F_{v,J}\). Lemma~\ref{lem:p-one-completion-invariance}
shows that replacing it by the right-continuous Stieltjes completion
\(\Phi_{1,v;J}\) changes neither the convex supremum nor the non-convex
Lebesgue--Stieltjes measure. In particular the initial mass becomes the atom
\(\Phi_{1,v;J}(c)^r\) at \(s=0\). The finite-ceiling argument in
\CompactLinearTransferProofLocation{} then extends the transported theorem
from finite input weights to regular weights that may equal \(+\infty\) on
forbidden sets. Since all estimates in the half-line theorem and in that
bridge are uniform, the comparison constants depend only on \((p,q)\), not
on \(J\), the weights, the regularisation level, or the ceiling parameter.
Thus the theorem is a compact-transfer statement, not a new
characterisation of the classical linear inequality.
\end{proof}

\subsection{The infinite-input endpoint}
\label{subsec:infinite-input-functional}

The case \(p=\infty\) is exact and does not require the preceding two-weight
theorem.

\begin{lemma}[Exact \(L^\infty(v)\) integration functional]
\label{lem:linfty-integration-functional}
Let \(v:J\to[\delta,\infty]\) be measurable for some \(\delta>0\).  Then,
for every \(x\in J\),
\begin{equation}
 \sup_{\|h\|_{L^\infty(v;J)}\le1}H_Jh(x)
 =
 \Phi_{\infty,v;J}(x).
 \label{eq:linfty-functional-identity}
\end{equation}
Moreover,
\begin{equation}
 \sigma_v(t):=\frac1{v(t)}
 =
 \begin{cases}
 1/v(t),&v(t)<\infty,\\
 0,&v(t)=\infty,
 \end{cases}
 \label{eq:linfty-extremiser}
\end{equation}
satisfies
\begin{equation}
 \|\sigma_v\|_{L^\infty(v;J)}\le1,
 \label{eq:linfty-extremiser-norm}
\end{equation}
and
\begin{equation}
 H_J\sigma_v(x)=\Phi_{\infty,v;J}(x),
 \qquad x\in J.
 \label{eq:linfty-extremiser-output}
\end{equation}
\end{lemma}

\begin{proof}
If \(N=\|h\|_{L^\infty(v;J)}<\infty\), then
\(h\le N/v\) almost everywhere, with both sides zero on
\(\{v=\infty\}\).  Hence
\(H_Jh(x)\le N\Phi_{\infty,v;J}(x)\).  The function \(\sigma_v\) has
weighted \(L^\infty\)-norm at most one under the \(\starprod\) convention and
attains equality, proving the claim.
\end{proof}

\begin{corollary}[Infinite-input freezing]
\label{cor:infinite-input-freezing}
Let \(v_1:J\to[\delta,\infty]\), let \(v_2\) and \(u\) be regularised
weights, and let \(0<q<\infty\) and \(1\le p_2<\infty\).  The optimal
constant in
\begin{equation}
 \|H_Jf\,H_Jg\|_{L^q(u;J)}
 \le
 C\|f\|_{L^\infty(v_1;J)}\|g\|_{L^{p_2}(v_2;J)}
 \label{eq:one-infinite-bilinear-inequality}
\end{equation}
is exactly
\begin{equation}
 C
 =
 \left\|
 H_J:
 L^{p_2}(v_2;J)
 \longrightarrow
 L^q\!\left(\Phi_{\infty,v_1;J}^{\,q}u;J\right)
 \right\|.
 \label{eq:one-infinite-linear-reduction}
\end{equation}
If both input exponents are infinite, then
\begin{equation}
 C
 =
 \left(
   \int_J
   \left[
     \Phi_{\infty,v_1;J}(x)\Phi_{\infty,v_2;J}(x)
   \right]^q
   u(x)\,\dd x
 \right)^{1/q}.
 \label{eq:two-infinite-exact-formula}
\end{equation}
\end{corollary}

\begin{proof}
Lemma~\ref{lem:linfty-integration-functional} gives the required pointwise
upper bound for each infinite input.  Testing with \(\sigma_{v_1}\) gives
equality in \eqref{eq:one-infinite-linear-reduction}. When both inputs are
infinite, testing simultaneously with \(\sigma_{v_1}\) and \(\sigma_{v_2}\)
gives \eqref{eq:two-infinite-exact-formula}.
\end{proof}

\begin{remark}
\label{rem:extended-infinite-input}
For arbitrary extended input weights, apply these identities at the
regularised level \(v_{i,m}=v_i+1/m\) and then use the operational supremum
of Corollary~\ref{cor:operational-characteristics}. The full limiting argument
is contained in \ExtendedRegularisationProofLocation.
\end{remark}

\section{A direct measure-valued Hardy--Copson principle}
\label{sec:hardy-copson}

The lower-triangle reduction of the bilinear problem produces Hardy and
Copson operators between spaces defined by Borel measures.  The relevant
criteria belong to the established general-measure Hardy theory. See, for
example, \cite{OkpotiPerssonSinnamon2007,OkpotiPerssonSinnamon2008,
SinnamonNormalForm}. Related four-weight inequalities involving Hardy and
Copson operators have also been characterised in
\cite{GogatishviliPickUnver2022}. For the bilinear argument we need a compact
version in which three endpoint features are explicit. The Hardy kernel is
closed, the Copson kernel is strict, and every left-boundary Stieltjes atom is
retained. No priority claim for the general Hardy--Copson problem is intended.
We derive that version directly from finite chains, exact Copson reversal and
atom-preserving interval partitions. Arbitrary Copson weights are then
recovered by monotone truncation.  The complete finite-chain and partition
estimates are recorded in Appendices~\ref{app:finite-chain-hardy}--\ref{app:partition-approximation}.

\subsection{Compact operators and endpoint conventions}
\label{subsec:measure-operators}

Let
\[
 J=[c,d],\qquad c<d,
\]
let \(\mu\) and \(\nu\) be finite nonnegative Borel measures on \(J\), and
fix
\begin{equation}
 1<P,Q<\infty,
 \qquad
 P'=\frac{P}{P-1}.
 \label{eq:measure-exponents}
\end{equation}
For \(h\ge0\), define the closed Hardy operator
\begin{equation}
 H_\mu h(x)
 :=
 \int_{[c,x]}h(t)\,\dd\mu(t),
 \qquad x\in J,
 \label{eq:measure-hardy-operator}
\end{equation}
and, for a Borel weight \(\omega:J\to[0,\infty]\), the strict weighted
Copson operator
\begin{equation}
 C_{\mu,\omega}h(x)
 :=
 \int_{(x,d]}\omega(t)h(t)\,\dd\mu(t),
 \label{eq:weighted-copson-operator}
\end{equation}
where an unbounded or infinite-valued \(\omega\) is interpreted by the
monotone truncation described below.

Put
\begin{equation}
 \begin{aligned}
 M(x)&:=\mu([c,x]), & N(x)&:=\nu([x,d]),\\
 K(x)&:=\nu([c,x]), &
 B_\omega(x)&:=\int_{(x,d]}\omega(t)^{P'}\,\dd\mu(t).
 \end{aligned}
 \label{eq:measure-cumulative-functions}
\end{equation}
The closed interval in \(M,N\) and the strict interval in \(B_\omega\) are
intentional.
A common input/output atom at \(x\) contributes to
\(H_\mu h(x)\) but not to \(C_{\mu,\omega}h(x)\).

If \(Q<P\), define \(R\in(0,\infty)\) by
\begin{equation}
 \frac1R=\frac1Q-\frac1P,
 \qquad
 \frac RQ=1+\frac RP>1.
 \label{eq:measure-r-exponent}
\end{equation}
Whenever \(F:[c,d]\to[0,\infty)\) is nondecreasing and right-continuous,
\(\dd(F^\gamma)\) denotes the Lebesgue--Stieltjes measure obtained after
extending \(F^\gamma\) by zero immediately to the left of \(c\).  Hence
\begin{equation}
 \dd(F^\gamma)(\{c\})=F(c)^\gamma.
 \label{eq:measure-stieltjes-boundary}
\end{equation}
In particular, the non-convex criteria below include the left-boundary
masses
\[
 M(c)^{R/P'}\delta_c,
 \qquad
 K(c)^{R/Q}\delta_c.
\]

Define the compact characteristics
\begin{equation}
 \mathfrak H_J
 :=
 \begin{cases}
 \displaystyle
 \sup_{x\in J}M(x)^{1/P'}N(x)^{1/Q}, & P\le Q,\\[3mm]
 \displaystyle
 \left[
   \int_J N(x)^{R/Q}\,
   \dd\!\left(M(x)^{R/P'}\right)
 \right]^{1/R}, & Q<P,
 \end{cases}
 \label{eq:compact-hardy-characteristic}
\end{equation}
and
\begin{equation}
 \mathfrak C_{\omega,J}
 :=
 \begin{cases}
 \displaystyle
 \sup_{x\in J}K(x)^{1/Q}B_\omega(x)^{1/P'}, & P\le Q,\\[3mm]
 \displaystyle
 \left[
   \int_J B_\omega(x)^{R/P'}\,
   \dd\!\left(K(x)^{R/Q}\right)
 \right]^{1/R}, & Q<P.
 \end{cases}
 \label{eq:compact-copson-characteristic}
\end{equation}

\subsection{Finite-chain inputs and the partition passage}
\label{subsec:finite-chain-inputs}

For an ordered chain \(x_1<\cdots<x_n\), let
\[
 \mu=\sum_{i=1}^n\mu_i\delta_{x_i},
 \qquad
 \nu=\sum_{k=1}^n\nu_k\delta_{x_k},
\]
and set
\[
 M_k:=\sum_{i=1}^k\mu_i,
 \qquad
 N_k:=\sum_{j=k}^n\nu_j.
\]
The closed finite-chain Hardy operator is
\[
 (H_na)_k:=\sum_{i=1}^k a_i\mu_i.
\]

\begin{theorem}[Complete finite-chain Hardy criterion]
\label{thm:finite-chain-hardy}
Let \(\cH_n\) be the optimal constant in
\[
 \left[
   \sum_{k=1}^n(H_na)_k^Q\nu_k
 \right]^{1/Q}
 \le
 \cH_n
 \left(\sum_{i=1}^n a_i^P\mu_i\right)^{1/P},
 \qquad a_i\ge0.
\]
Then
\begin{equation}
 \cH_n
 \asymp_{P,Q}
 \begin{cases}
 \displaystyle
 \max_{1\le k\le n}M_k^{1/P'}N_k^{1/Q}, & P\le Q,\\[3mm]
 \displaystyle
 \left[
   \sum_{k=1}^n N_k^{R/Q}
   \left(M_k^{R/P'}-M_{k-1}^{R/P'}\right)
 \right]^{1/R}, & Q<P,
 \end{cases}
 \label{eq:finite-chain-hardy-criterion}
\end{equation}
where \(M_0=0\), and the comparison constants are independent of \(n\)
and of the masses.
\end{theorem}

\begin{proof}[Proof architecture]
The complete proof is given in \FiniteChainHardyProofLocation.  Necessity
in the convex range follows from initial-segment tests.  Sufficiency is a
Hölder estimate combined with the elementary power-increment comparison on
the cumulative masses.  When \(Q<P\), Abel summation converts the displayed
Stieltjes-type sum to the equivalent weighted form used for the sharp test.
Hölder with exponents \(R/Q\) and \(P/Q\) then gives the sufficiency estimate.
All constants are uniform in the chain length.
\end{proof}

For the strict Copson branch let \(\omega_i\in[0,\infty)\), put
\[
 K_k:=\sum_{j=1}^k\nu_j,
 \qquad
 B_k:=\sum_{i=k+1}^n\omega_i^{P'}\mu_i,
 \qquad B_n:=0,
\]
and define
\[
 (C_{\omega,n}a)_k
 :=
 \sum_{i=k+1}^n\omega_i a_i\mu_i.
\]

\begin{theorem}[Strict finite-chain Copson criterion]
\label{thm:finite-chain-copson}
Let \(\cC_n\) be the optimal constant in
\[
 \left[
   \sum_{k=1}^n(C_{\omega,n}a)_k^Q\nu_k
 \right]^{1/Q}
 \le
 \cC_n
 \left(\sum_{i=1}^n a_i^P\mu_i\right)^{1/P}.
\]
Then
\begin{equation}
 \cC_n
 \asymp_{P,Q}
 \begin{cases}
 \displaystyle
 \max_{1\le k\le n}K_k^{1/Q}B_k^{1/P'}, & P\le Q,\\[3mm]
 \displaystyle
 \left[
   \sum_{k=1}^n B_k^{R/P'}
   \left(K_k^{R/Q}-K_{k-1}^{R/Q}\right)
 \right]^{1/R}, & Q<P,
 \end{cases}
 \label{eq:finite-chain-copson-criterion}
\end{equation}
where \(K_0=0\).  The diagonal input atom is excluded from \(B_k\).
\end{theorem}

\begin{proof}[Proof architecture]
The exact reversal is proved in \StrictCopsonReversalProofLocation.  On the
active coordinates the substitution
\[
 \lambda_i=\omega_i^{P'}\mu_i,
 \qquad
 b_i=a_i\omega_i^{-1/(P-1)}
\]
is isometric and removes the weight.  Reversing the chain and deleting the
unused endpoint turns \(i>k\) exactly into a closed Hardy relation.  The
reversed cumulative masses are \(B_k\) and \(K_k\).
In the non-convex range
a finite Abel summation yields \eqref{eq:finite-chain-copson-criterion}.
This reversal is also what fixes the strict exclusion of a same-point input
atom.
\end{proof}

The finite-chain criteria pass to finite Borel measures through refining
ordered interval partitions that isolate every atom and whose nonsingleton
cells have uniformly vanishing input/output mass.

\begin{theorem}[Atom-preserving partition approximation]
\label{thm:partition-approximation}
There are refining ordered interval partitions of \(J\) for which the
associated closed atomic Hardy operators and strict atomic Copson operators
converge in operator norm to their continuous counterparts.  Along the same
partitions, the finite-chain convex characteristics and the non-convex
Stieltjes sums converge to the corresponding continuous expressions in
\eqref{eq:compact-hardy-characteristic} and
\eqref{eq:compact-copson-characteristic}.  Every fixed atom is eventually a
singleton cell, so the left-boundary mass and the closed-Hardy/strict-Copson
diagonal allocation are preserved exactly.
\end{theorem}

The full construction, including the cell-error estimates, exact block-chain
norm identities and non-convex Stieltjes limits, is given in
\PartitionApproximationProofLocation.

For an arbitrary Borel weight \(\omega:J\to[0,\infty]\), put
\begin{equation}
 \omega_m:=\omega\wedge m.
 \label{eq:copson-weight-truncation}
\end{equation}
Then, for every \(h\ge0\),
\begin{equation}
 C_{\mu,\omega_m}h(x)\uparrow C_{\mu,\omega}h(x),
 \qquad
 B_{\omega_m}(x)\uparrow B_\omega(x),
 \label{eq:copson-truncation-limits}
\end{equation}
and monotone convergence gives
\begin{equation}
 \|C_{\mu,\omega_m}:L^P(\mu)\to L^Q(\nu)\|
 \uparrow
 \|C_{\mu,\omega}:L^P(\mu)\to L^Q(\nu)\|.
 \label{eq:copson-norm-truncation}
\end{equation}
The convex supremum and the non-convex Stieltjes characteristic increase to
their counterparts for \(\omega\).  Thus no separate finiteness assumption
on the strict Copson weight is required.

\subsection{The compact measure-valued theorem}
\label{subsec:compact-measure-theorem}

\begin{theorem}[Compact direct measure-valued Hardy--Copson theorem]
\label{thm:compact-measure-hardy-copson}
Let \(\mu,\nu\) be finite nonnegative Borel measures on \(J=[c,d]\), let
\(1<P,Q<\infty\), and let \(\omega:J\to[0,\infty]\) be Borel measurable.
Then
\begin{equation}
 \|H_\mu:L^P(\mu)\to L^Q(\nu)\|
 \asymp_{P,Q}
 \mathfrak H_J,
 \label{eq:compact-measure-hardy-result}
\end{equation}
and
\begin{equation}
 \|C_{\mu,\omega}:L^P(\mu)\to L^Q(\nu)\|
 \asymp_{P,Q}
 \mathfrak C_{\omega,J}.
 \label{eq:compact-measure-copson-result}
\end{equation}
Both equivalences hold in the extended sense, with constants depending only
on \(P,Q\).  The Stieltjes measures include the left-boundary masses from
\eqref{eq:measure-stieltjes-boundary}. The Hardy input tail is closed and the
Copson input tail is strict.
\end{theorem}

\begin{proof}
Apply Theorem~\ref{thm:finite-chain-hardy} and the strict Copson criterion
Theorem~\ref{thm:finite-chain-copson} on the atom-preserving partition
chains, then pass to the finite-measure limits described above.  For bounded
\(\omega\), the continuous substitution
\(\dd\lambda=\omega^{P'}\dd\mu\) reduces the weighted Copson problem
isometrically to the unweighted strict one.  Finally,
\eqref{eq:copson-truncation-limits}--\eqref{eq:copson-norm-truncation}
remove the boundedness assumption.  The complete finite-chain, reversal and
partition arguments are given in \FiniteChainHardyProofLocation,
\StrictCopsonReversalProofLocation, and \PartitionApproximationProofLocation.
\end{proof}

\begin{remark}[The strict differential]
\label{rem:strict-copson-differential}
In the non-convex Copson condition the Stieltjes differential is
\(\dd(K^{R/Q})\), generated by the output head
\(K(x)=\nu([c,x])\). It is not generated by the strict input tail
\(B_\omega\).  This orientation is forced both by exact chain reversal and
by the atom-preserving partition limit.
\end{remark}

Only Theorem~\ref{thm:compact-measure-hardy-copson} is required in the local
bilinear proof.  For completeness, the corresponding locally finite
arbitrary-interval extension, obtained by directed compact exhaustion, is
recorded in \DirectedMeasureExtensionLocation.  This separation avoids
introducing a global Stieltjes derivative at a non-attained endpoint into the
main proof.

\section{Local bilinear characteristics}
\label{sec:local-characteristics}

Fix
\[
 J=[c,d]\Subset I
\]
and one regularisation level, writing simply
\(u=u_m\) and \(v_i=v_{i,m}\), \(i=1,2\). Thus \(u\) is finite and
integrable on \(J\), while each \(v_i\) is bounded below by a positive
constant on its finite part and may equal \(+\infty\) on a forbidden input
set. Let
\begin{equation}
 \Phi_{i,J}:=\Phi_{p_i,v_i;J},
 \qquad i=1,2,
 \label{eq:local-input-profiles}
\end{equation}
and
\begin{equation}
 U_J(x):=\int_x^d u(t)\,\dd t,
 \qquad x\in J.
 \label{eq:local-output-tail-bilinear}
\end{equation}
All characteristics below may take the value \(+\infty\). When
\(p_i=1\), \(\Phi_{i,J}\) denotes the right-continuous Stieltjes completion
from Section~\ref{subsec:local-profiles}. The exact truncated integration
functional is \(F_{v_i,J}\). Lemma~\ref{lem:p-one-completion-invariance}
shows that this completion leaves the convex supremum characteristic and the
non-convex profile measure unchanged, so the formulas below remain valid
without identifying \(F_{v_i,J}(x)\) pointwise with \(\Phi_{i,J}(x)\).

Whenever \(q<p_i<\infty\), define
\begin{equation}
 \frac1{r_i}:=\frac1q-\frac1{p_i},
 \qquad
 \mu_{i,J}:=\dd\!\left(\Phi_{i,J}^{\,r_i}\right).
 \label{eq:ri-definition}
\end{equation}
The generator is extended by zero immediately to the left of \(c\), so that
\begin{equation}
 \mu_{i,J}([c,x])=\Phi_{i,J}(x)^{r_i},
 \qquad
 \mu_{i,J}(\{c\})=\Phi_{i,J}(c)^{r_i}.
 \label{eq:profile-measure-cumulative-section5}
\end{equation}
In particular, the left-boundary atom is retained and may be nonzero when
\(p_i=1\).

\subsection{The upper characteristic}
\label{subsec:upper-characteristic}

For \(p_1,p_2<\infty\) with \(p_1,p_2\le q\), set
\begin{equation}
 \mathfrak B_{1,J}
 :=
 \sup_{x\in J}
 U_J(x)^{1/q}\Phi_{1,J}(x)\Phi_{2,J}(x).
 \label{eq:B1-characteristic}
\end{equation}
This is the complete upper-region characteristic and also the frozen boundary
component in the mixed regimes.

\subsection{The mixed characteristics}
\label{subsec:mixed-characteristics}

Assume first that
\begin{equation}
 p_1\le q<p_2<\infty.
 \label{eq:first-mixed-region}
\end{equation}
The complete frozen characteristic and its strict-tail component are
\begin{equation}
 \widehat{\mathfrak B}_{2,J}
 :=
 \sup_{x\in J}
 \Phi_{1,J}(x)
 \left[
   \int_J U_J\!\left(\max\{x,t\}\right)^{r_2/q}
   \,\dd\mu_{2,J}(t)
 \right]^{1/r_2},
 \label{eq:B2-complete}
\end{equation}
and
\begin{equation}
 \mathfrak B_{2,J}
 :=
 \sup_{x\in J}
 \Phi_{1,J}(x)
 \left[
   \int_{(x,d]} U_J(t)^{r_2/q}
   \,\dd\mu_{2,J}(t)
 \right]^{1/r_2}.
 \label{eq:B2-strict}
\end{equation}
In the symmetric regime
\begin{equation}
 p_2\le q<p_1<\infty,
 \label{eq:second-mixed-region}
\end{equation}
put
\begin{equation}
 \widehat{\mathfrak B}_{3,J}
 :=
 \sup_{x\in J}
 \Phi_{2,J}(x)
 \left[
   \int_J U_J\!\left(\max\{x,t\}\right)^{r_1/q}
   \,\dd\mu_{1,J}(t)
 \right]^{1/r_1},
 \label{eq:B3-complete}
\end{equation}
and
\begin{equation}
 \mathfrak B_{3,J}
 :=
 \sup_{x\in J}
 \Phi_{2,J}(x)
 \left[
   \int_{(x,d]} U_J(t)^{r_1/q}
   \,\dd\mu_{1,J}(t)
 \right]^{1/r_1}.
 \label{eq:B3-strict}
\end{equation}

\begin{lemma}[Exact frozen-tail decomposition]
\label{lem:frozen-tail-decomposition}
Under \eqref{eq:first-mixed-region}, for every \(x\in J\),
\begin{equation}
 \begin{aligned}
 &\int_J
 U_J\!\left(\max\{x,t\}\right)^{r_2/q}\,\dd\mu_{2,J}(t)
 \\
 &\qquad=
 U_J(x)^{r_2/q}\Phi_{2,J}(x)^{r_2}
 +
 \int_{(x,d]}U_J(t)^{r_2/q}\,\dd\mu_{2,J}(t).
 \end{aligned}
 \label{eq:exact-frozen-tail-decomposition}
\end{equation}
Consequently,
\begin{equation}
 \max\{\mathfrak B_{1,J},\mathfrak B_{2,J}\}
 \le \widehat{\mathfrak B}_{2,J}
 \le \mathfrak B_{1,J}+\mathfrak B_{2,J}
 \le 2\widehat{\mathfrak B}_{2,J},
 \label{eq:mixed-equivalence}
\end{equation}
and, after interchanging the indices,
\begin{equation}
 \max\{\mathfrak B_{1,J},\mathfrak B_{3,J}\}
 \le \widehat{\mathfrak B}_{3,J}
 \le \mathfrak B_{1,J}+\mathfrak B_{3,J}
 \le 2\widehat{\mathfrak B}_{3,J}.
 \label{eq:symmetric-mixed-equivalence}
\end{equation}
\end{lemma}

\begin{proof}
Split the integral at \([c,x]\cup(x,d]\). On the first part the kernel is
\(U_J(x)\), while on the second it is \(U_J(t)\). The identity follows from
\(\mu_{2,J}([c,x])=\Phi_{2,J}(x)^{r_2}\). The inequalities are then the
elementary comparison of
\((a^{r_2}+b^{r_2})^{1/r_2}\) with \(\max\{a,b\}\) and \(a+b\), followed by
taking suprema. The symmetric statement is identical.
\end{proof}

Thus the strict tail alone is not the mixed condition. Indeed,
\begin{equation}
 \widehat{\mathfrak B}_{2,J}
 \asymp
 \mathfrak B_{1,J}+\mathfrak B_{2,J},
 \qquad
 \widehat{\mathfrak B}_{3,J}
 \asymp
 \mathfrak B_{1,J}+\mathfrak B_{3,J}.
 \label{eq:complete-mixed-condition-summary}
\end{equation}
The separated-support example in Section~\ref{sec:examples} shows that the
frozen term cannot in general be discarded.

\subsection{The lower-triangle characteristics}
\label{subsec:lower-characteristics}

Assume
\begin{equation}
 q<\min\{p_1,p_2\}<\infty.
 \label{eq:lower-triangle-region}
\end{equation}
Then both \(r_1\) and \(r_2\) are defined by \eqref{eq:ri-definition}. Set
\begin{equation}
 \mathfrak B_{4,J}
 :=
 \sup_{x\in J}
 \Phi_{1,J}(x)
 \left[
   \int_{[x,d]} U_J(t)^{r_2/q}\,\dd\mu_{2,J}(t)
 \right]^{1/r_2},
 \label{eq:B4-characteristic}
\end{equation}
and
\begin{equation}
 \mathfrak B_{5,J}
 :=
 \sup_{x\in J}
 \Phi_{2,J}(x)
 \left[
   \int_{(x,d]} U_J(t)^{r_1/q}\,\dd\mu_{1,J}(t)
 \right]^{1/r_1}.
 \label{eq:B5-characteristic}
\end{equation}
The first tail is closed because it comes from the Hardy component and retains
the diagonal atom. The second is strict because it comes from the Copson
component and excludes that atom. In the shallow lower triangle,
\begin{equation}
 \frac1q\le\frac1{p_1}+\frac1{p_2},
 \label{eq:shallow-lower-region}
\end{equation}
the complete quantity is
\begin{equation}
 \mathfrak A_{45,J}:=\mathfrak B_{4,J}+\mathfrak B_{5,J}.
 \label{eq:A45-characteristic}
\end{equation}

\subsection{The deep lower-triangle characteristics}
\label{subsec:deep-lower-characteristics}

If instead
\begin{equation}
 \frac1q>\frac1{p_1}+\frac1{p_2},
 \label{eq:deep-lower-region}
\end{equation}
define
\begin{equation}
 \frac1\kappa
 :=
 \frac1q-\frac1{p_1}-\frac1{p_2}.
 \label{eq:kappa-definition}
\end{equation}
The deep characteristics are
\begin{equation}
 \mathfrak B_{6,J}
 :=
 \left[
   \int_J
   \left(
     \int_{[x,d]} U_J(t)^{r_2/q}\,\dd\mu_{2,J}(t)
   \right)^{\kappa/r_2}
   \dd\!\left(\Phi_{1,J}(x)^\kappa\right)
 \right]^{1/\kappa},
 \label{eq:B6-characteristic}
\end{equation}
and
\begin{equation}
 \mathfrak B_{7,J}
 :=
 \left[
   \int_J
   \left(
     \int_{(x,d]} U_J(t)^{r_1/q}\,\dd\mu_{1,J}(t)
   \right)^{\kappa/r_1}
   \dd\!\left(\Phi_{2,J}(x)^\kappa\right)
 \right]^{1/\kappa}.
 \label{eq:B7-characteristic}
\end{equation}
Here \(\Phi_{i,J}^\kappa\) is again extended by zero to the left of \(c\), so
its left-boundary atom is included. The complete deep quantity is
\begin{equation}
 \mathfrak A_{67,J}:=\mathfrak B_{6,J}+\mathfrak B_{7,J}.
 \label{eq:A67-characteristic}
\end{equation}
The ordered pairs
\((\mathfrak B_{4,J},\mathfrak B_{5,J})\) and
\((\mathfrak B_{6,J},\mathfrak B_{7,J})\) therefore retain the same
closed-Hardy/strict-Copson allocation. Making both tails closed would add a
false diagonal Copson contribution, while making both strict would lose the
diagonal Hardy contribution.

\subsection{Infinite input exponents}
\label{subsec:one-infinite-characteristics}

If \(p_1=\infty\) and \(1\le p_2<\infty\), the exact freezing reduction of
Corollary~\ref{cor:infinite-input-freezing} gives the effective tail
\begin{equation}
 W_{1,J}(x)
 :=
 \int_x^d \Phi_{1,J}(t)^q u(t)\,\dd t,
 \label{eq:W1-infinite-tail}
\end{equation}
and the linear Hardy characteristic
\begin{equation}
 \mathfrak B_{\infty,2;J}
 :=
 \begin{cases}
 \displaystyle
 \sup_{x\in J}\Phi_{2,J}(x)W_{1,J}(x)^{1/q},
 & p_2\le q,\\[3mm]
 \displaystyle
 \left[
   \int_J W_{1,J}(x)^{r_2/q}
   \,\dd\!\left(\Phi_{2,J}(x)^{r_2}\right)
 \right]^{1/r_2},
 & q<p_2<\infty.
 \end{cases}
 \label{eq:B-infinity-2}
\end{equation}
For \(1\le p_1<\infty\) and \(p_2=\infty\), interchange the indices. With
\begin{equation}
 W_{2,J}(x)
 :=
 \int_x^d \Phi_{2,J}(t)^q u(t)\,\dd t,
 \label{eq:W2-infinite-tail}
\end{equation}
set
\begin{equation}
 \mathfrak B_{1,\infty;J}
 :=
 \begin{cases}
 \displaystyle
 \sup_{x\in J}\Phi_{1,J}(x)W_{2,J}(x)^{1/q},
 & p_1\le q,\\[3mm]
 \displaystyle
 \left[
   \int_J W_{2,J}(x)^{r_1/q}
   \,\dd\!\left(\Phi_{1,J}(x)^{r_1}\right)
 \right]^{1/r_1},
 & q<p_1<\infty.
 \end{cases}
 \label{eq:B-1-infinity}
\end{equation}
If \(p_1=p_2=\infty\), define
\begin{equation}
 \mathfrak B_{\infty,\infty;J}
 :=
 \left[
   \int_J
   \bigl(\Phi_{1,J}(x)\Phi_{2,J}(x)\bigr)^q
   u(x)\,\dd x
 \right]^{1/q}.
 \label{eq:B-infinity-infinity}
\end{equation}
At each regularisation level the last quantity equals the local bilinear
operator norm exactly.

\subsection{Operational characteristics for extended weights}
\label{subsec:operational-local-characteristics}

For arbitrary measurable weights, use the simultaneous regularisation of
Section~\ref{subsec:extended-regularisation},
\(u_m=u\wedge m\) and \(v_{i,m}=v_i+1/m\), and attach the superscript
\((m)\) to the resulting profiles, measures and characteristics. Define
\begin{equation}
 \mathfrak A_J^{(m)}
 :=
 \begin{cases}
 \mathfrak B_{1,J}^{(m)},
 & p_1,p_2<\infty,\ p_1,p_2\le q,\\[1mm]
 \widehat{\mathfrak B}_{2,J}^{(m)},
 & p_1\le q<p_2<\infty,\\[1mm]
 \widehat{\mathfrak B}_{3,J}^{(m)},
 & p_2\le q<p_1<\infty,\\[1mm]
 \mathfrak B_{4,J}^{(m)}+\mathfrak B_{5,J}^{(m)},
 & q<\min\{p_1,p_2\}<\infty,\quad
   \displaystyle\frac1q\le\frac1{p_1}+\frac1{p_2},\\[3mm]
 \mathfrak B_{6,J}^{(m)}+\mathfrak B_{7,J}^{(m)},
 & q<\min\{p_1,p_2\}<\infty,\quad
   \displaystyle\frac1q>\frac1{p_1}+\frac1{p_2},\\[3mm]
 \mathfrak B_{\infty,2;J}^{(m)},
 & p_1=\infty,\ p_2<\infty,\\[1mm]
 \mathfrak B_{1,\infty;J}^{(m)},
 & p_1<\infty,\ p_2=\infty,\\[1mm]
 \mathfrak B_{\infty,\infty;J}^{(m)},
 & p_1=p_2=\infty.
 \end{cases}
 \label{eq:complete-regularised-characteristic}
\end{equation}

\begin{definition}[Operational local characteristic]
\label{def:operational-local-characteristic}
The extended local characteristic is
\begin{equation}
 \mathfrak A_J^{\mathrm{op}}
 :=
 \sup_{m\ge1}\mathfrak A_J^{(m)}.
 \label{eq:operational-local-characteristic}
\end{equation}
\end{definition}

The supremum is always taken over the \emph{complete} local expression. Thus,
in the lower regions it is
\(\sup_m(\mathfrak B_{4,J}^{(m)}+\mathfrak B_{5,J}^{(m)})\) or
\(\sup_m(\mathfrak B_{6,J}^{(m)}+\mathfrak B_{7,J}^{(m)})\), not a sum of
separate operational suprema. Likewise, Lemma~\ref{lem:frozen-tail-decomposition}
applied at each level gives
\begin{equation}
 \begin{aligned}
 \sup_m\widehat{\mathfrak B}_{2,J}^{(m)}
 &\asymp
 \sup_m\bigl(\mathfrak B_{1,J}^{(m)}+\mathfrak B_{2,J}^{(m)}\bigr),\\
 \sup_m\widehat{\mathfrak B}_{3,J}^{(m)}
 &\asymp
 \sup_m\bigl(\mathfrak B_{1,J}^{(m)}+\mathfrak B_{3,J}^{(m)}\bigr),
 \end{aligned}
 \label{eq:operational-mixed-equivalence}
\end{equation}
with comparison constants \(1\) and \(2\).

No unregularised extended-valued Stieltjes measure is introduced. Every
Stieltjes expression is evaluated at a regularised level and only the complete
characteristics are directed through \(\sup_m\). If the compact weights are
already finite and regular, the regularisation may be omitted and the
superscript \(\mathrm{op}\) is suppressed.

\section{The compact bilinear characterisation}
\label{sec:compact-bilinear}

We now prove the local bilinear theorem on a compact interval
\[
 J=[c,d]\Subset I.
\]
The proof combines the compact linear Hardy theorem from
Section~\ref{sec:linear-hardy}, the measure-valued Hardy--Copson principle
from Section~\ref{sec:hardy-copson}, and the characteristics of
Section~\ref{sec:local-characteristics}.

\subsection{An exact freezing identity}
\label{subsec:freezing-identity}

\begin{lemma}[Exact freezing of one input]
\label{lem:freezing-identity}
Assume that \(1\le p_1,p_2<\infty\) and that the weights are regular on
\(J\). For \(g\ge0\), define
\begin{equation}
 w_g(y):=H_Jg(y)^q u(y),
 \qquad y\in J.
 \label{eq:frozen-output-weight}
\end{equation}
Then
\begin{equation}
 C_J
 =
 \sup_{\substack{g\ge0\\0<\|g\|_{L^{p_2}(v_2;J)}<\infty}}
 \frac{L_{p_1,q}(v_1,w_g;J)}{\|g\|_{L^{p_2}(v_2;J)}},
 \label{eq:exact-freezing}
\end{equation}
where \(L_{p_1,q}(v_1,w_g;J)\) is the optimal linear Hardy constant in
\eqref{eq:compact-linear-hardy-inequality}.
\end{lemma}

\begin{proof}
For fixed \(g\),
\[
 \|H_Jf\,H_Jg\|_{L^q(u;J)}
 =\|H_Jf\|_{L^q(w_g;J)}.
\]
Taking first the supremum over \(f\) and then over \(g\) gives
\eqref{eq:exact-freezing}.
No triangle inequality in \(L^q\) is used.
\end{proof}

\subsection{Statement of the compact theorem}
\label{subsec:compact-theorem-statement}

\begin{theorem}[Compact endpoint-safe bilinear characterisation]
\label{thm:compact-bilinear}
Let
\[
 J=[c,d]\Subset I,
 \qquad 0<q<\infty,
 \qquad 1\le p_1,p_2\le\infty,
\]
and let \(u,v_1,v_2:J\to[0,\infty]\) be arbitrary measurable weights.
If \(C_J\) is the optimal constant in \eqref{eq:local-bilinear-inequality},
then
\begin{equation}
 C_J\asymp_{p_1,p_2,q}\mathfrak A_J^{\mathrm{op}},
 \label{eq:compact-master-form}
\end{equation}
where \(\mathfrak A_J^{\mathrm{op}}\) is defined by
\eqref{eq:complete-regularised-characteristic}--\eqref{eq:operational-local-characteristic}.
The regime/characteristic pairing is the following compact restatement of
\eqref{eq:complete-regularised-characteristic}.
\begin{center}
\small
\begin{tabular}{@{}ll@{}}
\(p_1,p_2<\infty,\ p_1,p_2\le q\)
  & \(\mathfrak B_{1,J}^{(m)}\)\\
\(p_1\le q<p_2<\infty\)
  & \(\widehat{\mathfrak B}_{2,J}^{(m)}\)\\
\(p_2\le q<p_1<\infty\)
  & \(\widehat{\mathfrak B}_{3,J}^{(m)}\)\\
\(q<\min\{p_1,p_2\},\ 1/q\le1/p_1+1/p_2\)
  & \(\mathfrak B_{4,J}^{(m)}+\mathfrak B_{5,J}^{(m)}\)\\
\(q<\min\{p_1,p_2\},\ 1/q>1/p_1+1/p_2\)
  & \(\mathfrak B_{6,J}^{(m)}+\mathfrak B_{7,J}^{(m)}\)\\
\(p_1=\infty,\ p_2<\infty\)
  & \(\mathfrak B_{\infty,2;J}^{(m)}\)\\
\(p_1<\infty,\ p_2=\infty\)
  & \(\mathfrak B_{1,\infty;J}^{(m)}\)\\
\(p_1=p_2=\infty\)
  & \(\mathfrak B_{\infty,\infty;J}^{(m)}\).
\end{tabular}
\end{center}
In every row the complete expression is directed through the
regularisation index. In the doubly infinite case,
\begin{equation}
 C_J
 =\sup_{m\ge1}\mathfrak B_{\infty,\infty;J}^{(m)}.
 \label{eq:compact-two-infinite-result}
\end{equation}
All comparison constants are independent of \(J\), the weights, and the
regularisation index.
\end{theorem}

\begin{proof}
\medskip
\noindent\textbf{Step 1. Fixed regularisation.}
By Proposition~\ref{prop:regularised-constants}, it suffices to work at one
regularised level, with constants independent of \(m\). We suppress the
index and write \(u,v_i,C_J,\Phi_{i,J},\mu_{i,J}\) for the regularised
quantities. If \(p_i=1\), \(\Phi_{i,J}\) is the right-continuous Stieltjes
completion, not the pointwise norm of the truncated integration functional.
Theorem~\ref{thm:compact-linear-hardy} incorporates
Lemma~\ref{lem:p-one-completion-invariance} and the finite-ceiling bridge of
\CompactLinearTransferProofLocation{}, so every freezing step below is valid
also for forbidden sets \(\{v_i=\infty\}\), with constants depending only on
the displayed exponents.

\medskip
\noindent\textbf{Step 2. The upper region.}
Assume \(p_1,p_2<\infty\) and \(p_1,p_2\le q\). With
\[
 W_g(x):=\int_x^d H_Jg(y)^q u(y)\,\dd y,
\]
Lemma~\ref{lem:freezing-identity} and the convex branch of
Theorem~\ref{thm:compact-linear-hardy} first give
\begin{equation}
 C_J
 \asymp_{p_1,q}
 \sup_{g\ne0}
 \frac{\displaystyle\sup_{x\in J}\Phi_{1,J}(x)W_g(x)^{1/q}}
 {\|g\|_{L^{p_2}(v_2;J)}}.
 \label{eq:first-upper-reduction}
\end{equation}
For fixed \(x\), let \(u_x=u\mathbf 1_{[x,d]}\). Then
\(W_g(x)^{1/q}=\|H_Jg\|_{L^q(u_x;J)}\), so the two suprema commute and
\begin{equation}
 C_J
 \asymp_{p_1,q}
 \sup_{x\in J}\Phi_{1,J}(x)L_{p_2,q}(v_2,u_x;J).
 \label{eq:second-upper-reduction}
\end{equation}
The tail of \(u_x\) is
\(U_J(\max\{x,t\})\). Since \(p_2\le q\), the second convex linear
criterion therefore gives
\begin{equation}
 C_J
 \asymp_{p_1,p_2,q}
 \sup_{x,t\in J}
 \Phi_{1,J}(x)\Phi_{2,J}(t)
 U_J(\max\{x,t\})^{1/q}.
 \label{eq:upper-double-supremum}
\end{equation}
If \(x\le t\), monotonicity of \(\Phi_{1,J}\) bounds the corresponding
term by \(\mathfrak B_{1,J}\). If \(t\le x\), use monotonicity of
\(\Phi_{2,J}\). Conversely, taking \(t=x\) recovers
\(\mathfrak B_{1,J}\). Hence
\begin{equation}
 C_J\asymp_{p_1,p_2,q}\mathfrak B_{1,J}.
 \label{eq:regular-upper-result}
\end{equation}

\medskip
\noindent\textbf{Step 3. The mixed regions.}
If \(p_1\le q<p_2<\infty\), the same first freezing followed by the
non-convex linear criterion for the second input yields
\begin{equation}
 C_J
 \asymp_{p_1,p_2,q}
 \sup_{x\in J}\Phi_{1,J}(x)
 \left[
   \int_J U_J(\max\{x,t\})^{r_2/q}\,\dd\mu_{2,J}(t)
 \right]^{1/r_2}
 =\widehat{\mathfrak B}_{2,J}.
 \label{eq:regular-first-mixed-result}
\end{equation}
Lemma~\ref{lem:frozen-tail-decomposition} therefore gives
\[
 C_J\asymp_{p_1,p_2,q}\mathfrak B_{1,J}+\mathfrak B_{2,J}.
\]
Interchanging the inputs gives, when \(p_2\le q<p_1<\infty\),
\begin{equation}
 C_J
 \asymp_{p_1,p_2,q}
 \widehat{\mathfrak B}_{3,J}
 \asymp
 \mathfrak B_{1,J}+\mathfrak B_{3,J}.
 \label{eq:regular-second-mixed-result}
\end{equation}

\medskip
\noindent\textbf{Step 4. Power lifting in the lower triangle.}
Assume
\begin{equation}
 q<\min\{p_1,p_2\}<\infty.
 \label{eq:lower-assumption-section6}
\end{equation}
The non-convex linear criterion applied after the first freezing gives
\begin{equation}
 C_J
 \asymp_{p_1,p_2,q}
 \sup_{g\ne0}
 \frac{
 \left[\int_J W_g(x)^{r_1/q}\,\dd\mu_{1,J}(x)\right]^{1/r_1}
 }{\|g\|_{L^{p_2}(v_2;J)}}.
 \label{eq:lower-first-linear-reduction}
\end{equation}
Set
\begin{equation}
 s_1:=\frac{r_1}{q}=\frac{p_1}{p_1-q}>1,
 \qquad
 s_1':=\frac{p_1}{q}>1.
 \label{eq:s1-exponents}
\end{equation}
Raising \eqref{eq:lower-first-linear-reduction} to the \(q\)-th power gives
\begin{equation}
 C_J^q
 \asymp_{p_1,p_2,q}
 \sup_{g\ne0}
 \frac{\|W_g\|_{L^{s_1}(\mu_{1,J})}}
 {\|g\|_{L^{p_2}(v_2;J)}^q}.
 \label{eq:lower-power-lifted}
\end{equation}
Positive \(L^{s_1}\)-duality is exact, giving
\begin{equation}
 \|W_g\|_{L^{s_1}(\mu_{1,J})}
 =
 \sup_{\substack{h\ge0\\
 \|h\|_{L^{s_1'}(\mu_{1,J})}\le1}}
 \int_J W_g(x)h(x)\,\dd\mu_{1,J}(x).
 \label{eq:positive-duality}
\end{equation}
Thus
\begin{equation}
 C_J^q
 \asymp
 \sup_{\substack{h\ge0\\\|h\|_{L^{p_1/q}(\mu_{1,J})}\le1}}
 \sup_{g\ne0}
 \frac{\displaystyle\int_J W_g(x)h(x)\,\dd\mu_{1,J}(x)}
 {\|g\|_{L^{p_2}(v_2;J)}^q}.
 \label{eq:lower-dual-suprema}
\end{equation}
Tonelli gives
\begin{equation}
 \int_J W_g(x)h(x)\,\dd\mu_{1,J}(x)
 =
 \int_J H_Jg(y)^q u(y)
 \left(\int_{[c,y]}h(t)\,\dd\mu_{1,J}(t)\right)\dd y.
 \label{eq:first-tonelli-lower}
\end{equation}
For fixed \(h\), put
\begin{equation}
 w_h(y):=u(y)\int_{[c,y]}h(t)\,\dd\mu_{1,J}(t).
 \label{eq:lower-temporary-weight}
\end{equation}
The inner supremum in \eqref{eq:lower-dual-suprema} is the \(q\)-th
power of the linear Hardy norm from \(L^{p_2}(v_2;J)\) to
\(L^q(w_h;J)\). Since \(q<p_2\), Theorem~\ref{thm:compact-linear-hardy}
gives
\begin{equation}
 \sup_{g\ne0}
 \frac{\displaystyle\int_J H_Jg(y)^q w_h(y)\,\dd y}
 {\|g\|_{L^{p_2}(v_2;J)}^q}
 \asymp_{p_2,q}
 \left[\int_J \mathcal T h(x)^{r_2/q}\,\dd\mu_{2,J}(x)\right]^{q/r_2},
 \label{eq:second-linear-lower}
\end{equation}
where
\begin{equation}
 \mathcal T h(x)
 :=\int_x^d u(y)
 \left(\int_{[c,y]}h(t)\,\dd\mu_{1,J}(t)\right)\dd y.
 \label{eq:lower-measure-operator}
\end{equation}
With
\begin{equation}
 P:=\frac{p_1}{q}>1,
 \qquad
 Q:=\frac{r_2}{q}=\frac{p_2}{p_2-q}>1,
 \label{eq:transformed-PQ}
\end{equation}
we conclude that
\begin{equation}
 C_J^q
 \asymp_{p_1,p_2,q}
 \left\|\mathcal T:L^P(\mu_{1,J})\longrightarrow L^Q(\mu_{2,J})\right\|.
 \label{eq:bilinear-to-measure-operator}
\end{equation}
The duality is therefore used only after power lifting to the Banach space
\(L^{r_1/q}(\mu_{1,J})\).
No duality is applied directly in \(L^q\), so
this step also covers \(0<q<1\).

\medskip
\noindent\textbf{Step 5. Exact Hardy--Copson splitting.}
A second Tonelli rearrangement gives
\begin{equation}
 \begin{aligned}
 \mathcal T h(x)
 &=U_J(x)\int_{[c,x]}h(t)\,\dd\mu_{1,J}(t)
   +\int_{(x,d]}U_J(t)h(t)\,\dd\mu_{1,J}(t)\\
 &=:\mathcal T_Hh(x)+\mathcal T_Ch(x).
 \end{aligned}
 \label{eq:exact-hardy-copson-split}
\end{equation}
The atom at \(t=x\) belongs to the Hardy term and is excluded from the
Copson term. Positivity gives
\begin{equation}
 \|\mathcal T\|\asymp \|\mathcal T_H\|+\|\mathcal T_C\|.
 \label{eq:T-norm-sum}
\end{equation}
Moreover,
\begin{equation}
 P\le Q
 \quad\Longleftrightarrow\quad
 \frac1q\le\frac1{p_1}+\frac1{p_2}.
 \label{eq:transformed-regime-equivalence}
\end{equation}

\medskip
\noindent\textbf{Step 6. The shallow lower triangle.}
Assume \(1/q\le1/p_1+1/p_2\), so \(P\le Q\). For the Hardy component,
Theorem~\ref{thm:compact-measure-hardy-copson} is applied with output
measure \(U_J^Q\mu_{2,J}\). For the Copson component it is applied with
input measure \(\mu_{1,J}\), output measure \(\mu_{2,J}\), and weight
\(U_J\). The corresponding heads and tails are
\begin{equation}
 \begin{aligned}
 M_H(x)&=\Phi_{1,J}(x)^{r_1},
 &N_H(x)&=\int_{[x,d]}U_J(t)^{r_2/q}\,\dd\mu_{2,J}(t),\\
 K_C(x)&=\Phi_{2,J}(x)^{r_2},
 &B_C(x)&=\int_{(x,d]}U_J(t)^{r_1/q}\,\dd\mu_{1,J}(t).
 \end{aligned}
 \label{eq:shallow-head-tail-map}
\end{equation}
Since
\begin{equation}
 P'=\frac{r_1}{q},
 \qquad
 Q=\frac{r_2}{q},
 \label{eq:shallow-exponent-identities}
\end{equation}
we have
\[
 M_H(x)^{1/P'}=\Phi_{1,J}(x)^q,
 \qquad
 N_H(x)^{1/Q}
 =\left[\int_{[x,d]}U_J(t)^{r_2/q}\,\dd\mu_{2,J}(t)\right]^{q/r_2},
\]
and
\[
 K_C(x)^{1/Q}=\Phi_{2,J}(x)^q,
 \qquad
 B_C(x)^{1/P'}
 =\left[\int_{(x,d]}U_J(t)^{r_1/q}\,\dd\mu_{1,J}(t)\right]^{q/r_1}.
\]
Thus the convex Hardy and strict-Copson criteria identify
\begin{equation}
 \|\mathcal T_H\|^{1/q}\asymp_{p_1,p_2,q}\mathfrak B_{4,J},
 \qquad
 \|\mathcal T_C\|^{1/q}\asymp_{p_1,p_2,q}\mathfrak B_{5,J}.
 \label{eq:shallow-component-characteristics}
\end{equation}
The first tail is closed and the second strict, exactly as in
\eqref{eq:B4-characteristic}--\eqref{eq:B5-characteristic}. Together with
\eqref{eq:bilinear-to-measure-operator} and \eqref{eq:T-norm-sum}, and the
elementary equivalence \((a^q+b^q)^{1/q}\asymp_q a+b\), this gives
\begin{equation}
 C_J\asymp_{p_1,p_2,q}\mathfrak B_{4,J}+\mathfrak B_{5,J}.
 \label{eq:regular-shallow-result}
\end{equation}

\medskip
\noindent\textbf{Step 7. The deep lower triangle.}
Assume \(1/q>1/p_1+1/p_2\), so \(Q<P\). If
\[
 \frac1{R_0}:=\frac1Q-\frac1P,
\]
then
\begin{equation}
 R_0=\frac\kappa q,
 \qquad
 \frac{R_0}{Q}=\frac\kappa{r_2},
 \qquad
 \frac{R_0}{P'}=\frac\kappa{r_1}.
 \label{eq:deep-exponent-identities}
\end{equation}
Applying the non-convex Hardy branch to \(\mathcal T_H\) gives
\begin{equation}
 \|\mathcal T_H\|
 \asymp_{p_1,p_2,q}
 \left[
   \int_J
   \left(\int_{[x,d]}U_J(t)^{r_2/q}\,\dd\mu_{2,J}(t)\right)^{\kappa/r_2}
   \dd\!\left(\Phi_{1,J}(x)^\kappa\right)
 \right]^{q/\kappa},
 \label{eq:TH-deep-measure-result}
\end{equation}
whereas the non-convex strict-Copson branch gives
\begin{equation}
 \|\mathcal T_C\|
 \asymp_{p_1,p_2,q}
 \left[
   \int_J
   \left(\int_{(x,d]}U_J(t)^{r_1/q}\,\dd\mu_{1,J}(t)\right)^{\kappa/r_1}
   \dd\!\left(\Phi_{2,J}(x)^\kappa\right)
 \right]^{q/\kappa}.
 \label{eq:TC-deep-measure-result}
\end{equation}
Thus
\begin{equation}
 \|\mathcal T_H\|^{1/q}\asymp\mathfrak B_{6,J},
 \qquad
 \|\mathcal T_C\|^{1/q}\asymp\mathfrak B_{7,J}.
 \label{eq:deep-component-characteristics}
\end{equation}
The differentials are generated by \(\Phi_{1,J}^{\kappa}\) and
\(\Phi_{2,J}^{\kappa}\), respectively, and retain their left-boundary
atoms. Hence
\begin{equation}
 C_J\asymp_{p_1,p_2,q}\mathfrak B_{6,J}+\mathfrak B_{7,J}.
 \label{eq:regular-deep-result}
\end{equation}

\medskip
\noindent\textbf{Step 8. Infinite input exponents.}
If \(p_1=\infty\) and \(p_2<\infty\),
Corollary~\ref{cor:infinite-input-freezing} gives the exact reduction
\[
 C_J=L_{p_2,q}(v_2,\Phi_{1,J}^{\,q}u;J),
\]
so Theorem~\ref{thm:compact-linear-hardy} yields
\begin{equation}
 C_J\asymp_{p_2,q}\mathfrak B_{\infty,2;J}.
 \label{eq:regular-one-infinite-result}
\end{equation}
The case \(p_2=\infty\), \(p_1<\infty\), is symmetric and gives
\(C_J\asymp_{p_1,q}\mathfrak B_{1,\infty;J}\). If both input exponents
are infinite, the same corollary gives exactly
\begin{equation}
 C_J
 =\left[\int_J(\Phi_{1,J}(x)\Phi_{2,J}(x))^q u(x)\,\dd x\right]^{1/q}
 =\mathfrak B_{\infty,\infty;J}.
 \label{eq:regular-two-infinite-result}
\end{equation}

\medskip
\noindent\textbf{Step 9. Removal of regularisation.}
All preceding estimates are uniform in \(m\). Therefore, outside the doubly
infinite case,
\[
 C_{J,m}\asymp_{p_1,p_2,q}\mathfrak A_J^{(m)}
\]
uniformly in \(m\). Proposition~\ref{prop:regularised-constants} and
Corollary~\ref{cor:operational-characteristics} yield
\begin{equation}
 C_J
 \asymp_{p_1,p_2,q}
 \sup_{m\ge1}\mathfrak A_J^{(m)}
 =\mathfrak A_J^{\mathrm{op}}.
 \label{eq:remove-regularisation-final}
\end{equation}
In the doubly infinite case the levelwise identity is exact, giving
\eqref{eq:compact-two-infinite-result}. This completes the proof.
\end{proof}

\section{The directed global characterisation}
\label{sec:global-theorem}

Let \(I\subset\mathbb R\) be any interval with nonempty interior, possibly
bounded or unbounded and with arbitrary endpoint type.  For each compact
restriction \(J=[c,d]\Subset I\), the Hardy operator is restarted at the
attained left endpoint,
\[
 H_Jh(x)=\int_c^x h(t)\,\dd t,
\]
and \(\mathfrak A_J^{\mathrm{op}}\) denotes the complete operational local
characteristic of Section~\ref{sec:local-characteristics}.

\subsection{Directed characteristics}
\label{subsec:directed-characteristics}

\begin{definition}[Directed global characteristic]
\label{def:directed-global-characteristic}
Define
\begin{equation}
\mathfrak A_I^{\mathrm{dir}}
 :=
 \sup_{J\Subset I}\mathfrak A_J^{\mathrm{op}}
 =
 \sup_{J\Subset I}\sup_{m\ge1}\mathfrak A_J^{(m)},
\label{eq:directed-global-characteristic}
\end{equation}
where \(\mathfrak A_J^{(m)}\) is the complete regime expression from
Section~\ref{sec:local-characteristics}, formed from the regularised weights
on the fixed interval \(J\).
\end{definition}

The complete local expression is directed as a whole.  Thus, in the lower
triangle one takes
\[
 \sup_{J\Subset I}\sup_{m\ge1}
 \bigl(\mathfrak B_{4,J}^{(m)}+\mathfrak B_{5,J}^{(m)}\bigr)
 \quad\text{or}\quad
 \sup_{J\Subset I}\sup_{m\ge1}
 \bigl(\mathfrak B_{6,J}^{(m)}+\mathfrak B_{7,J}^{(m)}\bigr),
\]
not independent directed suprema of the two summands. The same convention
applies to the frozen and strict components of the mixed characteristics.
Analytically, the order of construction remains local. We fix \(J\), restart the
operator, regularise the restricted weights, form the local profiles and
Stieltjes measures, and only then take the regularisation and compact
suprema.

\subsection{The global theorem}
\label{subsec:directed-global-theorem}

\begin{theorem}[Directed global bilinear characterisation]
\label{thm:directed-global-bilinear}
Let \(I\subset\mathbb R\) be any interval with nonempty interior, let
\[
 0<q<\infty,
 \qquad
 1\le p_1,p_2\le\infty,
\]
and let \(u,v_1,v_2:I\to[0,\infty]\) be arbitrary measurable weights. Let
\(\mathcal B_I(f,g)=H_If\starprod H_Ig\) and let \(C_I\) be the infimum of
all finite admissible constants in \eqref{eq:global-bilinear-hardy-inequality},
with \(\inf\varnothing=+\infty\). Then
\begin{equation}
C_I
 \asymp_{p_1,p_2,q}
 \mathfrak A_I^{\mathrm{dir}},
\label{eq:directed-global-result}
\end{equation}
with comparison constants independent of the interval and of the weights.
If \(p_1=p_2=\infty\), then the equivalence is exact, with
\begin{equation}
C_I=\mathfrak A_I^{\mathrm{dir}}.
\label{eq:directed-global-infinite-equality}
\end{equation}
\end{theorem}

\begin{proof}
By Proposition~\ref{prop:directed-exhaustion},
\[
 C_I=\sup_{J\Subset I}C_J.
\]
The compact characterisation, Theorem~\ref{thm:compact-bilinear}, gives
\(C_J\asymp_{p_1,p_2,q}\mathfrak A_J^{\mathrm{op}}\) with constants uniform
in \(J\) and in the weights.  Taking the supremum over \(J\Subset I\) proves
\eqref{eq:directed-global-result}.  When \(p_1=p_2=\infty\), the compact
identity is exact for every \(J\), so the same supremum gives
\eqref{eq:directed-global-infinite-equality}.
\end{proof}

Consequently, the bilinear inequality is bounded on \(I\) if and only if
\(\mathfrak A_I^{\mathrm{dir}}<\infty\).

\subsection{Evaluation along a fixed compact exhaustion}
\label{subsec:fixed-global-exhaustion}

\begin{proposition}[Exhaustion independence]
\label{prop:characteristic-exhaustion-independence}
Let
\begin{equation}
J_1\subset J_2\subset\cdots\Subset I,
 \qquad
 \bigcup_{n=1}^{\infty}J_n=I.
\label{eq:section7-fixed-exhaustion}
\end{equation}
Then
\begin{equation}
\mathfrak A_I^{\mathrm{dir}}
 \asymp_{p_1,p_2,q}
 \sup_{n\ge1}\mathfrak A_{J_n}^{\mathrm{op}},
\label{eq:characteristic-fixed-exhaustion}
\end{equation}
and
\begin{equation}
C_I=\lim_{n\to\infty}C_{J_n}.
\label{eq:norm-fixed-exhaustion}
\end{equation}
\end{proposition}

\begin{proof}
The norm identity is Corollary~\ref{cor:exhaustion-independence}.  Since the
compact comparison constants are uniform, monotonicity of the operator norms
under \(J\subset J_n\) yields
\(
 \mathfrak A_J^{\mathrm{op}}\lesssim C_J\le C_{J_n}
 \lesssim\mathfrak A_{J_n}^{\mathrm{op}}
\)
for all sufficiently large \(n\).  Taking the required suprema proves
\eqref{eq:characteristic-fixed-exhaustion}.
\end{proof}

\begin{remark}
\label{rem:characteristics-not-monotone}
Although \(C_{J_n}\) is nondecreasing, individual local characteristics need
not be.
Enlarging the compact interval changes its left endpoint, the local
profiles, the output tail and the Stieltjes measures.  The proposition
asserts equivalence of the resulting suprema, not termwise monotonicity.
\end{remark}

\subsection{Why no single global Stieltjes profile is used}
\label{subsec:no-global-profile}

The directed formulation also removes the endpoint defect that occurs when
\(I\) has no smallest point.  A global profile may have positive limiting
mass at the missing left endpoint while its interior Stieltjes derivative
records no such mass.  On each \(J=[c,d]\Subset I\), however, the attained
left endpoint carries the atom
\[
 \dd\!\left(\Phi_{i,J}^{\,r_i}\right)(\{c\})
 =
 \Phi_{i,J}(c)^{r_i},
\]
so the local characteristic is endpoint-safe.  Thus the global object in
Theorem~\ref{thm:directed-global-bilinear} is the directed family of compact
characteristics, not a single Stieltjes measure constructed directly on
\(I\). Subsection~\ref{subsec:nonattained-endpoint-example} gives the concrete
failure mechanism.

\section{Agreement with the classical characterisation}
\label{sec:classical-comparison}

We now verify that the endpoint-safe local characteristics reduce to the
classical conditions in the regular interior range.  Fix
\(J=[c,d]\Subset I\), assume
\[
 1<p_1,p_2,q<\infty,
 \qquad
 0<u,v_1,v_2<\infty \quad\text{a.e. on }J,
\]
and put
\begin{equation}
 U(x):=\int_x^d u(t)\,\dd t,
 \qquad
 V_i(x):=\int_c^x v_i(t)^{1-p_i'}\,\dd t,
 \qquad i=1,2.
 \label{eq:classical-UV}
\end{equation}
Then
\begin{equation}
 \Phi_{i,J}(x)=V_i(x)^{1/p_i'},
 \qquad i=1,2.
 \label{eq:classical-profile-identification}
\end{equation}
Whenever \(q<p_i\), let
\begin{equation}
 \frac1{r_i}=\frac1q-\frac1{p_i},
 \qquad
 r_i=\frac{p_iq}{p_i-q}.
 \label{eq:classical-ri-definition}
\end{equation}
Since \(q>1\),
\begin{equation}
 \frac{r_i}{p_i'}-1=\frac{r_i}{q'}.
 \label{eq:classical-ri-identity}
\end{equation}
In the deep lower triangle,
\begin{equation}
 \frac1\kappa
 :=\frac1q-\frac1{p_1}-\frac1{p_2}>0.
 \label{eq:classical-kappa-definition}
\end{equation}

The comparison is made only in this interior setting.  Here the profiles are
absolutely continuous and vanish at the left endpoint, so their Stieltjes
measures have no atoms and the closed/strict distinction disappears.  The
resulting formulas are precisely the classical integral structures of
\cite{AguilarCanestroOrtegaRamirez2012}.
The equation-number notation used in
\cite{Krepela2017} is noted at the end of the section.

\subsection{The canonical conditions \texorpdfstring{\(A_1,\ldots,A_7\)}{A1--A7}}
\label{subsec:canonical-A-conditions}

For \(p_1,p_2\le q\), define
\begin{equation}
 A_1
 :=
 \sup_{c<x<d}
 U(x)^{1/q}V_1(x)^{1/p_1'}V_2(x)^{1/p_2'}.
 \label{eq:classical-A1}
\end{equation}

If \(p_1\le q<p_2\), put
\begin{equation}
 A_2
 :=
 \sup_{c<x<d}V_1(x)^{1/p_1'}
 \left[
   \int_x^d
   U(t)^{r_2/q}V_2(t)^{r_2/q'}
   v_2(t)^{1-p_2'}\,\dd t
 \right]^{1/r_2}.
 \label{eq:classical-A2}
\end{equation}
If \(p_2\le q<p_1\), define symmetrically
\begin{equation}
 A_3
 :=
 \sup_{c<x<d}V_2(x)^{1/p_2'}
 \left[
   \int_x^d
   U(t)^{r_1/q}V_1(t)^{r_1/q'}
   v_1(t)^{1-p_1'}\,\dd t
 \right]^{1/r_1}.
 \label{eq:classical-A3}
\end{equation}
The complete mixed conditions are \(A_1+A_2\) and \(A_1+A_3\), respectively.

In the shallow lower triangle
\[
 q<\min\{p_1,p_2\},
 \qquad
 \frac1q\le\frac1{p_1}+\frac1{p_2},
\]
define
\begin{equation}
 A_4
 :=
 \sup_{c<x<d}V_1(x)^{1/p_1'}
 \left[
   \int_x^d
   U(t)^{r_2/q}V_2(t)^{r_2/q'}
   v_2(t)^{1-p_2'}\,\dd t
 \right]^{1/r_2},
 \label{eq:classical-A4}
\end{equation}
and
\begin{equation}
 A_5
 :=
 \sup_{c<x<d}V_2(x)^{1/p_2'}
 \left[
   \int_x^d
   U(t)^{r_1/q}V_1(t)^{r_1/q'}
   v_1(t)^{1-p_1'}\,\dd t
 \right]^{1/r_1}.
 \label{eq:classical-A5}
\end{equation}
Thus \(A_4\) and \(A_2\), and likewise \(A_5\) and \(A_3\), have the same
analytic forms but occur in different exponent regions.

Finally, in the deep lower triangle
\[
 q<\min\{p_1,p_2\},
 \qquad
 \frac1q>\frac1{p_1}+\frac1{p_2},
\]
define
\begin{equation}
 \begin{aligned}
 A_6
 &:=[\int_c^d
   \bigl(\int_x^d
     U(t)^{r_2/q}V_2(t)^{r_2/q'}
     v_2(t)^{1-p_2'}\,\dd t\bigr)^{\kappa/r_2}
   V_1(x)^{\kappa/r_2'}v_1(x)^{1-p_1'}\,\dd x]^{1/\kappa},
 \\
 A_7
 &:=[\int_c^d
   \bigl(\int_x^d
     U(t)^{r_1/q}V_1(t)^{r_1/q'}
     v_1(t)^{1-p_1'}\,\dd t\bigr)^{\kappa/r_1}
   V_2(x)^{\kappa/r_1'}v_2(x)^{1-p_2'}\,\dd x]^{1/\kappa}.
 \end{aligned}
 \label{eq:classical-A6-A7}
\end{equation}

\subsection{Densities and termwise identification}
\label{subsec:termwise-classical-identification}

The conversion from the local Stieltjes characteristics is elementary.

\begin{lemma}[Classical profile densities]
\label{lem:classical-profile-densities}
If \(q<p_i\), then
\begin{equation}
 \dd(\Phi_{i,J}^{r_i})
 =\frac{r_i}{p_i'}
   V_i^{r_i/q'}v_i^{1-p_i'}\,\dd x.
 \label{eq:ri-profile-density}
\end{equation}
In the deep lower triangle,
\begin{equation}
 \begin{aligned}
 \dd(\Phi_{1,J}^{\kappa})
 &=\frac{\kappa}{p_1'}
   V_1^{\kappa/r_2'}v_1^{1-p_1'}\,\dd x,
 \\
 \dd(\Phi_{2,J}^{\kappa})
 &=\frac{\kappa}{p_2'}
   V_2^{\kappa/r_1'}v_2^{1-p_2'}\,\dd x.
 \end{aligned}
 \label{eq:kappa-profile-densities}
\end{equation}
\end{lemma}

\begin{proof}
From \(\Phi_{i,J}^{r_i}=V_i^{r_i/p_i'}\) and
\(V_i'=v_i^{1-p_i'}\), differentiation together with
\eqref{eq:classical-ri-identity} gives \eqref{eq:ri-profile-density}.
Moreover
\[
 \frac1{p_1'}-\frac1\kappa=\frac1{r_2'},
 \qquad
 \frac1{p_2'}-\frac1\kappa=\frac1{r_1'},
\]
which gives \eqref{eq:kappa-profile-densities} by the same calculation.
\end{proof}

Because \(p_i>1\), \(\Phi_{i,J}(c)=0\), and the measures in the lemma are
absolutely continuous.  Hence \([x,d]\) and \((x,d]\) give the same
integrals in every characteristic appearing below.

\begin{proposition}[Identification of the seven structures]
\label{prop:seven-classical-identifications}
Whenever the corresponding quantities are defined,
\begin{equation}
 \begin{aligned}
 \mathfrak B_{1,J}
 &=A_1,
 \\
 \mathfrak B_{2,J}
 &=\left(\frac{r_2}{p_2'}\right)^{1/r_2}A_2,
 &\qquad
 \mathfrak B_{3,J}
 &=\left(\frac{r_1}{p_1'}\right)^{1/r_1}A_3,
 \\
 \mathfrak B_{4,J}
 &=\left(\frac{r_2}{p_2'}\right)^{1/r_2}A_4,
 &
 \mathfrak B_{5,J}
 &=\left(\frac{r_1}{p_1'}\right)^{1/r_1}A_5,
 \\
 \mathfrak B_{6,J}
 &=\left(\frac{r_2}{p_2'}\right)^{1/r_2}
   \left(\frac{\kappa}{p_1'}\right)^{1/\kappa}A_6,
 &
 \mathfrak B_{7,J}
 &=\left(\frac{r_1}{p_1'}\right)^{1/r_1}
   \left(\frac{\kappa}{p_2'}\right)^{1/\kappa}A_7.
 \end{aligned}
 \label{eq:seven-classical-identifications}
\end{equation}
\end{proposition}

\begin{proof}
The first identity follows from \eqref{eq:classical-profile-identification}.
For example, inserting \eqref{eq:ri-profile-density} with \(i=2\) into the
strict mixed characteristic gives
\[
 \mathfrak B_{2,J}
 =\left(\frac{r_2}{p_2'}\right)^{1/r_2}A_2.
\]
The identities for \(\mathfrak B_{3,J},\mathfrak B_{4,J}\), and
\(\mathfrak B_{5,J}\) follow identically (with the two inputs interchanged
where appropriate).  In the deep region, use
\eqref{eq:ri-profile-density} in the inner integral and
\eqref{eq:kappa-profile-densities} in the outer Stieltjes integral.
The two
constants are then raised to the powers \(1/r_i\) and \(1/\kappa\), giving
the last two identities.
\end{proof}

Together with Lemma~\ref{lem:frozen-tail-decomposition}, the proposition
also gives the complete mixed criteria
\begin{equation}
 \widehat{\mathfrak B}_{2,J}\asymp_{p_2,q}A_1+A_2,
 \qquad
 \widehat{\mathfrak B}_{3,J}\asymp_{p_1,q}A_1+A_3.
 \label{eq:complete-classical-mixed-criteria}
\end{equation}
Thus the strict term alone does not display the frozen contribution. An
equivalent one-expression formulation may absorb both pieces by integration
by parts.

\subsection{Recovery of the classical regime theorem}
\label{subsec:recovery-classical-regimes}

\begin{theorem}[Recovery of the classical \(A_1,\ldots,A_7\) conditions]
\label{thm:recovery-classical-A1-A7}
Under the assumptions above,
\begin{equation}
 \begin{array}{c|c}
 \text{parameter region} & \text{optimal compact constant} \\ \hline
 p_1,p_2\le q
   & C_J\asymp_{p_1,p_2,q}A_1 \\[1mm]
 p_1\le q<p_2
   & C_J\asymp_{p_1,p_2,q}A_1+A_2 \\[1mm]
 p_2\le q<p_1
   & C_J\asymp_{p_1,p_2,q}A_1+A_3 \\[1mm]
 q<\min\{p_1,p_2\},\ q^{-1}\le p_1^{-1}+p_2^{-1}
   & C_J\asymp_{p_1,p_2,q}A_4+A_5 \\[1mm]
 q<\min\{p_1,p_2\},\ q^{-1}>p_1^{-1}+p_2^{-1}
   & C_J\asymp_{p_1,p_2,q}A_6+A_7.
 \end{array}
 \label{eq:classical-regime-recovery}
\end{equation}
\end{theorem}

\begin{proof}
Apply Theorem~\ref{thm:compact-bilinear} and
Proposition~\ref{prop:seven-classical-identifications}. In the two mixed
regions use \eqref{eq:complete-classical-mixed-criteria}.  This gives all
five rows of \eqref{eq:classical-regime-recovery}.
\end{proof}

Some classical formulations retain an additional \(A_1\) in the shallow
lower triangle. Here \(A_1=\mathfrak B_{1,J}\lesssim C_J\lesssim A_4+A_5\)
by Theorem~\ref{thm:compact-bilinear} and
Proposition~\ref{prop:seven-classical-identifications}. The analogous
deep-regime auxiliary terms are controlled by \(A_6+A_7\). Thus the displayed
reduced forms are equivalent to the longer formulations. Compare also
\cite[Propositions~3.2--3.3]{Krepela2017}.

\subsection{Relation with K\v{r}epela's notation}
\label{subsec:krepela-cross-reference}

K\v{r}epela \cite{Krepela2017} labels his conditions by equation number rather
than by the sequence \(A_1,\ldots,A_7\).  At regime level, his
\(A^{(2.2)}\), \(A^{(2.3)}\), \(A^{(2.4)}+A^{(2.5)}\), and
\(A^{(2.6)}+A^{(2.7)}\) correspond respectively to the upper, mixed, shallow,
and deep structures above (with input interchange in the symmetric mixed
case).  In particular,
\[
 A^{(2.3)}\asymp A^{(2.2)}+A^{(2.5)}
\]
is the same structural completion as \(A_1+A_2\).  A fuller equation-number
cross-reference is given below.

\subsubsection*{Equation-number correspondence}

For ease of comparison with the 2017 K\v{r}epela paper discussed
above, the regime-level correspondence between the canonical
\(A_1,\ldots,A_7\) notation and his equation-number notation is
\[
 \begin{array}{c|c|c}
 \text{parameter region}
 & \text{canonical notation}
 & \text{K\v{r}epela notation}
 \\ \hline
 p_1,p_2\le q
 & A_1
 & A^{(2.2)}
 \\[1mm]
 p_1\le q<p_2
 & A_1+A_2
 & A^{(2.3)}
 \\[1mm]
 p_2\le q<p_1
 & A_1+A_3
 & \text{index-swapped }A^{(2.3)}
 \\[1mm]
 q<\min\{p_1,p_2\},\ q^{-1}\le p_1^{-1}+p_2^{-1}
 & A_4+A_5
 & A^{(2.4)}+A^{(2.5)}
 \\[1mm]
 q<\min\{p_1,p_2\},\ q^{-1}>p_1^{-1}+p_2^{-1}
 & A_6+A_7
 & A^{(2.6)}+A^{(2.7)}.
 \end{array}
\]
In the first mixed region the integration-by-parts comparison
\[
 A^{(2.3)}\asymp A^{(2.2)}+A^{(2.5)}
\]
corresponds to the complete condition \(A_1+A_2\).  The identification of
the two individual shallow terms depends on the order in which the inputs
are frozen, whereas their sum is invariant under interchange of input
indices.

\subsection{Interpretation}
\label{subsec:interpretation-classical-comparison}

Theorem~\ref{thm:recovery-classical-A1-A7} shows that no new boundedness
criterion is introduced in the regular interior range.  The extension is
structural. Absolutely continuous densities are replaced by local
Lebesgue--Stieltjes measures, the \(p_i=1\) left-boundary atom is retained,
closed Hardy and strict Copson pieces are separated at common atoms,
\(p_i=\infty\) is treated through exact integration functionals, extended
weights are handled operationally, and arbitrary intervals are reached by
directed compact restrictions.  Thus the endpoint-safe theorem agrees with
the classical \(A_1,\ldots,A_7\) characterisation wherever the latter
applies while remaining valid beyond that setting.

\section{Endpoint and structural examples}
\label{sec:examples}

The following examples isolate four structural features of the endpoint-safe
formulation. These are the left-boundary atom, completion of a strict mixed tail by its
frozen term, the closed Hardy/strict Copson allocation of a common atom, and
operational blow-up for zero-cost inputs.

\subsection{A non-attained left endpoint}
\label{subsec:nonattained-endpoint-example}

Take
\[
 I=(0,1),\qquad p_1=p_2=1,\qquad q=\frac12,
 \qquad u=v_1=v_2=1.
\]
For \(F=H_If\) and \(G=H_Ig\), Cauchy--Schwarz and Tonelli give
\[
 \begin{aligned}
 \|FG\|_{L^{1/2}(0,1)}
 &=\left(\int_0^1\sqrt{F(x)G(x)}\,\dd x\right)^2\\
 &\le
 \left(\int_0^1F(x)\,\dd x\right)
 \left(\int_0^1G(x)\,\dd x\right)\\
 &=
 \left(\int_0^1(1-t)f(t)\,\dd t\right)
 \left(\int_0^1(1-t)g(t)\,\dd t\right)\\
 &\le \|f\|_{L^1(0,1)}\|g\|_{L^1(0,1)}.
 \end{aligned}
\]
Thus \(C_I\le1\). For
\(f_n=g_n=n\mathbf 1_{(0,1/n)}\), both input norms equal one and
\[
 \|H_If_n\,H_Ig_n\|_{L^{1/2}(0,1)}
 =\left(1-\frac1{2n}\right)^2\longrightarrow1,
\]
so
\begin{equation}
 C_I=1.
 \label{eq:open-endpoint-exact-norm}
\end{equation}

A naive global \(p_i=1\) profile on \((0,1)\) is constant. In fact, \(\Phi_i(x)=1\). Its Stieltjes measure inside the open interval is therefore
zero and cannot record the initial profile mass. On a compact restriction
\(J=[c,d]\Subset(0,1)\), however,
\[
 \Phi_{1,J}=\Phi_{2,J}=1,
 \qquad
 r_1=r_2=1,
 \qquad
 \mu_{1,J}=\mu_{2,J}=\delta_c,
 \qquad
 U_J(x)=d-x.
\]
Since \(1/q=1/p_1+1/p_2\), the shallow lower-triangle terms become
\begin{equation}
 \mathfrak B_{4,J}=(d-c)^2,
 \qquad
 \mathfrak B_{5,J}=0.
 \label{eq:open-endpoint-B4}
\end{equation}
The strict term vanishes because the only profile mass is the atom at \(c\),
which is excluded from every tail \((x,d]\). The same argument as above,
with interval length \(d-c\), yields
\[
 C_J=(d-c)^2,
 \qquad
 \sup_{J\Subset(0,1)}C_J=1.
\]
Thus the local left-boundary atom exactly restores the mass missed by an
interior Stieltjes derivative at a non-attained endpoint.

\subsection{The strict mixed tail is not a complete condition}
\label{subsec:mixed-strict-tail-example}

Let \(J=[0,1]\), \(0<\alpha<1\), \(p_1=q=1\), \(p_2=2\), and \(u=1\).
Choose
\begin{equation}
 v_1(t)=
 \begin{cases}
 \infty,&0\le t<\alpha,\\
 1,&\alpha\le t\le1,
 \end{cases}
 \qquad
 v_2(t)=
 \begin{cases}
 1,&0\le t\le\alpha,\\
 \infty,&\alpha<t\le1.
 \end{cases}
 \label{eq:separated-support-weights}
\end{equation}
Finite-norm \(f\) and \(g\) are therefore supported, respectively, to the
right and left of \(\alpha\). The associated profiles and measure are
\[
 \Phi_{1,J}(x)=
 \begin{cases}
 0,&x<\alpha,\\
 1,&x\ge\alpha,
 \end{cases}
 \qquad
 \Phi_{2,J}(x)=\min\{x,\alpha\}^{1/2},
 \qquad
 \mu_{2,J}=\mathbf 1_{(0,\alpha)}(t)\,\dd t.
\]
Consequently the strict mixed term vanishes.
If \(x<\alpha\), then
\(\Phi_{1,J}(x)=0\), whereas if \(x\ge\alpha\), then
\(\mu_{2,J}((x,1])=0\). Hence
\begin{equation}
 \mathfrak B_{2,J}=0,
 \qquad
 \mathfrak B_{1,J}=(1-\alpha)\sqrt{\alpha},
 \qquad
 \widehat{\mathfrak B}_{2,J}=(1-\alpha)\sqrt{\alpha}.
 \label{eq:separated-complete-mixed}
\end{equation}

The operator norm has the same positive value. Put
\(G=\int_0^\alpha g(t)\,\dd t\). Tonelli gives
\begin{equation}
 \|H_Jf\,H_Jg\|_{L^1(0,1)}
 =G\int_\alpha^1(1-t)f(t)\,\dd t
 \le (1-\alpha)\sqrt{\alpha}\,
 \|f\|_{L^1(v_1;J)}\|g\|_{L^2(v_2;J)}.
 \label{eq:separated-output-calculation}
\end{equation}
For
\[
 g_\alpha=\alpha^{-1/2}\mathbf 1_{(0,\alpha)},
 \qquad
 f_n=n\mathbf 1_{(\alpha,\alpha+1/n)},
\]
the input norms equal one and the output tends to
\((1-\alpha)\sqrt\alpha\). Therefore
\begin{equation}
 C_J=(1-\alpha)\sqrt\alpha>0
 \qquad\text{while}\qquad
 \mathfrak B_{2,J}=0.
 \label{eq:separated-exact-norm}
\end{equation}
Thus the strict tail alone is incomplete. The missing contribution is exactly
the frozen term \(\mathfrak B_{1,J}\).

\subsection{Allocation of a common atom}
\label{subsec:common-atom-example}

The closed/strict convention is already forced by a single common atom. In
the measure operator of \eqref{eq:exact-hardy-copson-split}, take
\(\mu_1=\mu_2=\delta_\xi\) for some \(\xi\in(c,d)\) and normalise
\(U(\xi)=1\). At the only point seen by the output measure,
\[
 \mathcal T_Hh(\xi)
 =U(\xi)\int_{[c,\xi]}h\,\dd\delta_\xi=h(\xi),
 \qquad
 \mathcal T_Ch(\xi)
 =\int_{(\xi,d]}Uh\,\dd\delta_\xi=0.
\]
Hence the common atom belongs exactly once, to the closed Hardy component. If
the Copson tail were replaced by \([\xi,d]\), the same atom would be counted
twice.
If both components were strict, it would be lost. The asymmetric
allocation \([c,x]\) versus \((x,d]\) is therefore structural rather than
notational.

For comparison, the doubly infinite branch requires no separate mechanism.
When \(J=[0,1]\) and \(u=v_1=v_2=1\), the choice \(f=g=1\) gives directly
\(C_J=(2q+1)^{-1/q}=\mathfrak B_{\infty,\infty;J}\).

\subsection{A zero-cost input and operational blow-up}
\label{subsec:zero-cost-input-example}

Let \(J=[0,1]\), \(p_1=p_2=q=1\), \(u=v_2=1\), and
\[
 v_1(t)=
 \begin{cases}
 0,&0<t<1/2,\\
 1,&1/2\le t\le1.
 \end{cases}
\]
With \(f=\mathbf 1_{(0,1/2)}\) and \(g=1\),
\[
 \|f\|_{L^1(v_1;J)}=0,
 \qquad
 \|g\|_{L^1(v_2;J)}=1,
\]
whereas
\[
 H_Jf(x)=
 \begin{cases}
 x,&0\le x\le1/2,\\
 1/2,&1/2<x\le1,
 \end{cases}
 \qquad
 H_Jg(x)=x.
\]
Therefore
\begin{equation}
 \|H_Jf\,H_Jg\|_{L^1(0,1)}
 =\int_0^{1/2}x^2\,\dd x
  +\frac12\int_{1/2}^1x\,\dd x
 =\frac{11}{48}>0,
 \label{eq:zero-cost-positive-output}
\end{equation}
so no finite constant is possible, and hence
\begin{equation}
 C_J=\infty.
 \label{eq:zero-cost-infinite-norm}
\end{equation}
Under the simultaneous regularisation of Section~\ref{sec:weighted-spaces},
the same test gives
\[
 C_{J,m}\ge \frac{11m}{24(1+m^{-1})}\longrightarrow\infty,
 \qquad
 \mathfrak A_J^{\mathrm{op}}=\infty.
\]
Thus operational regularisation detects precisely the blow-up forced by a
zero-cost input producing nonzero output.

\section{Discussion and concluding remarks}
\label{sec:discussion}

The results above give an endpoint-safe characterisation of the weighted
bilinear Hardy inequality on an arbitrary real interval for
\(0<q<\infty\) and \(1\le p_1,p_2\le\infty\).  The global statement is
obtained by restarting the operator on compact restrictions
\(J=[c,d]\Subset I\), forming the complete regularised local
characteristic, and then taking the directed supremum.  Thus the central
global object is
\[
 C_I\asymp_{p_1,p_2,q}
 \sup_{J\Subset I}\mathfrak A_J^{\mathrm{op}}
 =
 \sup_{J\Subset I}\sup_{m\ge1}\mathfrak A_J^{(m)},
\]
with equality in the doubly infinite case.  This local-to-global order is
part of the formulation rather than an auxiliary device.
Endpoint
completion and extended-weight regularisation are performed before the
directed limit.

\subsection{Nature of the contribution}
\label{subsec:nature-contribution}

Classical interior boundedness criteria for weighted bilinear Hardy
inequalities are already available by discretisation, anti-discretisation
and iterative reduction methods.
See
\cite{AguilarCanestroOrtegaRamirez2012,Krepela2017,
KanjilalPerssonShambilova2019,StepanovShambilova2019,
GogatishviliJainKanjilal2022}. In particular,
\cite{KanjilalPerssonShambilova2019} states the five-region
\(A_1,\ldots,A_7\) theorem for \(0<q<\infty\) and
\(1<p_1,p_2<\infty\), while \cite{StepanovShambilova2019} treats a direct
same-direction Oinarov-kernel product that contains the ordinary Hardy
product for unit kernels. Endpoint input exponents and \(q<1\) also occur in
earlier multidimensional theory \cite{BilgicliMustafayevUnver2020}, while
recent adjacent developments treat metric-measure bilinear inequalities
\cite{RuzhanskyShriwastawaVerma2024}, weak-type targets
\cite{GarciaGarciaOrtegaSalvador2023}, and the mixed product \(Hf\,H^*g\)
\cite{MohantyJainJain2025}.
Section~\ref{sec:classical-comparison} verifies that, when
\(1<p_1,p_2,q<\infty\) and the weights are regular and finite, the present
characteristics reduce to the classical quantities \(A_1,\ldots,A_7\).
The unit-kernel comparison with Temirkhanova, Zhangabergenova and Oinarov
further confirms that the classical interior criteria and the associated
reduction mechanisms are not separate priority claims of the present paper.
Accordingly, no claim is made for a first bilinear boundedness
characterisation, a first exponent range, the freezing principle, or a first
direct reduction. The contribution claimed here is restricted to the
combination of directed compact restart, explicit endpoint and atomic
bookkeeping, the closed-Hardy/strict-Copson allocation, one direct
finite-output architecture, and the present operational treatment of
extended-valued weights on arbitrary real intervals. To the best of our
knowledge, we are not aware of an earlier continuous same-direction strong-type
theorem assembling all of these features in one formulation.

\subsection{Endpoint and atomic architecture}
\label{subsec:role-compact-restrictions}

The compact-restriction mechanism is essential even on bounded open
intervals.  If the left endpoint of \(I\) is not attained, a global
\(p_i=1\) profile may have no Stieltjes variation inside \(I\) although
the corresponding integration functional is non-trivial.  Restarting on
\(J=[c,d]\Subset I\) makes \(c\) an attained boundary and records the
missing mass through the atom
\(\dd(\Phi_{i,J}^{r_i})(\{c\})=\Phi_{i,J}(c)^{r_i}\).  The global norm is
then recovered exactly from \(C_I=\sup_{J\Subset I}C_J\).
See
Section~\ref{sec:global-theorem} and
Example~\ref{subsec:nonattained-endpoint-example}.

The same endpoint bookkeeping determines the lower-triangle geometry.
The exact decomposition in Section~\ref{sec:compact-bilinear} assigns the
diagonal atom to the closed Hardy part and excludes it from the strict
Copson part.  Consequently the characteristics
\(\mathfrak B_{4,J},\mathfrak B_{6,J}\) use closed tails, whereas
\(\mathfrak B_{5,J},\mathfrak B_{7,J}\) use strict tails.  This
difference disappears for atomless profile measures but is indispensable
at endpoints.
Example~\ref{subsec:common-atom-example} shows the failure
caused by an incorrect allocation.  Likewise, the mixed criterion must
retain both the frozen and strict-tail components,
\(\widehat{\mathfrak B}_{2,J}\asymp
\mathfrak B_{1,J}+\mathfrak B_{2,J}\). The separated-support example
\ref{subsec:mixed-strict-tail-example} shows that the strict term alone
may vanish while the operator norm remains positive.

\subsection{Proof architecture and extended weights}
\label{subsec:directness-proof}

The bilinear proof is direct in the sense that it avoids the
usual discretising-sequence and anti-discretisation machinery.  Exact
freezing reduces the upper and mixed regimes to compact linear Hardy
inequalities, while the lower triangle is treated by power lifting,
positive duality and the measure-valued Hardy--Copson theorem.  When
\(q<1\), no duality is taken in the original quasi-Banach output space.
After lifting, the relevant exponent is
\(r_i/q=p_i/(p_i-q)>1\), so duality occurs only in an auxiliary Banach
space.  The auxiliary measure theorem itself is obtained from finite
chains, exact strict-Copson reversal and atom-preserving partition
approximation, with the full technical arguments recorded in the appendices.

Related power-lifting and duality arguments, together with Stieltjes-type
differential reductions in the regular interior setting, also occur in
\cite{TemirkhanovaZhangabergenovaOinarov2026}. These mechanisms are not
claimed here as separate innovations.

For weights taking values in \([0,\infty]\), Section~\ref{sec:weighted-spaces}
uses the simultaneous regularisation
\[
 u_m=u\wedge m,
 \qquad
 v_{i,m}=v_i+\frac1m,
\]
for which \(C_{J,m}\uparrow C_J\) and
\(\mathfrak A_J^{\mathrm{op}}=\sup_m\mathfrak A_J^{(m)}\).  This
operational formulation handles zero-cost and forbidden input regions
without assigning a formal Stieltjes measure to an extended-valued
unregularised profile.  In particular, the proof needs no assertion of
weak convergence of the regularised profile measures.
Only the uniform
local equivalence and monotone convergence of the operator constants are
used.  The detailed limiting arguments are collected in
\ExtendedRegularisationProofLocation.

\subsection{Scope and further directions}
\label{subsec:scope-limitations}

The theorem concerns the positive same-direction product
\((f,g)\mapsto H_If\,H_Ig\) with finite output exponent \(q\).  It does
not cover the single integral \(\int fg\), Hardy--Steklov operators with
variable limits, kernel-weighted or mixed forward--backward products,
weak-type or Lorentz targets, the endpoint \(q=\infty\), compactness,
sharp equivalence constants, genuinely multilinear products, or
cancellation effects for signed kernels. Weighted weak-type bilinear Hardy
inequalities have been studied in \cite{GarciaGarciaOrtegaSalvador2023},
and the mixed orientation \(Hf\,H^*g\) has been treated for interior
exponents in \cite{MohantyJainJain2025}.
These are adjacent problems rather
than cases of the theorem proved here. Positivity immediately extends the
boundedness assertion from nonnegative inputs to real or complex inputs by
absolute values, but the argument does not exploit cancellation. Except in
the doubly infinite case and special examples, the optimal constant is
determined only up to factors depending on the exponents. The auxiliary
measure-valued theorem is needed only for \(1<P,Q<\infty\), precisely the
range created by the lower-triangle power lifting, and should be viewed
against the broader Hardy--Copson literature such as
\cite{GogatishviliPickUnver2022}.

Natural extensions of the present architecture include compactness criteria
obtained by supplementing the local characteristics with vanishing
conditions at the ends of an exhaustion. Further directions include
multilinear products of three or more Hardy factors and endpoint-safe
arbitrary-interval versions of mixed-orientation and weak-type problems
already present in the literature. Other possibilities are metric-measure
analogues beyond the current one-dimensional compact-restart setting
\cite{RuzhanskyShriwastawaVerma2024} and separate theories for
\(q=\infty\) and sharp constants. These problems require additional
arguments and are not consequences of the present boundedness theorem.

\subsection{Conclusion}
\label{subsec:conclusion}

The main structural conclusion is that one may treat
\(0<q<\infty\) and \(1\le p_1,p_2\le\infty\) on arbitrary intervals
within a common endpoint-safe architecture, provided that the problem is
formulated through directed compact restrictions.  On each compact
interval, the exact profiles and their Stieltjes completions encode the
input integration functionals and endpoint masses, the left-boundary atom
completes the \(p_i=1\) case, the
Hardy--Copson splitting allocates common atoms without loss or double
counting, and simultaneous regularisation supplies the operational
meaning of extended weights.

In the regular interior range these conditions agree with the classical
\(A_1,\ldots,A_7\) characterisation.  The same local architecture then
extends the formulation to endpoint exponents, non-attained interval
boundaries and degenerate weights while preserving the exact global
identity \(C_I=\sup_{J\Subset I}C_J\).

\appendix

\section{The complete finite-chain Hardy proof}
\label{app:finite-chain-hardy}

This appendix proves the finite-chain Hardy estimate used in
\HardyCopsonSectionReference. The result is uniform in the length of
the chain and in the underlying masses. It provides all four
implications required for the subsequent passage to finite Borel
measures.

\subsection{Notation and a power-increment estimate}
\label{subsec:appA-notation}

Let
\[
 1<P,Q<\infty,
 \qquad
 P'=\frac{P}{P-1}.
\]
Consider a finite ordered chain
\[
 x_1<x_2<\cdots<x_n
\]
carrying input and output masses
\[
 \mu_i>0,
 \qquad
 \nu_i\ge0,
 \qquad
 1\le i\le n.
\]
Zero input masses may be removed by the usual merging procedure. For a
noninitial zero-mass coordinate, merge the corresponding output mass with
the preceding output coordinate.
An initial zero-mass coordinate may be
discarded because its Hardy prefix is zero. Hence the assumption
\(\mu_i>0\) entails no loss of generality.

Define
\begin{equation}
 M_k:=\sum_{i=1}^k\mu_i,
 \qquad
 N_k:=\sum_{j=k}^n\nu_j,
 \qquad
 1\le k\le n,
 \label{eq:appA-cumulative-masses}
\end{equation}
with
\begin{equation}
 M_0:=0,
 \qquad
 N_{n+1}:=0,
 \label{eq:appA-boundary-cumulative-masses}
\end{equation}
and put
\begin{equation}
 V_k:=M_k^{1/P'},
 \qquad
 V_0:=0.
 \label{eq:appA-Vk}
\end{equation}

For a nonnegative sequence \(a=(a_i)_{i=1}^n\), let
\begin{equation}
 F_k:=(H_na)_k
 :=
 \sum_{i=1}^k a_i\mu_i,
 \qquad
 1\le k\le n.
 \label{eq:appA-finite-Hardy}
\end{equation}

We first record the elementary increment estimate used repeatedly in
the proof.

\begin{lemma}[Power-increment estimate]
\label{lem:appA-power-increment}
Let
\[
 0=s_0\le s_1\le\cdots\le s_m<\infty
\]
and let \(\gamma>0\). A summand with zero increment is understood to be
zero. In particular, when \(0<\gamma<1\), the formally indeterminate
product
\[
 0^{\gamma-1}(0-0)
\]
is assigned the value zero. Then
\begin{equation}
 c_\gamma s_m^\gamma
 \le
 \sum_{i=1}^m
 s_i^{\gamma-1}(s_i-s_{i-1})
 \le
 C_\gamma s_m^\gamma,
 \label{eq:appA-increasing-power-increment}
\end{equation}
where
\begin{equation}
 c_\gamma:=\min\{1,\gamma^{-1}\},
 \qquad
 C_\gamma:=\max\{1,\gamma^{-1}\}.
 \label{eq:appA-power-constants}
\end{equation}

If
\[
 t_1\ge t_2\ge\cdots\ge t_{m+1}=0,
\]
with the same zero-increment convention, then
\begin{equation}
 c_\gamma t_1^\gamma
 \le
 \sum_{i=1}^m
 t_i^{\gamma-1}(t_i-t_{i+1})
 \le
 C_\gamma t_1^\gamma.
 \label{eq:appA-decreasing-power-increment}
\end{equation}
\end{lemma}

\begin{proof}
Let \(0\le y<x\) and put \(u=y/x\). Then
\[
 \frac{x^{\gamma-1}(x-y)}
      {x^\gamma-y^\gamma}
 =
 \frac{1-u}{1-u^\gamma}.
\]
When \(0<\gamma\le1\), this ratio lies between \(1\) and
\(\gamma^{-1}\).
When \(\gamma\ge1\), it lies between
\(\gamma^{-1}\) and \(1\). Hence
\[
 c_\gamma(x^\gamma-y^\gamma)
 \le
 x^{\gamma-1}(x-y)
 \le
 C_\gamma(x^\gamma-y^\gamma).
\]
The same inequalities are immediate when \(x=y\), under the stated
zero-increment convention.

Apply this comparison with
\[
 x=s_i,
 \qquad
 y=s_{i-1},
\]
and sum over \(i\) to obtain
\eqref{eq:appA-increasing-power-increment}. Applying it instead with
\[
 x=t_i,
 \qquad
 y=t_{i+1},
\]
gives \eqref{eq:appA-decreasing-power-increment}.
\end{proof}

\subsection{Statement of the finite-chain theorem}
\label{subsec:appA-theorem-statement}

Let \(\mathcal H_n\) denote the optimal constant in
\begin{equation}
 \left(
   \sum_{k=1}^n F_k^Q\nu_k
 \right)^{1/Q}
 \le
 \mathcal H_n
 \left(
   \sum_{i=1}^n a_i^P\mu_i
 \right)^{1/P},
 \qquad
 a_i\ge0.
 \label{eq:appA-finite-chain-inequality}
\end{equation}

\begin{theorem}[Complete finite-chain Hardy theorem]
\label{thm:appA-finite-chain-hardy}
Let \(1<P,Q<\infty\).

If \(P\le Q\), define
\begin{equation}
 \mathfrak H_n
 :=
 \max_{1\le k\le n}
 V_kN_k^{1/Q}
 =
 \max_{1\le k\le n}
 M_k^{1/P'}N_k^{1/Q}.
 \label{eq:appA-convex-characteristic}
\end{equation}
Then
\begin{equation}
 \mathfrak H_n
 \le
 \mathcal H_n
 \le
 C_{P,Q}\mathfrak H_n.
 \label{eq:appA-convex-result}
\end{equation}

If \(Q<P\), define \(R\in(0,\infty)\) by
\begin{equation}
 \frac1R
 =
 \frac1Q-\frac1P
 \label{eq:appA-R-definition}
\end{equation}
and put
\begin{equation}
 \mathfrak H_n
 :=
 \left[
   \sum_{k=1}^n
   N_k^{R/Q}
   \left(
     V_k^R-V_{k-1}^R
   \right)
 \right]^{1/R}.
 \label{eq:appA-nonconvex-characteristic}
\end{equation}
Then
\begin{equation}
 c_{P,Q}\mathfrak H_n
 \le
 \mathcal H_n
 \le
 C_{P,Q}\mathfrak H_n.
 \label{eq:appA-nonconvex-result}
\end{equation}

The comparison constants depend only on \(P\) and \(Q\), and are
independent of \(n\), \((\mu_i)\), and \((\nu_i)\).
\end{theorem}

\begin{proof}
We establish the two implications in each exponent regime separately.

\medskip
\noindent\textbf{Step 1. A local weighted estimate.}

Let \(0<\varepsilon<1\). For every \(1\le k\le n\),
\begin{equation}
 F_k
 \le
 C_{P,\varepsilon}
 \left(
   \sum_{i=1}^k
   a_i^P V_i^{\varepsilon P}\mu_i
 \right)^{1/P}
 V_k^{1-\varepsilon}.
 \label{eq:appA-local-estimate}
\end{equation}
Indeed, Hölder's inequality gives
\[
 \begin{aligned}
 F_k
 &=
 \sum_{i=1}^k
 \left(
   a_iV_i^\varepsilon
 \right)
 V_i^{-\varepsilon}\mu_i
 \\
 &\le
 \left(
   \sum_{i=1}^k
   a_i^PV_i^{\varepsilon P}\mu_i
 \right)^{1/P}
 \left(
   \sum_{i=1}^k
   V_i^{-\varepsilon P'}\mu_i
 \right)^{1/P'}.
 \end{aligned}
\]
Since
\[
 V_i^{P'}=M_i
 \qquad\text{and}\qquad
 \mu_i=M_i-M_{i-1},
\]
Lemma~\ref{lem:appA-power-increment}, applied with exponent
\(1-\varepsilon\), yields
\[
 \begin{aligned}
 \sum_{i=1}^kV_i^{-\varepsilon P'}\mu_i
 &=
 \sum_{i=1}^k
 M_i^{-\varepsilon}(M_i-M_{i-1})
 \\
 &\lesssim_\varepsilon
 M_k^{1-\varepsilon}.
 \end{aligned}
\]
Taking the \(P'\)-th root proves
\eqref{eq:appA-local-estimate}.

\medskip
\noindent\textbf{Step 2. Necessity when \(P\le Q\).}

Fix \(1\le k\le n\) and choose
\[
 a_i=
 \begin{cases}
 1,&i\le k,\\
 0,&i>k.
 \end{cases}
\]
Then
\[
 F_j=M_k
 \qquad\text{for every }j\ge k.
\]
Consequently,
\[
 \left(
   \sum_{j=1}^nF_j^Q\nu_j
 \right)^{1/Q}
 \ge
 M_kN_k^{1/Q},
\]
whereas
\[
 \left(
   \sum_{i=1}^na_i^P\mu_i
 \right)^{1/P}
 =
 M_k^{1/P}.
\]
It follows that
\[
 \mathcal H_n
 \ge
 M_k^{1/P'}N_k^{1/Q}.
\]
Taking the maximum over \(k\) gives the lower estimate in
\eqref{eq:appA-convex-result}.

\medskip
\noindent\textbf{Step 3. Sufficiency when \(P\le Q\).}

Let
\[
 A:=\mathfrak H_n.
\]
If \(N_1=0\), then every \(\nu_k\) is zero and the operator is trivial.
Otherwise, let
\[
 n_0:=\max\{k:N_k>0\}.
\]
For \(k>n_0\), one has \(\nu_k=0\), and input coordinates with
\(i>n_0\) cannot affect any output coordinate carrying positive
\(\nu\)-mass. We may therefore replace \(n\) by \(n_0\) and assume
\begin{equation}
 N_k>0,
 \qquad
 1\le k\le n.
 \label{eq:appA-positive-output-tails}
\end{equation}

Fix, for example,
\[
 \varepsilon=\frac12.
\]
By \eqref{eq:appA-convex-characteristic},
\[
 V_i
 \le
 A N_i^{-1/Q}.
\]
Substitution into \eqref{eq:appA-local-estimate} gives
\begin{equation}
 F_k
 \lesssim_{P,Q}
 A
 \left(
   \sum_{i=1}^k
   a_i^P
   N_i^{-\varepsilon P/Q}
   \mu_i
 \right)^{1/P}
 N_k^{(\varepsilon-1)/Q}.
 \label{eq:appA-convex-local-reduction}
\end{equation}

Put
\begin{equation}
 s:=\frac QP\ge1,
 \qquad
 b_i:=a_i^P
 N_i^{-\varepsilon P/Q}\mu_i,
 \qquad
 w_k:=N_k^{\varepsilon-1}\nu_k.
 \label{eq:appA-convex-Minkowski-data}
\end{equation}
Minkowski's inequality on the finite measure space
\(\{1,\ldots,n\}\) gives
\begin{equation}
 \left[
   \sum_{k=1}^n
   \left(
     \sum_{i=1}^kb_i
   \right)^s
   w_k
 \right]^{1/s}
 \le
 \sum_{i=1}^n
 b_i
 \left(
   \sum_{k=i}^nw_k
 \right)^{1/s}.
 \label{eq:appA-discrete-Minkowski}
\end{equation}

Because
\[
 N_k-N_{k+1}=\nu_k,
\]
the decreasing form of
Lemma~\ref{lem:appA-power-increment}, applied with exponent
\(\varepsilon\), yields
\begin{equation}
 \sum_{k=i}^n
 N_k^{\varepsilon-1}\nu_k
 \lesssim_\varepsilon
 N_i^\varepsilon.
 \label{eq:appA-output-tail-increment}
\end{equation}
Since \(s^{-1}=P/Q\), the powers of \(N_i\) cancel, giving
\[
 \begin{aligned}
 b_i
 \left(
   \sum_{k=i}^nw_k
 \right)^{1/s}
 &\lesssim_{P,Q}
 a_i^P
 N_i^{-\varepsilon P/Q}
 \mu_i
 N_i^{\varepsilon P/Q}
 \\
 &=
 a_i^P\mu_i.
 \end{aligned}
\]
Therefore,
\begin{equation}
 \left[
   \sum_{k=1}^n
   \left(
     \sum_{i=1}^kb_i
   \right)^s
   w_k
 \right]^{1/s}
 \lesssim_{P,Q}
 \sum_{i=1}^na_i^P\mu_i.
 \label{eq:appA-Minkowski-conclusion}
\end{equation}

Raising \eqref{eq:appA-Minkowski-conclusion} to the power
\(s=Q/P\) and using
\eqref{eq:appA-convex-local-reduction}, we obtain
\[
 \sum_{k=1}^nF_k^Q\nu_k
 \lesssim_{P,Q}
 A^Q
 \left(
   \sum_{i=1}^na_i^P\mu_i
 \right)^{Q/P}.
\]
This proves the upper estimate in
\eqref{eq:appA-convex-result}.

\medskip
\noindent\textbf{Step 4. Necessity when \(Q<P\).}

Assume now that \(Q<P\), and let \(R\) be defined by
\eqref{eq:appA-R-definition}. Then
\begin{equation}
 \frac RQ
 =
 1+\frac RP
 >
 1.
 \label{eq:appA-R-exponent-identity}
\end{equation}
Set
\[
 \alpha:=\frac RQ
\]
and define
\begin{equation}
 \mathcal S^R
 :=
 \sum_{k=1}^n
 N_k^\alpha
 \left(
   V_k^R-V_{k-1}^R
 \right).
 \label{eq:appA-S-characteristic}
\end{equation}
Thus
\[
 \mathcal S=\mathfrak H_n.
\]

Finite Abel summation gives
\begin{equation}
 \mathcal S^R
 =
 \sum_{k=1}^n
 V_k^R
 \left(
   N_k^\alpha-N_{k+1}^\alpha
 \right).
 \label{eq:appA-Abel-necessity}
\end{equation}
Since
\[
 N_k-N_{k+1}=\nu_k
 \qquad\text{and}\qquad
 \alpha-1=\frac RP,
\]
Lemma~\ref{lem:appA-power-increment} implies
\begin{equation}
 \mathcal S^R
 \asymp_{P,Q}
 \mathcal T,
 \qquad
 \mathcal T
 :=
 \sum_{k=1}^n
 V_k^R N_k^{R/P}\nu_k.
 \label{eq:appA-T-equivalence}
\end{equation}

If \(\mathcal T=0\), the desired estimate is immediate. Suppose
\(\mathcal T>0\). Define
\begin{equation}
 \theta:=\frac{R}{2P'},
 \qquad
 \lambda:=R-\theta P'=\frac R2>0,
 \label{eq:appA-theta-lambda}
\end{equation}
and put
\begin{equation}
 D_i
 :=
 \sum_{k=i}^n
 N_k^{R/P}\nu_kV_k^\lambda,
 \qquad
 1\le i\le n,
 \label{eq:appA-Di}
\end{equation}
together with the test sequence
\begin{equation}
 a_i
 :=
 D_i^{1/P}
 V_i^{(\theta-1)(P'-1)}.
 \label{eq:appA-test-sequence}
\end{equation}

Since
\[
 P(P'-1)=P',
\]
finite Fubini gives
\[
 \begin{aligned}
 \sum_{i=1}^na_i^P\mu_i
 &=
 \sum_{i=1}^n
 D_i
 V_i^{(\theta-1)P'}\mu_i
 \\
 &=
 \sum_{k=1}^n
 N_k^{R/P}\nu_kV_k^\lambda
 \sum_{i=1}^k
 M_i^{\theta-1}\mu_i.
 \end{aligned}
\]
By Lemma~\ref{lem:appA-power-increment}, applied with exponent
\(\theta\),
\[
 \sum_{i=1}^k
 M_i^{\theta-1}\mu_i
 \lesssim_{P,Q}
 M_k^\theta
 =
 V_k^{\theta P'}.
\]
Because
\[
 \lambda+\theta P'=R,
\]
we obtain
\begin{equation}
 \sum_{i=1}^na_i^P\mu_i
 \lesssim_{P,Q}
 \sum_{k=1}^n
 N_k^{R/P}\nu_kV_k^R
 =
 \mathcal T.
 \label{eq:appA-test-input-bound}
\end{equation}

For \(i\le k\), the sequence \((D_i)\) is decreasing, so
\[
 D_i\ge D_k.
\]
Consequently,
\[
 \begin{aligned}
 F_k
 &=
 \sum_{i=1}^k
 D_i^{1/P}
 V_i^{(\theta-1)(P'-1)}
 \mu_i
 \\
 &\ge
 D_k^{1/P}
 \sum_{i=1}^k
 M_i^{(\theta-1)/P}\mu_i.
 \end{aligned}
\]
Another application of
Lemma~\ref{lem:appA-power-increment} gives
\begin{equation}
 F_k
 \gtrsim_{P,Q}
 D_k^{1/P}
 V_k^{1+\theta(P'-1)}.
 \label{eq:appA-test-output-first-bound}
\end{equation}

Moreover,
\[
 \begin{aligned}
 D_k
 &\ge
 V_k^\lambda
 \sum_{j=k}^n
 N_j^{R/P}\nu_j
 \\
 &\asymp_{P,Q}
 V_k^\lambda N_k^{R/Q},
 \end{aligned}
\]
where the final comparison follows from the decreasing form of
Lemma~\ref{lem:appA-power-increment}, using
\[
 \frac RQ=1+\frac RP.
\]
Hence
\begin{equation}
 D_k
 \gtrsim_{P,Q}
 V_k^\lambda N_k^{R/Q}.
 \label{eq:appA-Dk-lower-bound}
\end{equation}

Combining
\eqref{eq:appA-test-output-first-bound} and
\eqref{eq:appA-Dk-lower-bound}, and using
\begin{equation}
 Q\left(
   1+\theta(P'-1)+\frac{\lambda}{P}
 \right)
 =
 R,
 \label{eq:appA-test-exponent-identity}
\end{equation}
we obtain
\[
 F_k^Q
 \gtrsim_{P,Q}
 V_k^R N_k^{R/P}.
\]
Therefore,
\begin{equation}
 \sum_{k=1}^nF_k^Q\nu_k
 \gtrsim_{P,Q}
 \mathcal T.
 \label{eq:appA-test-output-bound}
\end{equation}

Apply \eqref{eq:appA-finite-chain-inequality} to the test sequence
\eqref{eq:appA-test-sequence}. Equations
\eqref{eq:appA-test-input-bound} and
\eqref{eq:appA-test-output-bound} give
\[
 \mathcal T^{1/Q}
 \lesssim_{P,Q}
 \mathcal H_n\mathcal T^{1/P}.
\]
Since
\[
 \frac1R=\frac1Q-\frac1P,
\]
it follows that
\[
 \mathcal T^{1/R}
 \lesssim_{P,Q}
 \mathcal H_n.
\]
Finally, \eqref{eq:appA-T-equivalence} yields
\[
 \mathfrak H_n
 =
 \mathcal S
 \lesssim_{P,Q}
 \mathcal H_n,
\]
which proves the lower estimate in
\eqref{eq:appA-nonconvex-result}.

\medskip
\noindent\textbf{Step 5. Sufficiency when \(Q<P\).}

Let
\[
 A:=\mathfrak H_n.
\]
Fix
\begin{equation}
 \delta:=1,
 \qquad
 \varepsilon:=\frac12,
 \qquad
 \beta:=(1-\varepsilon+\delta)Q=\frac{3Q}{2},
 \label{eq:appA-step5-parameters}
\end{equation}
and define
\begin{equation}
 X_k
 :=
 \sum_{i=1}^k
 a_i^PV_i^{P/2}\mu_i.
 \label{eq:appA-Xk}
\end{equation}
The local estimate
\eqref{eq:appA-local-estimate} gives
\[
 F_k^Q
 \lesssim_P
 X_k^{Q/P}V_k^{Q/2}.
\]
Since
\[
 \frac Q2
 =
 -\delta Q+\beta,
\]
we have
\begin{equation}
 \sum_{k=1}^nF_k^Q\nu_k
 \lesssim_P
 \sum_{k=1}^n
 X_k^{Q/P}
 V_k^{-\delta Q}
 V_k^\beta\nu_k.
 \label{eq:appA-step5-start}
\end{equation}

Because \(Q/P<1\),
\[
 V_k^{-\delta P}
 =
 \left(
   V_k^{-\delta P}-V_n^{-\delta P}
 \right)
 +
 V_n^{-\delta P}
\]
and concavity of \(t\mapsto t^{Q/P}\) imply
\begin{equation}
 \sum_{k=1}^nF_k^Q\nu_k
 \lesssim_{P,Q}
 I+II,
 \label{eq:appA-I-II-splitting}
\end{equation}
where
\begin{equation}
 I
 :=
 \sum_{k=1}^n
 X_k^{Q/P}
 \left(
   V_k^{-\delta P}-V_n^{-\delta P}
 \right)^{Q/P}
 V_k^\beta\nu_k
 \label{eq:appA-I-definition}
\end{equation}
and
\begin{equation}
 II
 :=
 V_n^{-\delta Q}
 \sum_{k=1}^n
 X_k^{Q/P}V_k^\beta\nu_k.
 \label{eq:appA-II-definition}
\end{equation}

For \(1\le\tau<n\), set
\begin{equation}
 d_\tau
 :=
 V_\tau^{-\delta P}
 -
 V_{\tau+1}^{-\delta P}
 \ge0,
 \label{eq:appA-dtau}
\end{equation}
and, for \(1\le s\le n\), put
\begin{equation}
 \Delta_s
 :=
 V_s^R-V_{s-1}^R.
 \label{eq:appA-Delta}
\end{equation}
Applying Lemma~\ref{lem:appA-power-increment} to the increasing
sequence \((V_s^R)\) gives
\begin{equation}
 V_k^\beta
 \asymp_{P,Q}
 \sum_{s=1}^k
 V_s^{\beta-R}\Delta_s.
 \label{eq:appA-beta-increment}
\end{equation}

Furthermore,
\[
 V_k^{-\delta P}-V_n^{-\delta P}
 =
 \sum_{\tau=k}^{n-1}d_\tau.
\]
Since \(X_k\) is nondecreasing,
\begin{equation}
 X_k
 \sum_{\tau=k}^{n-1}d_\tau
 \le
 \sum_{\tau=k}^{n-1}X_\tau d_\tau
 =:
 E_k.
 \label{eq:appA-Ek-definition}
\end{equation}
The sequence \((E_k)\) is nonincreasing. By
\eqref{eq:appA-beta-increment}, finite Fubini,
\eqref{eq:appA-Ek-definition}, and
\[
 \sum_{k=s}^n\nu_k=N_s,
\]
we obtain
\begin{equation}
 I
 \lesssim_{P,Q}
 \sum_{s=1}^n
 E_s^{Q/P}
 V_s^{\beta-R}
 N_s\Delta_s.
 \label{eq:appA-I-before-Holder}
\end{equation}

Apply Hölder's inequality with conjugate exponents
\[
 \frac RQ
 \qquad\text{and}\qquad
 \frac PQ.
\]
Using the definition of \(A\), we obtain
\begin{equation}
 I
 \lesssim_{P,Q}
 A^QJ^{Q/P},
 \label{eq:appA-I-Holder}
\end{equation}
where
\begin{equation}
 J
 :=
 \sum_{s=1}^n
 E_s
 V_s^{(\beta-R)P/Q}
 \Delta_s.
 \label{eq:appA-J-definition}
\end{equation}

Finite Fubini gives
\begin{equation}
 J
 =
 \sum_{\tau=1}^{n-1}
 X_\tau d_\tau
 \sum_{s=1}^\tau
 V_s^{(\beta-R)P/Q}\Delta_s.
 \label{eq:appA-J-Fubini}
\end{equation}
The exponent identity
\begin{equation}
 R+\frac PQ(\beta-R)
 =
 (\delta-\varepsilon)P
 =
 \frac P2
 >0
 \label{eq:appA-step5-exponent-identity}
\end{equation}
and Lemma~\ref{lem:appA-power-increment} imply
\begin{equation}
 \sum_{s=1}^\tau
 V_s^{(\beta-R)P/Q}\Delta_s
 \lesssim_{P,Q}
 V_\tau^{P/2}.
 \label{eq:appA-inner-J-bound}
\end{equation}

Substitute \eqref{eq:appA-Xk} into
\eqref{eq:appA-J-Fubini}, use
\eqref{eq:appA-inner-J-bound}, and reverse the finite sums. This gives
\begin{equation}
 \begin{aligned}
 J
 &\lesssim_{P,Q}
 \sum_{i=1}^n
 a_i^PV_i^{P/2}\mu_i
 \\
 &\qquad\times
 \sum_{\tau=i}^{n-1}
 V_\tau^{P/2}
 \left(
   V_\tau^{-P}-V_{\tau+1}^{-P}
 \right).
 \end{aligned}
 \label{eq:appA-J-reversed}
\end{equation}

For \(i\le\tau\le n\), set
\[
 t_\tau:=V_\tau^{-P}
\]
and append the terminal value \(t_{n+1}:=0\). Then
\[
 V_\tau^{P/2}
 \left(
   V_\tau^{-P}-V_{\tau+1}^{-P}
 \right)
 =
 t_\tau^{-1/2}(t_\tau-t_{\tau+1}).
\]
Adding the nonnegative terminal increment and applying the decreasing
form of Lemma~\ref{lem:appA-power-increment} with exponent \(1/2\),
we obtain
\[
 \begin{aligned}
 &\sum_{\tau=i}^{n-1}
 V_\tau^{P/2}
 \left(
   V_\tau^{-P}-V_{\tau+1}^{-P}
 \right)
 \\
 &\quad\le
 \sum_{\tau=i}^{n}
 t_\tau^{-1/2}(t_\tau-t_{\tau+1})
 \lesssim
 t_i^{1/2}
 =
 V_i^{-P/2}.
 \end{aligned}
\]
Consequently,
\begin{equation}
 J
 \lesssim_{P,Q}
 \sum_{i=1}^na_i^P\mu_i.
 \label{eq:appA-J-final-bound}
\end{equation}
Combining
\eqref{eq:appA-I-Holder} and
\eqref{eq:appA-J-final-bound}, we obtain
\begin{equation}
 I
 \lesssim_{P,Q}
 A^Q
 \left(
   \sum_{i=1}^na_i^P\mu_i
 \right)^{Q/P}.
 \label{eq:appA-I-final-bound}
\end{equation}

It remains to estimate \(II\). Since \(V_i\le V_n\),
\[
 X_k
 \le
 X_n
 \le
 V_n^{P/2}
 \sum_{i=1}^na_i^P\mu_i.
\]
Hence
\begin{equation}
 II
 \le
 V_n^{-Q/2}
 \left(
   \sum_{i=1}^na_i^P\mu_i
 \right)^{Q/P}
 \sum_{k=1}^nV_k^\beta\nu_k.
 \label{eq:appA-II-first-bound}
\end{equation}

Finite Abel summation yields
\[
 \sum_{k=1}^nV_k^\beta\nu_k
 =
 \sum_{k=1}^n
 N_k
 \left(
   V_k^\beta-V_{k-1}^\beta
 \right).
\]
Using Lemma~\ref{lem:appA-power-increment},
\begin{equation}
 \sum_{k=1}^nV_k^\beta\nu_k
 \lesssim_{P,Q}
 \sum_{k=1}^n
 N_kV_k^{\beta-R}\Delta_k.
 \label{eq:appA-II-Abel}
\end{equation}
Apply Hölder's inequality as in
\eqref{eq:appA-I-Holder}, now without the factor \(E_k\). Using
\eqref{eq:appA-inner-J-bound} with \(\tau=n\), we obtain
\begin{equation}
 \sum_{k=1}^nV_k^\beta\nu_k
 \lesssim_{P,Q}
 A^QV_n^{Q/2}.
 \label{eq:appA-beta-output-bound}
\end{equation}
Equations
\eqref{eq:appA-II-first-bound} and
\eqref{eq:appA-beta-output-bound} imply
\begin{equation}
 II
 \lesssim_{P,Q}
 A^Q
 \left(
   \sum_{i=1}^na_i^P\mu_i
 \right)^{Q/P}.
 \label{eq:appA-II-final-bound}
\end{equation}

Finally,
\eqref{eq:appA-I-II-splitting},
\eqref{eq:appA-I-final-bound}, and
\eqref{eq:appA-II-final-bound} give
\[
 \sum_{k=1}^nF_k^Q\nu_k
 \lesssim_{P,Q}
 A^Q
 \left(
   \sum_{i=1}^na_i^P\mu_i
 \right)^{Q/P}.
\]
This proves the upper estimate in
\eqref{eq:appA-nonconvex-result} and completes the proof.
\end{proof}

\begin{remark}[Uniformity of the finite-chain theorem]
\label{rem:appA-uniformity}
Every comparison constant in
Theorem~\ref{thm:appA-finite-chain-hardy} depends only on \(P\) and
\(Q\). In particular, the constants are uniform over all ordered
partitions and all chain lengths. This uniformity is what permits the
partition limit in Appendix~\ref{app:partition-approximation}.
\end{remark}

\begin{remark}[Logical role]
\label{rem:appA-logical-role}
Theorem~\ref{thm:appA-finite-chain-hardy} contains both necessity and
sufficiency in both exponent regimes. The strict Copson theorem in
Appendix~\ref{app:copson-reversal} is obtained from this result by an
exact weighted substitution and reversal of the shortened chain. No
additional discrete Hardy theorem is invoked.
\end{remark}

\section{Exact reversal for the strict Copson operator}
\label{app:copson-reversal}

This appendix derives the finite-chain strict Copson theorem from
Theorem~\ref{thm:appA-finite-chain-hardy}. The reduction is exact. A
weighted isometric substitution removes the Copson weight, the inactive
input and output endpoints are discarded, and the remaining chain is
reversed. No additional discrete Hardy inequality is required.

Throughout this appendix,
\[
 1<P,Q<\infty,
 \qquad
 P'=\frac{P}{P-1}.
\]
Let
\[
 \mu_i>0,
 \qquad
 \nu_i\ge0,
 \qquad
 0\le\omega_i<\infty,
 \qquad
 1\le i\le n.
\]
The treatment of unbounded Copson weights is deferred to the monotone
truncation argument in Appendix~\ref{app:extended-regularisation}.

\subsection{The strict finite-chain operator}
\label{subsec:appB-strict-operator}

For a nonnegative sequence \(a=(a_i)_{i=1}^n\), define
\begin{equation}
 (C_{\omega,n}a)_k
 :=
 \sum_{i=k+1}^n
 \omega_i a_i\mu_i,
 \qquad
 1\le k\le n.
 \label{eq:appB-strict-Copson}
\end{equation}
The strict inequality \(i>k\) is part of the definition. In particular,
\[
 (C_{\omega,n}a)_n=0,
\]
and the first input coordinate \(a_1\) does not occur in any output
coordinate.

Let \(\mathcal C_n\) denote the optimal constant in
\begin{equation}
 \left[
   \sum_{k=1}^n
   (C_{\omega,n}a)_k^Q\nu_k
 \right]^{1/Q}
 \le
 \mathcal C_n
 \left[
   \sum_{i=1}^n
   a_i^P\mu_i
 \right]^{1/P},
 \qquad a_i\ge0.
 \label{eq:appB-strict-Copson-inequality}
\end{equation}

Define the cumulative output heads
\begin{equation}
 K_k
 :=
 \sum_{j=1}^k\nu_j,
 \qquad
 0\le k\le n,
 \label{eq:appB-Kk}
\end{equation}
where \(K_0:=0\), and the weighted strict input tails
\begin{equation}
 B_k
 :=
 \sum_{i=k+1}^n
 \omega_i^{P'}\mu_i,
 \qquad
 1\le k\le n.
 \label{eq:appB-Bk}
\end{equation}
Thus
\[
 B_n=0.
\]

\subsection{Removal of the Copson weight}
\label{subsec:appB-weight-removal}

The first step converts the weighted strict Copson operator into an
unweighted strict tail operator with a modified input measure.

\begin{lemma}[Weighted isometric substitution]
\label{lem:appB-weighted-isometry}
For \(1\le i\le n\), put
\begin{equation}
 \lambda_i
 :=
 \omega_i^{P'}\mu_i.
 \label{eq:appB-lambda}
\end{equation}
In computing the optimal constant \(\mathcal C_n\), it is sufficient to
consider sequences satisfying
\begin{equation}
 a_1=0
 \qquad\text{and}\qquad
 a_i=0
 \quad\text{whenever }\omega_i=0.
 \label{eq:appB-active-sequences}
\end{equation}
For such a sequence, define
\begin{equation}
 b_i
 :=
 \begin{cases}
 a_i\omega_i^{-1/(P-1)},&\omega_i>0,\\
 0,&\omega_i=0.
 \end{cases}
 \label{eq:appB-bi-substitution}
\end{equation}
Then
\begin{equation}
 a_i^P\mu_i
 =
 b_i^P\lambda_i
 \label{eq:appB-input-isometry}
\end{equation}
and
\begin{equation}
 \omega_i a_i\mu_i
 =
 b_i\lambda_i.
 \label{eq:appB-output-isometry}
\end{equation}
Consequently,
\begin{equation}
 (C_{\omega,n}a)_k
 =
 \sum_{i=k+1}^n b_i\lambda_i.
 \label{eq:appB-unweighted-strict-tail}
\end{equation}
\end{lemma}

\begin{proof}
The coordinate \(a_1\) does not contribute to
\(C_{\omega,n}a\). Likewise, when \(\omega_i=0\), the coordinate \(a_i\)
contributes nothing to the output. Replacing all such coordinates by
zero leaves the output unchanged and does not increase the input norm.
Thus the supremum defining \(\mathcal C_n\) may be restricted to
\eqref{eq:appB-active-sequences}.

For \(\omega_i>0\), the definition of \(b_i\) gives
\[
 a_i=b_i\omega_i^{1/(P-1)}.
\]
Since
\[
 \frac{P}{P-1}=P',
\]
we obtain
\[
 a_i^P\mu_i
 =
 b_i^P\omega_i^{P'}\mu_i
 =
 b_i^P\lambda_i,
\]
and
\[
 \omega_i a_i\mu_i
 =
 b_i\omega_i^{1+1/(P-1)}\mu_i
 =
 b_i\omega_i^{P'}\mu_i
 =
 b_i\lambda_i.
\]
Both identities are also valid on the inactive coordinates under
\eqref{eq:appB-active-sequences}. Substitution into
\eqref{eq:appB-strict-Copson} proves
\eqref{eq:appB-unweighted-strict-tail}.
\end{proof}

\begin{remark}
\label{rem:appB-zero-lambda}
Some of the transformed masses \(\lambda_i\) may vanish. Such
coordinates may be suppressed as in
Subsection~\ref{subsec:appA-notation}, or retained with zero mass. The
finite-chain Hardy theorem is unchanged under either convention.
\end{remark}

\subsection{Exact reversal of the shortened chain}
\label{subsec:appB-exact-reversal}

The strict operator has two inactive endpoint coordinates.

\[
 a_1
 \quad\text{is never used},
 \qquad
 (C_{\omega,n}a)_n=0.
\]
The active problem therefore has \(n-1\) input and \(n-1\) output
coordinates.

\begin{lemma}[Exact strict-Copson reversal]
\label{lem:appB-exact-reversal}
Assume \(n\ge2\), and set
\[
 m:=n-1.
\]
For \(1\le j,\ell\le m\), define
\begin{equation}
 \widehat\mu_j
 :=
 \lambda_{n-j+1},
 \qquad
 \widehat b_j
 :=
 b_{n-j+1},
 \qquad
 \widehat\nu_\ell
 :=
 \nu_{n-\ell}.
 \label{eq:appB-reversed-data}
\end{equation}
Let
\begin{equation}
 (\widehat H_m\widehat b)_\ell
 :=
 \sum_{j=1}^{\ell}
 \widehat b_j\widehat\mu_j.
 \label{eq:appB-reversed-Hardy}
\end{equation}
Then, for every \(1\le\ell\le m\),
\begin{equation}
 (\widehat H_m\widehat b)_\ell
 =
 (C_{\omega,n}a)_{n-\ell}.
 \label{eq:appB-pointwise-reversal}
\end{equation}
Moreover,
\begin{equation}
 \sum_{\ell=1}^{m}
 (\widehat H_m\widehat b)_\ell^Q
 \widehat\nu_\ell
 =
 \sum_{k=1}^{n}
 (C_{\omega,n}a)_k^Q\nu_k
 \label{eq:appB-output-norm-reversal}
\end{equation}
and
\begin{equation}
 \sum_{j=1}^{m}
 \widehat b_j^P\widehat\mu_j
 =
 \sum_{i=1}^{n}
 a_i^P\mu_i
 \label{eq:appB-input-norm-reversal}
\end{equation}
for sequences satisfying \eqref{eq:appB-active-sequences}.

Consequently,
\begin{equation}
 \mathcal C_n
 =
 \widehat{\mathcal H}_{m},
 \label{eq:appB-exact-norm-equality}
\end{equation}
where \(\widehat{\mathcal H}_{m}\) is the optimal Hardy constant for the
reversed chain
\[
 \widehat H_m:
 \ell^P(\widehat\mu)
 \longrightarrow
 \ell^Q(\widehat\nu).
\]
\end{lemma}

\begin{proof}
For \(1\le\ell\le m\),
\[
 \begin{aligned}
 (\widehat H_m\widehat b)_\ell
 &=
 \sum_{j=1}^{\ell}
 b_{n-j+1}\lambda_{n-j+1}
 \\
 &=
 \sum_{i=n-\ell+1}^{n}
 b_i\lambda_i
 \\
 &=
 (C_{\omega,n}a)_{n-\ell},
 \end{aligned}
\]
which proves \eqref{eq:appB-pointwise-reversal}.

Using the change of index
\[
 k=n-\ell,
\]
we obtain
\[
 \begin{aligned}
 \sum_{\ell=1}^{m}
 (\widehat H_m\widehat b)_\ell^Q
 \widehat\nu_\ell
 &=
 \sum_{\ell=1}^{n-1}
 (C_{\omega,n}a)_{n-\ell}^Q
 \nu_{n-\ell}
 \\
 &=
 \sum_{k=1}^{n-1}
 (C_{\omega,n}a)_k^Q\nu_k.
 \end{aligned}
\]
Since \((C_{\omega,n}a)_n=0\), this is
\eqref{eq:appB-output-norm-reversal}.

Similarly,
\[
 \begin{aligned}
 \sum_{j=1}^{m}
 \widehat b_j^P\widehat\mu_j
 &=
 \sum_{j=1}^{n-1}
 b_{n-j+1}^P\lambda_{n-j+1}
 \\
 &=
 \sum_{i=2}^{n}b_i^P\lambda_i
 \\
 &=
 \sum_{i=1}^{n}a_i^P\mu_i,
 \end{aligned}
\]
where Lemma~\ref{lem:appB-weighted-isometry} and \(a_1=0\) were used.

Thus the two operator quotients agree exactly. The substitutions and
index reversal are invertible on the active coordinates, so taking the
two suprema gives \eqref{eq:appB-exact-norm-equality}.
\end{proof}

\subsection{Transformation of the cumulative quantities}
\label{subsec:appB-cumulative-reversal}

For the reversed Hardy chain, define
\begin{equation}
 \widehat M_\ell
 :=
 \sum_{j=1}^{\ell}\widehat\mu_j,
 \qquad
 \widehat N_\ell
 :=
 \sum_{s=\ell}^{m}\widehat\nu_s,
 \qquad
 1\le\ell\le m.
 \label{eq:appB-reversed-cumulatives}
\end{equation}

\begin{lemma}[Reversal of heads and tails]
\label{lem:appB-cumulative-reversal}
For \(1\le\ell\le n-1\),
\begin{equation}
 \widehat M_\ell
 =
 B_{n-\ell}
 \label{eq:appB-Mhat-B}
\end{equation}
and
\begin{equation}
 \widehat N_\ell
 =
 K_{n-\ell}.
 \label{eq:appB-Nhat-K}
\end{equation}
\end{lemma}

\begin{proof}
By \eqref{eq:appB-reversed-data},
\[
 \begin{aligned}
 \widehat M_\ell
 &=
 \sum_{j=1}^{\ell}
 \lambda_{n-j+1}
 \\
 &=
 \sum_{i=n-\ell+1}^{n}\lambda_i
 \\
 &=
 B_{n-\ell}.
 \end{aligned}
\]
Likewise,
\[
 \begin{aligned}
 \widehat N_\ell
 &=
 \sum_{s=\ell}^{n-1}\nu_{n-s}
 \\
 &=
 \sum_{k=1}^{n-\ell}\nu_k
 \\
 &=
 K_{n-\ell}.
 \end{aligned}
\]
\end{proof}

\subsection{The strict Copson theorem}
\label{subsec:appB-Copson-theorem}

\begin{theorem}[Finite-chain strict Copson theorem]
\label{thm:appB-finite-chain-Copson}
Let \(\mathcal C_n\) be the optimal constant in
\eqref{eq:appB-strict-Copson-inequality}.

If \(P\le Q\), define
\begin{equation}
 \mathfrak C_n
 :=
 \max_{1\le k\le n}
 K_k^{1/Q}B_k^{1/P'}.
 \label{eq:appB-convex-Copson-characteristic}
\end{equation}
Then
\begin{equation}
 \mathfrak C_n
 \le
 \mathcal C_n
 \le
 C_{P,Q}\mathfrak C_n.
 \label{eq:appB-convex-Copson-result}
\end{equation}

If \(Q<P\), define \(R\in(0,\infty)\) by
\begin{equation}
 \frac1R
 =
 \frac1Q-\frac1P
 \label{eq:appB-Copson-R}
\end{equation}
and put
\begin{equation}
 \mathfrak C_n
 :=
 \left[
   \sum_{k=1}^{n}
   B_k^{R/P'}
   \left(
     K_k^{R/Q}-K_{k-1}^{R/Q}
   \right)
 \right]^{1/R}.
 \label{eq:appB-nonconvex-Copson-characteristic}
\end{equation}
Then
\begin{equation}
 c_{P,Q}\mathfrak C_n
 \le
 \mathcal C_n
 \le
 C_{P,Q}\mathfrak C_n.
 \label{eq:appB-nonconvex-Copson-result}
\end{equation}

The constants are independent of \(n\), the masses, and the Copson
weight.
\end{theorem}

\begin{proof}
If \(n=1\), then \(C_{\omega,1}=0\), \(B_1=0\), and all asserted
quantities vanish. Assume henceforth that \(n\ge2\).

By Lemma~\ref{lem:appB-exact-reversal},
\[
 \mathcal C_n=\widehat{\mathcal H}_{n-1}.
\]

\medskip
\noindent\textbf{The case \(P\le Q\).}

Theorem~\ref{thm:appA-finite-chain-hardy} and
Lemma~\ref{lem:appB-cumulative-reversal} give
\[
 \begin{aligned}
 \widehat{\mathcal H}_{n-1}
 &\asymp_{P,Q}
 \max_{1\le\ell\le n-1}
 \widehat M_\ell^{1/P'}
 \widehat N_\ell^{1/Q}
 \\
 &=
 \max_{1\le\ell\le n-1}
 B_{n-\ell}^{1/P'}
 K_{n-\ell}^{1/Q}
 \\
 &=
 \max_{1\le k\le n-1}
 K_k^{1/Q}B_k^{1/P'}.
 \end{aligned}
\]
Since \(B_n=0\), adjoining the index \(k=n\) does not change the
maximum. This proves \eqref{eq:appB-convex-Copson-result}.

\medskip
\noindent\textbf{The case \(Q<P\).}

The non-convex branch of
Theorem~\ref{thm:appA-finite-chain-hardy} gives
\begin{equation}
 \begin{aligned}
 \widehat{\mathcal H}_{n-1}^R
 &\asymp_{P,Q}
 \sum_{\ell=1}^{n-1}
 \widehat N_\ell^{R/Q}
 \left(
   \widehat M_\ell^{R/P'}
   -
   \widehat M_{\ell-1}^{R/P'}
 \right).
 \end{aligned}
 \label{eq:appB-reversed-Hardy-characteristic}
\end{equation}
where \(\widehat M_0:=0\).

Under the substitution
\[
 k=n-\ell,
\]
one has
\[
 \widehat M_\ell=B_k,
 \qquad
 \widehat M_{\ell-1}=B_{k+1},
 \qquad
 \widehat N_\ell=K_k.
\]
Therefore,
\begin{equation}
 \widehat{\mathcal H}_{n-1}^R
 \asymp_{P,Q}
 \sum_{k=1}^{n-1}
 K_k^{R/Q}
 \left(
   B_k^{R/P'}-B_{k+1}^{R/P'}
 \right).
 \label{eq:appB-reversed-difference-form}
\end{equation}

Set
\[
 X_k:=K_k^{R/Q},
 \qquad
 Y_k:=B_k^{R/P'}.
\]
Since
\[
 X_0=0,
 \qquad
 Y_n=0,
\]
finite Abel summation gives the exact identity
\begin{equation}
 \sum_{k=1}^{n-1}
 X_k(Y_k-Y_{k+1})
 =
 \sum_{k=1}^{n}
 Y_k(X_k-X_{k-1}).
 \label{eq:appB-Abel-identity}
\end{equation}
Hence
\begin{equation}
 \begin{aligned}
 &\sum_{k=1}^{n-1}
 K_k^{R/Q}
 \left(
   B_k^{R/P'}-B_{k+1}^{R/P'}
 \right)
 \\
 &\qquad=
 \sum_{k=1}^{n}
 B_k^{R/P'}
 \left(
   K_k^{R/Q}-K_{k-1}^{R/Q}
 \right).
 \end{aligned}
 \label{eq:appB-two-Copson-forms}
\end{equation}
Combining
\eqref{eq:appB-reversed-difference-form} and
\eqref{eq:appB-two-Copson-forms} proves
\eqref{eq:appB-nonconvex-Copson-result}.
\end{proof}

\begin{remark}[Why the differential is generated by \(K\)]
\label{rem:appB-K-differential}
The reversal first produces the equivalent discrete expression
\[
 \sum_{k=1}^{n-1}
 K_k^{R/Q}
 \left(
   B_k^{R/P'}-B_{k+1}^{R/P'}
 \right).
\]
The exact Abel identity
\eqref{eq:appB-two-Copson-forms} converts it into
\[
 \sum_{k=1}^{n}
 B_k^{R/P'}
 \left(
   K_k^{R/Q}-K_{k-1}^{R/Q}
 \right).
\]
It is the second form that passes directly to the measure-valued
criterion
\[
 \int B_\omega(x)^{R/P'}\,
 \mathrm d\!\left(K(x)^{R/Q}\right).
\]
Thus the Stieltjes differential in the strict Copson theorem is
generated by \(K^{R/Q}\), not by the strict tail \(B_\omega\).
\end{remark}

\begin{remark}[Strictness and endpoint allocation]
\label{rem:appB-strictness}
The quantity
\[
 B_k=\sum_{i=k+1}^n\omega_i^{P'}\mu_i
\]
excludes the diagonal input mass at \(i=k\). By contrast,
\[
 K_k=\sum_{j=1}^k\nu_j
\]
includes the output mass at \(j=k\). This closed-head/strict-tail
geometry is preserved exactly by the chain reversal and is the discrete
counterpart of the convention
\[
 [c,x]
 \quad\text{for the Hardy part},
 \qquad
 (x,d]
 \quad\text{for the Copson part}.
\]
\end{remark}

\begin{remark}[Uniformity]
\label{rem:appB-uniformity}
The comparison constants in
Theorem~\ref{thm:appB-finite-chain-Copson} are precisely those inherited
from Theorem~\ref{thm:appA-finite-chain-hardy}. They are therefore
uniform over the ordered partitions used in
Appendix~\ref{app:partition-approximation}.
\end{remark}
\section{Partition approximation for finite Borel measures}
\label{app:partition-approximation}

This appendix passes from the finite-chain results of
Appendices~\ref{app:finite-chain-hardy} and
\ref{app:copson-reversal} to Hardy and strict Copson operators defined
by finite Borel measures on a compact interval.

The approximation is designed to preserve the endpoint convention at
atoms. Large atoms are isolated as singleton cells, while every
non-singleton cell has asymptotically small input and output mass. The
resulting block operators differ from the continuous operators only
through an intrablock term whose norm tends to zero.

Throughout this appendix, let
\[
 J=[c,d],
 \qquad
 1<P,Q<\infty,
\]
and let \(\sigma\) and \(\nu\) be finite nonnegative Borel measures on
\(J\). We write
\[
 P'=\frac{P}{P-1}.
\]
When \(Q<P\), \(R\in(0,\infty)\) is defined by
\begin{equation}
 \frac1R=\frac1Q-\frac1P.
 \label{eq:appC-R}
\end{equation}

\subsection{Ordered atom-preserving partitions}
\label{subsec:appC-ordered-partitions}

\begin{definition}[Ordered interval partition]
\label{def:appC-ordered-partition}
An ordered interval partition of \(J\) is a finite Borel partition
\[
 \mathcal P=\{E_1,\ldots,E_N\}
\]
such that
\[
 x<y
 \qquad
 \text{whenever }
 x\in E_i,\ y\in E_j,\ i<j,
\]
and every prefix
\[
 A_k:=\bigcup_{i=1}^kE_i
\]
is an initial interval of one of the forms
\[
 [c,t_k],
 \qquad
 [c,t_k),
 \qquad
 \text{or }J.
\]
The cells may therefore be ordinary half-open intervals or
singletons.
\end{definition}

For such a partition, define
\begin{equation}
 \varepsilon(\mathcal P)
 :=
 \max_{\substack{1\le k\le N\\E_k\text{ not a singleton}}}
 \max\{\sigma(E_k),\nu(E_k)\},
 \label{eq:appC-partition-error}
\end{equation}
with the maximum over the empty set interpreted as zero.

\begin{lemma}[Existence of adapted partitions]
\label{lem:appC-adapted-partitions}
There exists a nested sequence
\[
 \mathcal P_1\prec\mathcal P_2\prec\cdots
\]
of ordered interval partitions of \(J\) such that

\begin{equation}
 \varepsilon(\mathcal P_m)\longrightarrow0.
 \label{eq:appC-small-cell-mass}
\end{equation}
and every atom of \(\sigma+\nu\) is a singleton cell of
\(\mathcal P_m\) for all sufficiently large \(m\).

The partitions may also be chosen so that the Euclidean diameter of
every non-singleton cell tends to zero.
\end{lemma}

\begin{proof}
Put
\[
 \eta:=\sigma+\nu.
\]
For each \(m\), the set
\[
 \mathcal A_m
 :=
 \left\{
   x\in J:
   \eta(\{x\})>2^{-m}
 \right\}
\]
is finite. Moreover,
\[
 \mathcal A_m\subset\mathcal A_{m+1},
\]
and every atom of \(\eta\) belongs to \(\mathcal A_m\) for all
sufficiently large \(m\).

Start with the endpoints \(c,d\), the points of \(\mathcal A_m\), and
all endpoints and singleton cells already used in
\(\mathcal P_{m-1}\). Isolate the points of \(\mathcal A_m\) as
singletons. On each remaining interval, subdivide successively by the
distribution function of \(\eta\) so that every resulting
non-singleton cell has \(\eta\)-mass at most \(2^{1-m}\). Additional
Euclidean subdivision points may be inserted so that its diameter is
at most \(2^{-m}\).

The resulting partition refines \(\mathcal P_{m-1}\), and every
non-singleton cell \(E\) satisfies
\[
 \sigma(E)\le\eta(E)\le2^{1-m},
 \qquad
 \nu(E)\le\eta(E)\le2^{1-m}.
\]
Hence \eqref{eq:appC-small-cell-mass} follows.
\end{proof}

For the rest of the appendix, fix an adapted sequence
\[
 \mathcal P_m
 =
 \{E_{m,1},\ldots,E_{m,N_m}\}
\]
as in Lemma~\ref{lem:appC-adapted-partitions}. To simplify notation,
the index \(m\) will occasionally be suppressed.

\subsection{Block Hardy and Copson operators}
\label{subsec:appC-block-operators}

Define the continuous closed Hardy operator
\begin{equation}
 H_\sigma h(x)
 :=
 \int_{[c,x]}h(t)\,\dd\sigma(t)
 \label{eq:appC-continuous-Hardy}
\end{equation}
and the continuous strict Copson operator
\begin{equation}
 C_\sigma h(x)
 :=
 \int_{(x,d]}h(t)\,\dd\sigma(t).
 \label{eq:appC-continuous-Copson}
\end{equation}

For \(x\in E_{m,k}\), define the block operators
\begin{equation}
 H_m^+h(x)
 :=
 \sum_{i=1}^{k}
 \int_{E_{m,i}}h\,\dd\sigma
 \label{eq:appC-block-Hardy}
\end{equation}
and
\begin{equation}
 C_m^-h(x)
 :=
 \sum_{i=k+1}^{N_m}
 \int_{E_{m,i}}h\,\dd\sigma.
 \label{eq:appC-block-Copson}
\end{equation}
The superscripts emphasise that the Hardy block operator includes the
whole current cell, whereas the Copson block operator excludes it.

Set
\begin{equation}
 D_mh(x)
 :=
 \sum_{\substack{1\le k\le N_m\\E_{m,k}\text{ not a singleton}}}
 \mathbf 1_{E_{m,k}}(x)
 \int_{E_{m,k}}h\,\dd\sigma.
 \label{eq:appC-diagonal-block}
\end{equation}

\begin{lemma}[Intracell error]
\label{lem:appC-intracell-error}
For every \(h\ge0\),
\begin{equation}
 0
 \le
 H_m^+h-H_\sigma h
 \le
 D_mh
 \label{eq:appC-Hardy-error}
\end{equation}
and
\begin{equation}
 0
 \le
 C_\sigma h-C_m^-h
 \le
 D_mh.
 \label{eq:appC-Copson-error}
\end{equation}
\end{lemma}

\begin{proof}
Fix \(x\in E_{m,k}\). Every point of \(E_{m,i}\) with \(i<k\) lies to
the left of \(x\), while every point of \(E_{m,i}\) with \(i>k\) lies
to its right. Hence the only difference between
\(H_m^+h(x)\) and \(H_\sigma h(x)\) is the portion of the current cell
lying to the right of \(x\).

If \(E_{m,k}\) is a singleton, that difference is zero because the
closed Hardy integral includes the atom at \(x\). Otherwise it is
bounded by
\[
 \int_{E_{m,k}}h\,\dd\sigma.
\]
This proves \eqref{eq:appC-Hardy-error}.

Similarly, the only difference between \(C_\sigma h(x)\) and
\(C_m^-h(x)\) is the part of the current cell lying strictly to the
right of \(x\). On a singleton cell this contribution is zero, because
the Copson kernel is strict. This proves
\eqref{eq:appC-Copson-error}.
\end{proof}

\subsection{Norm of the diagonal error}
\label{subsec:appC-diagonal-estimate}

\begin{lemma}[Small-cell diagonal estimate]
\label{lem:appC-diagonal-estimate}
Let
\[
 \delta_m:=\varepsilon(\mathcal P_m).
\]
Then
\begin{equation}
 \|D_m:L^P(\sigma)\to L^Q(\nu)\|
 \longrightarrow0.
 \label{eq:appC-diagonal-norm-zero}
\end{equation}

More precisely, if \(P\le Q\), then
\begin{equation}
 \|D_m\|
 \le
 \delta_m^{\,1/P'+1/Q},
 \label{eq:appC-diagonal-convex-bound}
\end{equation}
whereas if \(Q<P\), then
\begin{equation}
 \|D_m\|
 \le
 \delta_m\,\nu(J)^{1/R}.
 \label{eq:appC-diagonal-nonconvex-bound}
\end{equation}
\end{lemma}

\begin{proof}
For a non-singleton cell \(E_{m,k}\), Hölder's inequality gives
\[
 \int_{E_{m,k}}h\,\dd\sigma
 \le
 \sigma(E_{m,k})^{1/P'}
 \left(
   \int_{E_{m,k}}h^P\,\dd\sigma
 \right)^{1/P}.
\]
Put
\[
 z_k
 :=
 \left(
   \int_{E_{m,k}}h^P\,\dd\sigma
 \right)^{1/P}
\]
and
\[
 d_k
 :=
 \sigma(E_{m,k})^{1/P'}
 \nu(E_{m,k})^{1/Q}.
\]
Then
\[
 \|D_mh\|_{L^Q(\nu)}
 \le
 \left(
   \sum_{k}d_k^Qz_k^Q
 \right)^{1/Q},
\]
where the sum is restricted to non-singleton cells.

If \(P\le Q\), the embedding
\[
 \ell^P\hookrightarrow\ell^Q
\]
gives
\[
 \|D_mh\|_{L^Q(\nu)}
 \le
 \left(\sup_kd_k\right)
 \left(\sum_kz_k^P\right)^{1/P}.
\]
Since
\[
 \sigma(E_{m,k}),
 \nu(E_{m,k})
 \le\delta_m,
\]
we obtain \eqref{eq:appC-diagonal-convex-bound}.

Assume now that \(Q<P\). The norm of the diagonal map
\[
 (z_k)\longmapsto(d_kz_k)
\]
from \(\ell^P\) to \(\ell^Q\) is
\[
 \left(\sum_kd_k^R\right)^{1/R}.
\]
Moreover,
\[
 \begin{aligned}
 \sum_kd_k^R
 &=
 \sum_k
 \sigma(E_{m,k})^{R/P'}
 \nu(E_{m,k})^{R/Q}
 \\
 &\le
 \delta_m^{R/P'}
 \delta_m^{R/Q-1}
 \sum_k\nu(E_{m,k}).
 \end{aligned}
\]
Using
\[
 \frac{R}{P'}+\frac{R}{Q}-1=R,
\]
we obtain
\[
 \sum_kd_k^R
 \le
 \delta_m^R\nu(J).
\]
This proves \eqref{eq:appC-diagonal-nonconvex-bound}.
\end{proof}

\subsection{Equality with the finite-chain norms}
\label{subsec:appC-chain-norms}

Define the cell masses
\begin{equation}
 \sigma_{m,k}
 :=
 \sigma(E_{m,k}),
 \qquad
 \nu_{m,k}
 :=
 \nu(E_{m,k}).
 \label{eq:appC-cell-masses}
\end{equation}
Cells with \(\sigma_{m,k}=0\) may be suppressed.

Let \(\mathcal H_m\) be the optimal constant of the finite-chain Hardy
inequality
\begin{equation}
 \left[
   \sum_{k=1}^{N_m}
   \left(
     \sum_{i=1}^ka_i\sigma_{m,i}
   \right)^Q
   \nu_{m,k}
 \right]^{1/Q}
 \le
 \mathcal H_m
 \left[
   \sum_{i=1}^{N_m}
   a_i^P\sigma_{m,i}
 \right]^{1/P},
 \label{eq:appC-chain-Hardy}
\end{equation}
and let \(\mathcal C_m\) be the optimal constant in
\begin{equation}
 \left[
   \sum_{k=1}^{N_m}
   \left(
     \sum_{i=k+1}^{N_m}a_i\sigma_{m,i}
   \right)^Q
   \nu_{m,k}
 \right]^{1/Q}
 \le
 \mathcal C_m
 \left[
   \sum_{i=1}^{N_m}
   a_i^P\sigma_{m,i}
 \right]^{1/P}.
 \label{eq:appC-chain-Copson}
\end{equation}

\begin{lemma}[Exact block-chain correspondence]
\label{lem:appC-block-chain-correspondence}
One has
\begin{equation}
 \|H_m^+:L^P(\sigma)\to L^Q(\nu)\|
 =
 \mathcal H_m
 \label{eq:appC-Hardy-block-equality}
\end{equation}
and
\begin{equation}
 \|C_m^-:L^P(\sigma)\to L^Q(\nu)\|
 =
 \mathcal C_m.
 \label{eq:appC-Copson-block-equality}
\end{equation}
\end{lemma}

\begin{proof}
For \(h\ge0\), define its cell averages by
\[
 a_k
 :=
 \begin{cases}
 \displaystyle
 \frac1{\sigma_{m,k}}
 \int_{E_{m,k}}h\,\dd\sigma,
 &\sigma_{m,k}>0,
 \\[3mm]
 0,
 &\sigma_{m,k}=0.
 \end{cases}
\]
Jensen's inequality yields
\[
 \sum_{k=1}^{N_m}
 a_k^P\sigma_{m,k}
 \le
 \int_Jh^P\,\dd\sigma.
\]
Equations \eqref{eq:appC-block-Hardy} and
\eqref{eq:appC-block-Copson} then give the upper bounds in
\eqref{eq:appC-Hardy-block-equality} and
\eqref{eq:appC-Copson-block-equality}.

Conversely, every nonnegative finite sequence \((a_k)\) is realised by
the cell-constant function
\[
 h=\sum_{k=1}^{N_m}a_k\mathbf 1_{E_{m,k}}.
\]
For such a function, both the input norm and the corresponding block
output norm agree exactly with the discrete expressions. Hence the
reverse inequalities follow.
\end{proof}

\begin{theorem}[Convergence of the operator norms]
\label{thm:appC-operator-norm-convergence}
Let
\[
 \mathcal H
 :=
 \|H_\sigma:L^P(\sigma)\to L^Q(\nu)\|
\]
and
\[
 \mathcal C
 :=
 \|C_\sigma:L^P(\sigma)\to L^Q(\nu)\|.
\]
Then
\begin{equation}
 \mathcal H_m\longrightarrow\mathcal H
 \label{eq:appC-Hardy-norm-convergence}
\end{equation}
and
\begin{equation}
 \mathcal C_m\longrightarrow\mathcal C.
 \label{eq:appC-Copson-norm-convergence}
\end{equation}
\end{theorem}

\begin{proof}
By Lemmas~\ref{lem:appC-intracell-error} and
\ref{lem:appC-diagonal-estimate},
\[
 \|H_m^+-H_\sigma\|
 \le
 \|D_m\|
 \longrightarrow0
\]
and
\[
 \|C_m^--C_\sigma\|
 \le
 \|D_m\|
 \longrightarrow0.
\]
Therefore,
\[
 \left|
   \|H_m^+\|-\|H_\sigma\|
 \right|
 \le
 \|H_m^+-H_\sigma\|
 \longrightarrow0,
\]
and similarly for the strict Copson operators. Apply
Lemma~\ref{lem:appC-block-chain-correspondence}.
\end{proof}

\subsection{Cumulative functions and their block approximants}
\label{subsec:appC-cumulative-functions}

Define
\begin{equation}
 M(x):=\sigma([c,x]),
 \qquad
 N(x):=\nu([x,d]),
 \label{eq:appC-MN}
\end{equation}
and
\begin{equation}
 K(x):=\nu([c,x]),
 \qquad
 B(x):=\sigma((x,d]).
 \label{eq:appC-KB}
\end{equation}

For a fixed partition \(\mathcal P_m\), put
\begin{equation}
 A_{m,k}:=\bigcup_{i=1}^kE_{m,i},
 \qquad
 T_{m,k}:=\bigcup_{i=k}^{N_m}E_{m,i}.
 \label{eq:appC-prefix-tail-sets}
\end{equation}
Define
\begin{equation}
 M_{m,k}:=\sigma(A_{m,k}),
 \qquad
 N_{m,k}:=\nu(T_{m,k}),
 \label{eq:appC-discrete-MN}
\end{equation}
and
\begin{equation}
 K_{m,k}:=\nu(A_{m,k}),
 \qquad
 B_{m,k}
 :=
 \sigma\left(
   \bigcup_{i=k+1}^{N_m}E_{m,i}
 \right).
 \label{eq:appC-discrete-KB}
\end{equation}
We use the conventions
\[
 M_{m,0}=K_{m,0}=0,
 \qquad
 B_{m,N_m}=0.
\]

For \(x\in E_{m,k}\), define the step functions
\begin{equation}
 M_m^+(x):=M_{m,k},
 \qquad
 N_m^+(x):=N_{m,k},
 \label{eq:appC-step-MN}
\end{equation}
and
\begin{equation}
 K_m^+(x):=K_{m,k},
 \qquad
 B_m^-(x):=B_{m,k}.
 \label{eq:appC-step-KB}
\end{equation}

\begin{lemma}[Uniform cumulative approximation]
\label{lem:appC-uniform-cumulative}
As \(m\to\infty\),
\begin{equation}
 \|M_m^+-M\|_{L^\infty(J)}
 +
 \|B_m^--B\|_{L^\infty(J)}
 \longrightarrow0
 \label{eq:appC-uniform-sigma-cumulative}
\end{equation}
and
\begin{equation}
 \|N_m^+-N\|_{L^\infty(J)}
 +
 \|K_m^+-K\|_{L^\infty(J)}
 \longrightarrow0.
 \label{eq:appC-uniform-nu-cumulative}
\end{equation}
\end{lemma}

\begin{proof}
Fix \(x\in E_{m,k}\). Since \(A_{m,k}\) contains \([c,x]\),
\[
 0
 \le
 M_m^+(x)-M(x)
 \le
 \sigma(E_{m,k}).
\]
Likewise,
\[
 0
 \le
 N_m^+(x)-N(x)
 \le
 \nu(E_{m,k}),
\]
and
\[
 0
 \le
 K_m^+(x)-K(x)
 \le
 \nu(E_{m,k}).
\]
Finally,
\[
 0
 \le
 B(x)-B_m^-(x)
 \le
 \sigma(E_{m,k}).
\]

If \(E_{m,k}\) is a singleton, each of the corresponding errors is
zero. If it is not a singleton, the errors are bounded by
\(\varepsilon(\mathcal P_m)\). The conclusion follows from
\eqref{eq:appC-small-cell-mass}.
\end{proof}

\subsection{Convergence of the Hardy characteristics}
\label{subsec:appC-Hardy-characteristics}

If \(P\le Q\), define
\begin{equation}
 \mathfrak H
 :=
 \sup_{x\in J}
 M(x)^{1/P'}N(x)^{1/Q}
 \label{eq:appC-continuous-Hardy-convex}
\end{equation}
and
\begin{equation}
 \mathfrak H_m
 :=
 \max_{1\le k\le N_m}
 M_{m,k}^{1/P'}N_{m,k}^{1/Q}.
 \label{eq:appC-discrete-Hardy-convex}
\end{equation}

If \(Q<P\), define
\begin{equation}
 \mathfrak H
 :=
 \left[
   \int_J
   N(x)^{R/Q}\,
   \dd\!\left(M(x)^{R/P'}\right)
 \right]^{1/R}
 \label{eq:appC-continuous-Hardy-nonconvex}
\end{equation}
and
\begin{equation}
 \mathfrak H_m
 :=
 \left[
   \sum_{k=1}^{N_m}
   N_{m,k}^{R/Q}
   \left(
     M_{m,k}^{R/P'}
     -
     M_{m,k-1}^{R/P'}
   \right)
 \right]^{1/R}.
 \label{eq:appC-discrete-Hardy-nonconvex}
\end{equation}

The Stieltjes measure in
\eqref{eq:appC-continuous-Hardy-nonconvex} is defined by extending
\(M^{R/P'}\) by zero to the left of \(c\).

\begin{lemma}[Cell increments of the Hardy Stieltjes measure]
\label{lem:appC-Hardy-cell-increments}
Let
\[
 \rho_H
 :=
 \dd\!\left(M^{R/P'}\right).
\]
Then
\begin{equation}
 \rho_H(E_{m,k})
 =
 M_{m,k}^{R/P'}
 -
 M_{m,k-1}^{R/P'}
 \label{eq:appC-Hardy-cell-measure}
\end{equation}
for every \(m\) and \(k\).
\end{lemma}

\begin{proof}
Since \(A_{m,k}\) is an initial interval, it is either
\([c,t]\), \([c,t)\), or \(J\). By the endpoint convention for the
Lebesgue--Stieltjes measure,
\[
 \rho_H(A_{m,k})
 =
 \sigma(A_{m,k})^{R/P'}
 =
 M_{m,k}^{R/P'}.
\]
Since
\[
 E_{m,k}=A_{m,k}\setminus A_{m,k-1},
\]
subtraction gives \eqref{eq:appC-Hardy-cell-measure}.
\end{proof}

\begin{theorem}[Convergence of the Hardy characteristics]
\label{thm:appC-Hardy-characteristic-convergence}
In both exponent regimes,
\begin{equation}
 \mathfrak H_m\longrightarrow\mathfrak H.
 \label{eq:appC-Hardy-characteristic-limit}
\end{equation}
\end{theorem}

\begin{proof}
Assume first that \(P\le Q\). Since
\(M_m^+\) and \(N_m^+\) are constant on every partition cell,
\[
 \mathfrak H_m
 =
 \sup_{x\in J}
 M_m^+(x)^{1/P'}
 N_m^+(x)^{1/Q}.
\]
Lemma~\ref{lem:appC-uniform-cumulative} and uniform continuity of the
power functions on bounded intervals imply uniform convergence of the
displayed products. Therefore their suprema converge, proving
\eqref{eq:appC-Hardy-characteristic-limit}.

Assume now that \(Q<P\). By
Lemma~\ref{lem:appC-Hardy-cell-increments},
\[
 \begin{aligned}
 \mathfrak H_m^R
 &=
 \sum_{k=1}^{N_m}
 N_{m,k}^{R/Q}\rho_H(E_{m,k})
 \\
 &=
 \int_J
 N_m^+(x)^{R/Q}\,\dd\rho_H(x).
 \end{aligned}
\]
By Lemma~\ref{lem:appC-uniform-cumulative},
\[
 N_m^+{}^{\,R/Q}
 \longrightarrow
 N^{R/Q}
\]
uniformly on \(J\). Since \(\rho_H\) is finite,
\[
 \mathfrak H_m^R
 \longrightarrow
 \int_JN(x)^{R/Q}\,\dd\rho_H(x)
 =
 \mathfrak H^R.
\]
Taking \(R\)-th roots completes the proof.
\end{proof}

\subsection{Convergence of the strict Copson characteristics}
\label{subsec:appC-Copson-characteristics}

If \(P\le Q\), define
\begin{equation}
 \mathfrak C
 :=
 \sup_{x\in J}
 K(x)^{1/Q}B(x)^{1/P'}
 \label{eq:appC-continuous-Copson-convex}
\end{equation}
and
\begin{equation}
 \mathfrak C_m
 :=
 \max_{1\le k\le N_m}
 K_{m,k}^{1/Q}B_{m,k}^{1/P'}.
 \label{eq:appC-discrete-Copson-convex}
\end{equation}

If \(Q<P\), define
\begin{equation}
 \mathfrak C
 :=
 \left[
   \int_J
   B(x)^{R/P'}\,
   \dd\!\left(K(x)^{R/Q}\right)
 \right]^{1/R}
 \label{eq:appC-continuous-Copson-nonconvex}
\end{equation}
and
\begin{equation}
 \mathfrak C_m
 :=
 \left[
   \sum_{k=1}^{N_m}
   B_{m,k}^{R/P'}
   \left(
     K_{m,k}^{R/Q}
     -
     K_{m,k-1}^{R/Q}
   \right)
 \right]^{1/R}.
 \label{eq:appC-discrete-Copson-nonconvex}
\end{equation}

The Stieltjes measure in
\eqref{eq:appC-continuous-Copson-nonconvex} is defined by extending
\(K^{R/Q}\) by zero to the left of \(c\).

\begin{lemma}[Cell increments of the Copson Stieltjes measure]
\label{lem:appC-Copson-cell-increments}
Let
\[
 \rho_C
 :=
 \dd\!\left(K^{R/Q}\right).
\]
Then
\begin{equation}
 \rho_C(E_{m,k})
 =
 K_{m,k}^{R/Q}
 -
 K_{m,k-1}^{R/Q}.
 \label{eq:appC-Copson-cell-measure}
\end{equation}
\end{lemma}

\begin{proof}
The proof is identical to that of
Lemma~\ref{lem:appC-Hardy-cell-increments}, with \(\nu\) and \(K\)
in place of \(\sigma\) and \(M\).
\end{proof}

\begin{theorem}[Convergence of the strict Copson characteristics]
\label{thm:appC-Copson-characteristic-convergence}
In both exponent regimes,
\begin{equation}
 \mathfrak C_m\longrightarrow\mathfrak C.
 \label{eq:appC-Copson-characteristic-limit}
\end{equation}
\end{theorem}

\begin{proof}
When \(P\le Q\),
\[
 \mathfrak C_m
 =
 \sup_{x\in J}
 K_m^+(x)^{1/Q}
 B_m^-(x)^{1/P'}.
\]
Uniform convergence from
Lemma~\ref{lem:appC-uniform-cumulative} proves the result.

Suppose that \(Q<P\). By
Lemma~\ref{lem:appC-Copson-cell-increments},
\[
 \begin{aligned}
 \mathfrak C_m^R
 &=
 \sum_{k=1}^{N_m}
 B_{m,k}^{R/P'}\rho_C(E_{m,k})
 \\
 &=
 \int_J
 B_m^-(x)^{R/P'}\,\dd\rho_C(x).
 \end{aligned}
\]
Since
\[
 B_m^-{}^{\,R/P'}
 \longrightarrow
 B^{R/P'}
\]
uniformly and \(\rho_C\) is finite, the integrals converge to
\[
 \int_J
 B(x)^{R/P'}\,
 \dd\!\left(K(x)^{R/Q}\right).
\]
This proves \eqref{eq:appC-Copson-characteristic-limit}.
\end{proof}

\subsection{The compact measure theorem}
\label{subsec:appC-compact-measure-conclusion}

\begin{corollary}[Compact measure Hardy theorem]
\label{cor:appC-compact-measure-Hardy}
The optimal constant
\[
 \mathcal H
 =
 \|H_\sigma:L^P(\sigma)\to L^Q(\nu)\|
\]
satisfies
\begin{equation}
 \mathcal H
 \asymp_{P,Q}
 \sup_{x\in J}
 M(x)^{1/P'}N(x)^{1/Q}
 \label{eq:appC-final-Hardy-convex}
\end{equation}
when \(P\le Q\), and
\begin{equation}
 \mathcal H
 \asymp_{P,Q}
 \left[
   \int_J
   N(x)^{R/Q}\,
   \dd\!\left(M(x)^{R/P'}\right)
 \right]^{1/R}
 \label{eq:appC-final-Hardy-nonconvex}
\end{equation}
when \(Q<P\).
\end{corollary}

\begin{proof}
Apply Theorem~\ref{thm:appA-finite-chain-hardy} to every partition
chain. Its comparison constants are independent of \(m\). Pass to the
limit using Theorems~\ref{thm:appC-operator-norm-convergence} and
\ref{thm:appC-Hardy-characteristic-convergence}.
\end{proof}

\begin{corollary}[Compact measure strict Copson theorem]
\label{cor:appC-compact-measure-Copson}
The optimal constant
\[
 \mathcal C
 =
 \|C_\sigma:L^P(\sigma)\to L^Q(\nu)\|
\]
satisfies
\begin{equation}
 \mathcal C
 \asymp_{P,Q}
 \sup_{x\in J}
 K(x)^{1/Q}B(x)^{1/P'}
 \label{eq:appC-final-Copson-convex}
\end{equation}
when \(P\le Q\), and
\begin{equation}
 \mathcal C
 \asymp_{P,Q}
 \left[
   \int_J
   B(x)^{R/P'}\,
   \dd\!\left(K(x)^{R/Q}\right)
 \right]^{1/R}
 \label{eq:appC-final-Copson-nonconvex}
\end{equation}
when \(Q<P\).
\end{corollary}

\begin{proof}
Apply Theorem~\ref{thm:appB-finite-chain-Copson} and pass to the limit
using Theorems~\ref{thm:appC-operator-norm-convergence} and
\ref{thm:appC-Copson-characteristic-convergence}.
\end{proof}

\subsection{Restoring the bounded Copson weight}
\label{subsec:appC-weighted-Copson}

Let \(\mu\) be a finite Borel measure on \(J\), and let
\[
 0\le\omega<\infty
\]
be bounded and Borel measurable. Define
\begin{equation}
 \lambda(E)
 :=
 \int_E\omega(t)^{P'}\,\dd\mu(t).
 \label{eq:appC-lambda-measure}
\end{equation}
Then
\[
 \lambda(J)<\infty.
\]

Consider
\begin{equation}
 C_{\mu,\omega}h(x)
 :=
 \int_{(x,d]}
 \omega(t)h(t)\,\dd\mu(t).
 \label{eq:appC-weighted-Copson}
\end{equation}

\begin{lemma}[Continuous weighted substitution]
\label{lem:appC-continuous-weighted-substitution}
The optimal norm of
\[
 C_{\mu,\omega}:L^P(\mu)\to L^Q(\nu)
\]
equals the optimal norm of
\[
 C_\lambda b(x)
 :=
 \int_{(x,d]}b(t)\,\dd\lambda(t)
\]
from \(L^P(\lambda)\) to \(L^Q(\nu)\).
\end{lemma}

\begin{proof}
Input values on the set \(\{\omega=0\}\) do not affect the output and
may be replaced by zero. On \(\{\omega>0\}\), put
\[
 b=h\,\omega^{-1/(P-1)}.
\]
Then
\[
 |h|^P\,\dd\mu
 =
 |b|^P\,\dd\lambda
\]
and
\[
 \omega h\,\dd\mu
 =
 b\,\dd\lambda.
\]
Thus both the input norms and the outputs agree exactly. The
substitution is reversible on the active set.
\end{proof}

Since
\[
 B_\omega(x)
 :=
 \int_{(x,d]}
 \omega(t)^{P'}\,\dd\mu(t)
 =
 \lambda((x,d]),
\]
Corollary~\ref{cor:appC-compact-measure-Copson} gives the bounded-weight
strict Copson criteria used in
\CompactMeasureHardyCopsonTheoremReference. Arbitrary
extended-valued Copson weights are obtained by the monotone truncation
argument of Appendix~\ref{app:extended-regularisation}.

\begin{remark}[Preservation of common atoms]
\label{rem:appC-common-atoms}
Every fixed atom of \(\sigma+\nu\) is eventually isolated as a
singleton partition cell. On such a cell, the block Hardy operator
includes the atom, while the block Copson operator excludes it.
Therefore the passage to the limit preserves the closed-Hardy and
strict-Copson allocation exactly.
No diagonal mass is lost or counted
twice.
\end{remark}

\begin{remark}[Role of the partition argument]
\label{rem:appC-role}
The partitions are used only to transfer uniform finite-chain estimates
to finite Borel measures. They are not discretising sequences attached
to the weights, and no anti-discretisation theorem is invoked. The
operator norms and the Stieltjes characteristics converge directly
along atom-preserving ordered partitions.
\end{remark}
\subsection{Directed locally finite extension}
\label{subsec:appC-directed-extension}

The compact theorem also has the natural auxiliary form on an arbitrary
interval.  Let \(I\subset\mathbb R\) have nonempty interior, let \(\mu\) and
\(\nu\) be locally finite nonnegative Borel measures on \(I\), and let
\(1<P,Q<\infty\).  Formal infinite endpoints are not points of \(I\) and
carry no atoms.  For \(h\ge0\), put
\begin{equation}
 H_\mu h(x)
 :=
 \int_{I\cap(-\infty,x]}h(t)\,\dd\mu(t).
 \label{eq:appC-global-measure-Hardy}
\end{equation}
For a Borel \(\omega:I\to[0,\infty]\), the strict Copson operator is defined
by simultaneous compact restriction and monotone truncation of \(\omega\).

For each \(J=[c,d]\Subset I\), form the compact characteristics
\(\mathfrak H_J\) and \(\mathfrak C_{\omega,J}\) from the restrictions of
\(\mu,\nu,\omega\) to \(J\), using the closed Hardy tail and strict Copson
tail described in the compact theorem.  Set
\begin{equation}
 \mathfrak H_I^{\mathrm{dir}}
 :=
 \sup_{J\Subset I}\mathfrak H_J,
 \qquad
 \mathfrak C_{\omega,I}^{\mathrm{dir}}
 :=
 \sup_{J\Subset I}\mathfrak C_{\omega,J}.
 \label{eq:appC-directed-measure-characteristics}
\end{equation}

\begin{corollary}[Directed measure-valued Hardy--Copson theorem]
\label{cor:appC-directed-measure-Hardy-Copson}
Under these assumptions,
\begin{equation}
 \|H_\mu:L^P(\mu)\to L^Q(\nu)\|
 \asymp_{P,Q}
 \mathfrak H_I^{\mathrm{dir}},
 \label{eq:appC-directed-Hardy-result}
\end{equation}
and
\begin{equation}
 \|C_{\mu,\omega}:L^P(\mu)\to L^Q(\nu)\|
 \asymp_{P,Q}
 \mathfrak C_{\omega,I}^{\mathrm{dir}}.
 \label{eq:appC-directed-Copson-result}
\end{equation}
The comparison constants depend only on \(P,Q\).
\end{corollary}

\begin{proof}
Choose an increasing compact exhaustion \(J_m\Subset I\) with
\(\bigcup_mJ_m=I\).  For either operator \(T\), let
\[
 T_mh:=\mathbf 1_{J_m}T(\mathbf 1_{J_m}h),
\]
with the Copson weight truncated simultaneously when necessary.  For
nonnegative \(h\), \(T_mh\uparrow Th\), so monotone convergence gives
\begin{equation}
 \|T\|=\sup_m\|T_m\|.
 \label{eq:appC-directed-operator-exhaustion}
\end{equation}
Apply \CompactMeasureHardyCopsonTheoremReference, with constants independent of \(m\), and
take the supremum.  Every compact \(J\Subset I\) is contained in some
\(J_m\), so the directed characteristic is independent of the chosen
exhaustion.
\end{proof}

\section{Compact transfer of the linear Hardy theorem}
\label{app:compact-linear-transfer}

This appendix derives the compact linear Hardy theorem used in
\LinearHardySectionReference{} from the corresponding half-line
characterisation. The transfer is performed by an order-preserving
change of variables that sends the finite interval onto the positive
half-line. It preserves the Hardy operator, both weighted norms, the
output tail, and the profile Stieltjes measure, including the initial
atom in the case \(p=1\).

Throughout this appendix, let
\[
 J=[c,d],
 \qquad
 c<d,
\]
and define
\begin{equation}
 H_Jh(x)
 :=
 \int_c^x h(t)\,\dd t,
 \qquad x\in J.
 \label{eq:appD-compact-Hardy}
\end{equation}
We first assume that
\[
 0<v(x)<\infty
 \quad\text{for almost every }x\in J,
\]
and that
\[
 w\ge0,
 \qquad
 w\in L^1(J).
\]
After the exact transport is established, Subsection~\ref{subsec:appD-ceiling}
removes the finiteness assumption on \(v\) by a finite-ceiling bridge. Thus
the final compact theorem applies to the regular local weights of
\LinearHardySectionReference{}, including forbidden sets on which
\(v=+\infty\).

For \(0<q<\infty\), let \(L_{p,q}(v,w;J)\) denote the optimal constant
in
\begin{equation}
 \|H_Jh\|_{L^q(w;J)}
 \le
 L_{p,q}(v,w;J)
 \|h\|_{L^p(v;J)},
 \qquad h\ge0.
 \label{eq:appD-compact-linear-inequality}
\end{equation}

\subsection{Profiles and boundary completion}
\label{subsec:appD-profiles}

For \(1<p<\infty\), put
\begin{equation}
 \Phi_{p,v;J}(x)
 :=
 \left(
   \int_c^x v(t)^{1-p'}\,\dd t
 \right)^{1/p'},
 \qquad x\in J.
 \label{eq:appD-profile-finite-p}
\end{equation}
For \(p=1\), define the exact functional profile
\begin{equation}
 F_{v,J}(x)
 :=
 \operatorname*{ess\,sup}_{c<t<x}
 \frac1{v(t)},
 \qquad c<x\le d.
 \label{eq:appD-profile-p-one-interior}
\end{equation}
It is nondecreasing and left-continuous on \((c,d]\). Its
right-continuous Stieltjes completion is
\begin{equation}
 \Phi_{1,v;J}(x)
 :=
 \begin{cases}
 \displaystyle\lim_{y\downarrow c}F_{v,J}(y),&x=c,\\[2mm]
 \displaystyle\lim_{y\downarrow x}F_{v,J}(y),&c<x<d,\\[2mm]
 F_{v,J}(d),&x=d.
 \end{cases}
 \label{eq:appD-profile-p-one}
\end{equation}
In particular,
\begin{equation}
 \Phi_{1,v;J}(c)
 =
 \lim_{x\downarrow c}F_{v,J}(x).
 \label{eq:appD-p-one-boundary-value}
\end{equation}

For \(p=\infty\), put
\begin{equation}
 \Phi_{\infty,v;J}(x)
 :=
 \int_c^x\frac{\dd t}{v(t)}.
 \label{eq:appD-profile-infinity}
\end{equation}

The output tail is
\begin{equation}
 W_J(x)
 :=
 \int_x^d w(t)\,\dd t,
 \qquad x\in J.
 \label{eq:appD-output-tail}
\end{equation}

\begin{lemma}[Norm of the truncated integration functional]
\label{lem:appD-integration-functional}
Let \(c<x\le d\).

If \(1<p<\infty\), then
\begin{equation}
 \sup_{\|h\|_{L^p(v;J)}\le1}
 \int_c^xh(t)\,\dd t
 =
 \Phi_{p,v;J}(x).
 \label{eq:appD-functional-finite-p}
\end{equation}

If \(p=1\), then
\begin{equation}
 \sup_{\|h\|_{L^1(v;J)}\le1}
 \int_c^xh(t)\,\dd t
 =
 F_{v,J}(x).
 \label{eq:appD-functional-p-one}
\end{equation}

If \(p=\infty\), then
\begin{equation}
 \sup_{\|h\|_{L^\infty(v;J)}\le1}
 \int_c^xh(t)\,\dd t
 =
 \Phi_{\infty,v;J}(x).
 \label{eq:appD-functional-infinity}
\end{equation}
\end{lemma}

\begin{proof}
For \(1<p<\infty\), Hölder's inequality gives
\[
 \int_c^xh(t)\,\dd t
 \le
 \|h\|_{L^p(v;J)}
 \left(
   \int_c^xv(t)^{1-p'}\,\dd t
 \right)^{1/p'}.
\]
The reverse inequality follows from the usual truncated extremising
sequence proportional to
\[
 v^{1-p'}\mathbf 1_{(c,x)}.
\]

For \(p=1\),
\[
 \int_c^xh(t)\,\dd t
 \le
 \left(
   \operatorname*{ess\,sup}_{c<t<x}\frac1{v(t)}
 \right)
 \int_c^xh(t)v(t)\,\dd t.
\]
The reverse inequality follows by choosing sets of positive measure on
which \(v^{-1}\) is arbitrarily close to its essential supremum.

For \(p=\infty\), the condition
\[
 \|h\|_{L^\infty(v;J)}\le1
\]
implies
\[
 h(t)\le\frac1{v(t)}
 \qquad\text{a.e.},
\]
and equality is obtained by taking \(h=v^{-1}\).
\end{proof}

\begin{lemma}[Left/right Stieltjes completion at \(p=1\)]
\label{lem:appD-p-one-completion}
Let \(Z:J\to[0,\infty)\) be continuous. Then
\begin{equation}
 \sup_{c<x\le d}F_{v,J}(x)Z(x)
 =
 \sup_{x\in J}\Phi_{1,v;J}(x)Z(x).
 \label{eq:appD-p-one-sup-completion}
\end{equation}
For every \(r>0\), let \(\lambda_F\) denote the left-continuous
Lebesgue--Stieltjes measure associated with \(F_{v,J}^r\), with initial
value zero immediately to the left of \(c\). Then
\begin{equation}
 \lambda_F
 =
 \dd\!\left(\Phi_{1,v;J}^{\,r}\right)
 \quad\text{as Borel measures on }J.
 \label{eq:appD-p-one-measure-completion}
\end{equation}
\end{lemma}

\begin{proof}
For the first identity, \(F_{v,J}\le\Phi_{1,v;J}\). Conversely, if \(x<d\)
and \(y_n\downarrow x\), then \(F_{v,J}(y_n)\to\Phi_{1,v;J}(x)\)
and \(Z(y_n)\to Z(x)\). The endpoint \(c\) is treated in the same way,
and at \(d\) the two profiles agree. For the second, the
left-continuous Stieltjes convention gives
\[
 \lambda_F([c,b))=F_{v,J}(b)^r,
 \qquad c<b\le d.
\]
If \(x<d\), continuity from above therefore yields
\[
 \lambda_F([c,x])
 =\lim_{b\downarrow x}F_{v,J}(b)^r
 =\Phi_{1,v;J}(x)^r.
\]
At \(x=d\) this is \(F_{v,J}(d)^r=\Phi_{1,v;J}(d)^r\). The
right-continuous Stieltjes measure generated by \(\Phi_{1,v;J}^r\) has the
same values on all closed initial intervals \([c,x]\), so the two Borel
measures coincide.
\end{proof}

\begin{remark}[Meaning of the \(p=1\) boundary value]
\label{rem:appD-p-one-boundary}
Although
\[
 H_Jh(c)=0,
\]
the number \(\Phi_{1,v;J}(c)\) records the limiting norm of the
integration functional as the upper endpoint decreases to \(c\). This
quantity is therefore relevant to functions concentrating immediately
to the right of \(c\). It is encoded as an initial Stieltjes atom rather
than as the value of the Hardy operator at the singleton \(\{c\}\).
\end{remark}

Whenever \(q<p<\infty\), define
\begin{equation}
 \frac1r
 :=
 \frac1q-\frac1p.
 \label{eq:appD-r-definition}
\end{equation}
This includes the endpoint \(p=1\), necessarily with \(0<q<1\).

For \(1\le p<\infty\), let
\[
 G_{p,v;J}(x):=\Phi_{p,v;J}(x)^r.
\]
The function \(G_{p,v;J}\) is extended by zero to the left of \(c\),
and
\[
 \dd G_{p,v;J}
\]
denotes the resulting Lebesgue--Stieltjes measure on \(J\). In
particular,
\begin{equation}
 \dd G_{p,v;J}(\{c\})
 =
 \Phi_{p,v;J}(c)^r.
 \label{eq:appD-left-boundary-atom}
\end{equation}
For \(1<p<\infty\), the right-hand side vanishes. For \(p=1\), it may
be positive.

\subsection{The interval-to-half-line transport}
\label{subsec:appD-transport}

Set
\[
 \ell:=d-c
\]
and define
\begin{equation}
 \psi(s)
 :=
 c+\ell\frac{s}{1+s},
 \qquad
 0<s<\infty.
 \label{eq:appD-psi}
\end{equation}
Then \(\psi\) is an increasing \(C^1\)-diffeomorphism from
\((0,\infty)\) onto \((c,d)\), with
\begin{equation}
 \psi'(s)
 =
 \frac{\ell}{(1+s)^2}.
 \label{eq:appD-psi-derivative}
\end{equation}
It extends continuously to the compactified half-line by
\[
 \psi(0)=c,
 \qquad
 \psi(\infty)=d.
\]

For \(1\le p<\infty\), define
\begin{equation}
 (\mathcal T_ph)(s)
 :=
 h(\psi(s))\psi'(s)
 \label{eq:appD-Tp}
\end{equation}
and the transported weights
\begin{equation}
 \widetilde v_p(s)
 :=
 v(\psi(s))\psi'(s)^{1-p},
 \label{eq:appD-transported-input-weight}
\end{equation}
\begin{equation}
 \widetilde w(s)
 :=
 w(\psi(s))\psi'(s).
 \label{eq:appD-transported-output-weight}
\end{equation}

For \(p=\infty\), retain
\[
 (\mathcal T_\infty h)(s)
 :=
 h(\psi(s))\psi'(s),
\]
but define
\begin{equation}
 \widetilde v_\infty(s)
 :=
 \frac{v(\psi(s))}{\psi'(s)}.
 \label{eq:appD-infinity-input-weight}
\end{equation}

Let
\begin{equation}
 H_+f(s)
 :=
 \int_0^sf(t)\,\dd t
 \label{eq:appD-half-line-Hardy}
\end{equation}
be the forward Hardy operator on the positive half-line.

\begin{lemma}[Exact transport identities]
\label{lem:appD-exact-transport}
For every \(h\ge0\) and \(s>0\),
\begin{equation}
 H_+(\mathcal T_ph)(s)
 =
 H_Jh(\psi(s)).
 \label{eq:appD-operator-transport}
\end{equation}

If \(1\le p<\infty\), then
\begin{equation}
 \|\mathcal T_ph\|_{L^p(\widetilde v_p;(0,\infty))}
 =
 \|h\|_{L^p(v;J)}.
 \label{eq:appD-input-isometry-finite}
\end{equation}
If \(p=\infty\), then
\begin{equation}
 \|\mathcal T_\infty h\|_{L^\infty(\widetilde v_\infty;(0,\infty))}
 =
 \|h\|_{L^\infty(v;J)}.
 \label{eq:appD-input-isometry-infinity}
\end{equation}
For every \(0<q<\infty\),
\begin{equation}
 \|H_+(\mathcal T_ph)\|_{L^q(\widetilde w;(0,\infty))}
 =
 \|H_Jh\|_{L^q(w;J)}.
 \label{eq:appD-output-isometry}
\end{equation}
Consequently, the compact and transported half-line optimal constants
are equal.
\end{lemma}

\begin{proof}
A change of variables gives
\[
 \begin{aligned}
 H_+(\mathcal T_ph)(s)
 &=
 \int_0^s
 h(\psi(t))\psi'(t)\,\dd t
 \\
 &=
 \int_c^{\psi(s)}h(x)\,\dd x
 =
 H_Jh(\psi(s)).
 \end{aligned}
\]
For \(p<\infty\),
\[
 \begin{aligned}
 &\int_0^\infty
 |\mathcal T_ph(s)|^p
 \widetilde v_p(s)\,\dd s
 \\
 &\quad=
 \int_0^\infty
 |h(\psi(s))|^p
 \psi'(s)^p
 v(\psi(s))
 \psi'(s)^{1-p}\,\dd s
 \\
 &\quad=
 \int_c^d|h(x)|^pv(x)\,\dd x.
 \end{aligned}
\]
The \(p=\infty\) identity follows from
\[
 |\mathcal T_\infty h(s)|
 \widetilde v_\infty(s)
 =
 |h(\psi(s))|v(\psi(s)).
\]
Finally,
\[
 \begin{aligned}
 &\int_0^\infty
 |H_+(\mathcal T_ph)(s)|^q
 \widetilde w(s)\,\dd s
 \\
 &\quad=
 \int_0^\infty
 |H_Jh(\psi(s))|^q
 w(\psi(s))\psi'(s)\,\dd s
 \\
 &\quad=
 \int_c^d|H_Jh(x)|^qw(x)\,\dd x.
 \end{aligned}
\]
\end{proof}

\subsection{Transport of tails and profiles}
\label{subsec:appD-profile-transport}

Let
\begin{equation}
 \widetilde W(s)
 :=
 \int_s^\infty\widetilde w(t)\,\dd t.
 \label{eq:appD-half-line-tail}
\end{equation}

\begin{lemma}[Tail and profile identities]
\label{lem:appD-tail-profile-identities}
For every \(s>0\),
\begin{equation}
 \widetilde W(s)
 =
 W_J(\psi(s)).
 \label{eq:appD-tail-identity}
\end{equation}

If \(1<p<\infty\), then
\begin{equation}
 \left(
   \int_0^s
   \widetilde v_p(t)^{1-p'}\,\dd t
 \right)^{1/p'}
 =
 \Phi_{p,v;J}(\psi(s)).
 \label{eq:appD-profile-transport-finite}
\end{equation}

If \(p=1\), then the exact left-continuous profiles satisfy
\begin{equation}
 \operatorname*{ess\,sup}_{0<t<s}
 \frac1{\widetilde v_1(t)}
 =
 F_{v,J}(\psi(s)),
 \label{eq:appD-profile-transport-one}
\end{equation}
and their right-continuous completions satisfy
\begin{equation}
 \widetilde\Phi_1(s)
 =
 \Phi_{1,v;J}(\psi(s)).
 \label{eq:appD-profile-completion-transport-one}
\end{equation}

If \(p=\infty\), then
\begin{equation}
 \int_0^s\frac{\dd t}{\widetilde v_\infty(t)}
 =
 \Phi_{\infty,v;J}(\psi(s)).
 \label{eq:appD-profile-transport-infinity}
\end{equation}
\end{lemma}

\begin{proof}
The tail identity follows from
\[
 \begin{aligned}
 \widetilde W(s)
 &=
 \int_s^\infty
 w(\psi(t))\psi'(t)\,\dd t
 \\
 &=
 \int_{\psi(s)}^dw(x)\,\dd x.
 \end{aligned}
\]

For \(1<p<\infty\), observe that
\[
 (1-p)(1-p')=1.
\]
Hence
\[
 \widetilde v_p(t)^{1-p'}
 =
 v(\psi(t))^{1-p'}\psi'(t),
\]
and therefore
\[
 \int_0^s
 \widetilde v_p(t)^{1-p'}\,\dd t
 =
 \int_c^{\psi(s)}
 v(x)^{1-p'}\,\dd x.
\]

When \(p=1\),
\[
 \widetilde v_1(t)=v(\psi(t)),
\]
so the exact essential supremum is preserved by the increasing
diffeomorphism \(\psi\), proving
\eqref{eq:appD-profile-transport-one}. Taking right limits and using the
continuity and monotonicity of \(\psi\) gives
\eqref{eq:appD-profile-completion-transport-one}.

For \(p=\infty\),
\[
 \frac1{\widetilde v_\infty(t)}
 =
 \frac{\psi'(t)}{v(\psi(t))},
\]
and another change of variables completes the proof.
\end{proof}

\begin{lemma}[Transport of the Stieltjes measure]
\label{lem:appD-Stieltjes-transport}
Assume \(1\le p<\infty\) and \(q<p\), and let \(r\) be defined by
\eqref{eq:appD-r-definition}. Let
\[
 \widetilde\Phi_p(s)
 :=
 \Phi_{p,v;J}(\psi(s)),
 \qquad s>0,
\]
and set
\[
 \widetilde\Phi_p(0)
 :=
 \Phi_{p,v;J}(c).
\]
For \(p=1\), this is the right-continuous completion of the exact half-line
profile in \eqref{eq:appD-profile-transport-one}. By
Lemma~\ref{lem:appD-p-one-completion}, it generates exactly the
Lebesgue--Stieltjes measure used by the half-line theorem.
Then, for every nonnegative Borel function \(F\),
\begin{equation}
 \int_{[0,\infty)}
 F(s)\,
 \dd\!\left(\widetilde\Phi_p(s)^r\right)
 =
 \int_J
 F(\psi^{-1}(x))\,
 \dd\!\left(\Phi_{p,v;J}(x)^r\right),
 \label{eq:appD-Stieltjes-pushforward}
\end{equation}
with the natural interpretation at \(x=c\).

In particular,
\begin{equation}
 \begin{aligned}
 &\int_{[0,\infty)}
 \widetilde W(s)^{r/q}\,
 \dd\!\left(\widetilde\Phi_p(s)^r\right)
 \\
 &\qquad=
 \int_J
 W_J(x)^{r/q}\,
 \dd\!\left(\Phi_{p,v;J}(x)^r\right).
 \end{aligned}
 \label{eq:appD-characteristic-transport}
\end{equation}
\end{lemma}

\begin{proof}
The map \(\psi\) is an order-preserving homeomorphism between the
compactified intervals
\[
 [0,\infty]
 \quad\text{and}\quad
 [c,d].
\]
Moreover,
\[
 \widetilde\Phi_p^r
 =
 \Phi_{p,v;J}^r\circ\psi.
\]
The corresponding Lebesgue--Stieltjes measures are therefore related
by the pushforward under \(\psi\). The atom at \(0\) is sent to the atom
at \(c\), with
\[
 \dd\!\left(\widetilde\Phi_p^r\right)(\{0\})
 =
 \Phi_{p,v;J}(c)^r
 =
 \dd\!\left(\Phi_{p,v;J}^r\right)(\{c\}).
\]
This proves \eqref{eq:appD-Stieltjes-pushforward}. Equation
\eqref{eq:appD-characteristic-transport} follows from
Lemma~\ref{lem:appD-tail-profile-identities}.
\end{proof}

\subsection{Finite-ceiling bridge for forbidden input sets}
\label{subsec:appD-ceiling}

Let now \(v:J\to[\delta,\infty]\) be regular in the sense of
\LinearHardySectionReference{}, and put
\[
 E_\infty:=\{x\in J:v(x)=\infty\}.
\]
For \(N\ge\delta\), define the finite-ceiling replacement
\begin{equation}
 v^{[N]}(x)
 :=
 \begin{cases}
 v(x),&v(x)<\infty,\\
 N,&v(x)=\infty.
 \end{cases}
 \label{eq:appD-finite-ceiling-weight}
\end{equation}
Thus \(v^{[N]}\) is finite almost everywhere and
\(v^{[N]}\uparrow v\) in the extended sense.

\begin{lemma}[Finite-ceiling convergence]
\label{lem:appD-finite-ceiling}
Let \(1\le p<\infty\), \(0<q<\infty\), and \(w\in L^1(J)\), \(w\ge0\).
Write
\[
 L_N:=L_{p,q}(v^{[N]},w;J),
 \qquad
 L:=L_{p,q}(v,w;J).
\]
Then
\begin{equation}
 L_N\downarrow L.
 \label{eq:appD-ceiling-constant-limit}
\end{equation}
Moreover, the compact Hardy characteristics formed from \(v^{[N]}\)
converge to the corresponding characteristic formed from \(v\). The
convergence is uniform at the profile level, and in the non-convex branch
\begin{equation}
 \int_J W_J^{r/q}\,\dd\!\left(\Phi_{p,v^{[N]};J}^{\,r}\right)
 \longrightarrow
 \int_J W_J^{r/q}\,\dd\!\left(\Phi_{p,v;J}^{\,r}\right).
 \label{eq:appD-ceiling-characteristic-limit}
\end{equation}
\end{lemma}

\begin{proof}
Because \(v^{[N]}\) increases with \(N\), \(L_N\) is nonincreasing and
\(L\le L_N\). Assume \(L<\infty\), since otherwise every \(L_N\) is
infinite. If \(\|h\|_{L^p(v^{[N]};J)}\le1\), decompose
\(h=h_0+h_\infty\) with \(h_\infty=h\mathbf1_{E_\infty}\). Then
\(\|h_0\|_{L^p(v;J)}\le1\), while Hölder's inequality gives
\[
 \sup_{x\in J}H_Jh_\infty(x)
 \le
 |J|^{1/p'}N^{-1/p},
\]
with \(|J|^{1/\infty}=1\) when \(p=1\). Hence, with
\[
 \varepsilon_N
 :=|J|^{1/p'}N^{-1/p}\|1\|_{L^q(w;J)},
\]
we have \(L_N\le L+\varepsilon_N\) for \(q\ge1\), and
\(L_N^q\le L^q+\varepsilon_N^q\) for \(0<q<1\). This proves
\eqref{eq:appD-ceiling-constant-limit}.

For \(1<p<\infty\),
\[
 0\le
 \int_c^x (v^{[N]})^{1-p'}-\int_c^x v^{1-p'}
 \le |J|N^{1-p'},
\]
uniformly in \(x\), so \(\Phi_{p,v^{[N]};J}\to\Phi_{p,v;J}\)
uniformly. For \(p=1\), the exact profiles satisfy
\[
 0\le F_{v^{[N]},J}(x)-F_{v,J}(x)\le N^{-1},
\]
and the same bound holds for their right-continuous completions. Thus the
convex suprema converge uniformly.

In the non-convex branch set
\(G_N=\Phi_{p,v^{[N]};J}^{r}\),
\(G=\Phi_{p,v;J}^{r}\), and \(Z=W_J^{r/q}\). Then
\(G_N\to G\) uniformly, while \(Z\) is continuous and of bounded variation.
Stieltjes integration by parts therefore gives
\[
 \left|
 \int_J Z\,\dd G_N-
 \int_J Z\,\dd G
 \right|
 \le
 \|G_N-G\|_\infty
 \bigl(Z(d)+\operatorname{Var}_J Z\bigr)
 \longrightarrow0,
\]
which proves \eqref{eq:appD-ceiling-characteristic-limit}.
\end{proof}

\subsection{The compact linear theorem}
\label{subsec:appD-compact-theorem}

We now apply the half-line Hardy characterisation to the transported
weights. Its comparison constants depend only on \(p\) and \(q\).

\begin{theorem}[Endpoint-safe compact linear Hardy theorem]
\label{thm:appD-compact-linear-Hardy}
Let
\[
 1\le p<\infty,
 \qquad
 0<q<\infty,
\]
let \(v:J\to[\delta,\infty]\) be regular for some \(\delta>0\), and let
\(w\ge0\) be integrable on \(J\).

If \(p\le q\), then
\begin{equation}
 L_{p,q}(v,w;J)
 \asymp_{p,q}
 \sup_{x\in J}
 \Phi_{p,v;J}(x)W_J(x)^{1/q}.
 \label{eq:appD-compact-upper-characteristic}
\end{equation}

If \(q<p<\infty\), let \(r\) be defined by
\eqref{eq:appD-r-definition}. Then
\begin{equation}
 L_{p,q}(v,w;J)
 \asymp_{p,q}
 \left[
   \int_J
   W_J(x)^{r/q}\,
   \dd\!\left(\Phi_{p,v;J}(x)^r\right)
 \right]^{1/r}.
 \label{eq:appD-compact-lower-characteristic}
\end{equation}

At the endpoint \(p=1>q\), the right-hand side of
\eqref{eq:appD-compact-lower-characteristic} contains the explicit
boundary contribution
\begin{equation}
 W_J(c)^{r/q}\Phi_{1,v;J}(c)^r.
 \label{eq:appD-p-one-boundary-contribution}
\end{equation}

All comparison constants are independent of \(J\), \(v\), and \(w\).
\end{theorem}

\begin{proof}
Assume first that \(v\) is finite almost everywhere. By
Lemma~\ref{lem:appD-exact-transport}, the optimal compact constant equals
the optimal half-line constant associated with
\(\widetilde v_p,\widetilde w\). The underlying Gogatishvili--Pick half-line input is
Theorem~2.1 and its proof, as cited in \LinearHardySectionReference{}. Its
necessity and sufficiency estimates yield comparison constants that depend
only on \((p,q)\).

Suppose first that \(p\le q\). For \(p>1\), the half-line profile is the
transported \(\widetilde\Phi_p\) above. For \(p=1\), the source theorem
uses the exact left-continuous profile in
\eqref{eq:appD-profile-transport-one}.
Lemma~\ref{lem:appD-p-one-completion} converts its supremum exactly to the
right-continuous completion. Lemma~\ref{lem:appD-tail-profile-identities}
then gives
\[
 L_{p,q}(v,w;J)
 \asymp_{p,q}
 \sup_{x\in J}\Phi_{p,v;J}(x)W_J(x)^{1/q},
\]
proving \eqref{eq:appD-compact-upper-characteristic}.

Assume now that \(q<p<\infty\). If \(p=1\), the Lebesgue--Stieltjes
measure in the half-line theorem is generated by the left-continuous exact
profile. Lemma~\ref{lem:appD-p-one-completion} identifies it with
\(\dd(\widetilde\Phi_1^r)\). For \(p>1\) no completion is needed. Hence the
half-line theorem gives
\[
 \begin{aligned}
 L_{p,q}(v,w;J)
 &\asymp_{p,q}
 \Bigg[
   \int_{[0,\infty)}
   \widetilde W(s)^{r/q}\,
   \dd\!\left(\widetilde\Phi_p(s)^r\right)
 \Bigg]^{1/r}.
 \end{aligned}
\]
Lemma~\ref{lem:appD-Stieltjes-transport} now yields
\eqref{eq:appD-compact-lower-characteristic}. At \(p=1\), the atom at the
origin is transported to
\[
 \dd\!\left(\Phi_{1,v;J}^r\right)(\{c\})
 =\Phi_{1,v;J}(c)^r,
\]
whose contribution is \eqref{eq:appD-p-one-boundary-contribution}.

Finally, if \(v\) may equal \(+\infty\), apply the finite-a.e. result to
\(v^{[N]}\). Lemma~\ref{lem:appD-finite-ceiling} gives convergence of both
the optimal constants and the corresponding characteristics. Passing to the
limit preserves the same comparison constants, so they depend only on
\((p,q)\) and are independent of \(J\), \(v\), \(w\), and \(N\).
\end{proof}

\begin{corollary}[The infinite-input branch]
\label{cor:appD-infinite-input}
Let \(0<q<\infty\). Then
\begin{equation}
 L_{\infty,q}(v,w;J)
 =
 \left[
   \int_J
   \Phi_{\infty,v;J}(x)^q
   w(x)\,\dd x
 \right]^{1/q}.
 \label{eq:appD-infinity-exact}
\end{equation}
\end{corollary}

\begin{proof}
By Lemma~\ref{lem:appD-integration-functional},
\[
 H_Jh(x)
 \le
 \Phi_{\infty,v;J}(x)
 \|h\|_{L^\infty(v;J)}.
\]
Integration against \(w\) gives the upper bound in
\eqref{eq:appD-infinity-exact}. Equality follows by taking
\[
 h=\frac1v,
\]
for which
\[
 H_Jh(x)=\Phi_{\infty,v;J}(x)
\]
almost everywhere.
\end{proof}

\begin{remark}[Uniformity under freezing]
\label{rem:appD-freezing-uniformity}
The constants in
Theorem~\ref{thm:appD-compact-linear-Hardy} depend only on the exponent
pair \((p,q)\). They are therefore uniform when the output weight \(w\)
is replaced by a frozen weight of the form
\[
 w_g(x)=H_Jg(x)^qu(x).
\]
This uniformity is essential in the proof of the compact bilinear
theorem.
\end{remark}

\begin{remark}[Logical scope of the transfer]
\label{rem:appD-transfer-scope}
The argument in this appendix transfers the established half-line
linear theorem to an attained compact interval without altering its
characteristic. At \(p=1\), the exact functional profile remains the
left-continuous function \(F_{v,J}\). Only its Stieltjes generator is
right-continuously completed, and Lemma~\ref{lem:appD-p-one-completion}
proves that this changes neither the convex supremum nor the non-convex
measure. Lemma~\ref{lem:appD-finite-ceiling} separately extends the result
to forbidden \(+\infty\)-weight sets without changing the exponent-only
comparison constants.
\end{remark}
\section{Regularisation and extended weights}
\label{app:extended-regularisation}

This appendix gives the operational interpretation of weights taking
values in \([0,\infty]\). It proves the monotone approximation of the
local bilinear norm, justifies the definition of the operational
characteristics, records the weighted \(L^\infty\) conventions, and
removes the boundedness assumption on the Copson weight imposed
temporarily in Appendices~\ref{app:copson-reversal} and
\ref{app:partition-approximation}.

Throughout this appendix, let
\[
 J=[c,d],
 \qquad
 0<q<\infty,
 \qquad
 1\le p_1,p_2\le\infty,
\]
and define
\[
 \mathcal B_J(f,g)(x)
 :=
 H_Jf(x)\starprod H_Jg(x),
 \qquad
 H_Jh(x)
 :=
 \int_c^x h(t)\,\dd t.
\]

\subsection{Extended weighted functionals}
\label{subsec:appE-extended-functionals}

For \(a,b\in[0,\infty]\), define the extended product
\(a\starprod b\) by the ordinary product whenever it is unambiguous,
together with the conventions
\begin{equation}
 0\starprod\infty
 =
 \infty\starprod0
 :=
 0.
 \label{eq:appE-star-product}
\end{equation}
Thus a function supported where the weight vanishes has zero weighted
cost, whereas a positive function on a set where the weight is infinite
has infinite weighted cost.

For \(0<s<\infty\), define
\begin{equation}
 \|h\|_{L^s(w;J)}
 :=
 \left(
   \int_J |h(x)|^s\starprod w(x)\,\dd x
 \right)^{1/s}.
 \label{eq:appE-finite-weighted-functional}
\end{equation}
When \(0<s<1\), this is understood as a quasi-norm.

For \(s=\infty\), define
\begin{equation}
 \|h\|_{L^\infty(w;J)}
 :=
 \operatorname*{ess\,sup}_{x\in J}
 \bigl(|h(x)|\starprod w(x)\bigr).
 \label{eq:appE-infinite-weighted-functional}
\end{equation}

The corresponding weighted spaces consist of measurable functions for
which these quantities are finite, modulo the natural zero-cost
identification induced by the weight.

\begin{lemma}[Truncation of the function]
\label{lem:appE-function-truncation}
Let
\[
 h_N:=\min\{|h|,N\}.
\]
Then, for every \(0<s\le\infty\),
\begin{equation}
 \|h_N\|_{L^s(w;J)}
 \uparrow
 \|h\|_{L^s(w;J)}
 \qquad
 \text{as }N\to\infty.
 \label{eq:appE-function-truncation-limit}
\end{equation}
\end{lemma}

\begin{proof}
For \(0<s<\infty\),
\[
 h_N^s\starprod w
 \uparrow
 |h|^s\starprod w
\]
pointwise under convention~\eqref{eq:appE-star-product}. The assertion
therefore follows from the monotone convergence theorem.

For \(s=\infty\), the functions
\[
 h_N\starprod w
\]
increase pointwise to
\[
 |h|\starprod w.
\]
The essential suprema of an increasing sequence converge to the
essential supremum of its pointwise limit.
\end{proof}

\begin{remark}[Zero-cost directions]
\label{rem:appE-zero-cost-directions}
Suppose that
\[
 \|f\|_{L^{p_1}(v_1;J)}=0
\]
and that, for some finite-norm function \(g\),
\[
 \mathcal B_J(f,g)\ne0
\]
on a set of positive \(u\)-measure. Then the bilinear operator norm is
infinite. This phenomenon is not removed by
\eqref{eq:appE-star-product}. It is detected by the regularised
constants introduced below.
\end{remark}

\subsection{Positive-floor continuity}
\label{subsec:appE-positive-floor}

For a weight \(v:J\to[0,\infty]\), define
\begin{equation}
 v_m:=v+\frac1m.
 \label{eq:appE-positive-floor-weight}
\end{equation}
Here
\[
 \infty+\frac1m=\infty.
\]
Thus \(v_m\) is strictly positive wherever \(v\) is finite, while an
infinite-weight region remains forbidden.

\begin{lemma}[Positive-floor convergence for bounded functions]
\label{lem:appE-positive-floor-convergence}
Let \(h\) be bounded and measurable.

If \(1\le p<\infty\), then
\begin{equation}
 \|h\|_{L^p(v_m;J)}
 \downarrow
 \|h\|_{L^p(v;J)}.
 \label{eq:appE-positive-floor-finite-p}
\end{equation}

If \(p=\infty\), then
\begin{equation}
 \|h\|_{L^\infty(v_m;J)}
 \downarrow
 \|h\|_{L^\infty(v;J)}.
 \label{eq:appE-positive-floor-infinity}
\end{equation}
\end{lemma}

\begin{proof}
Assume first that \(p<\infty\). On the set \(\{v<\infty\}\),
\[
 |h|^p\starprod v_m
 =
 |h|^p\starprod v+\frac{|h|^p}{m},
\]
whereas on \(\{v=\infty\}\) the two extended products agree. Hence
\begin{equation}
 \begin{aligned}
 \|h\|_{L^p(v_m;J)}^p
 &=
 \|h\|_{L^p(v;J)}^p
 \\
 &\quad+
 \frac1m
 \int_{\{v<\infty\}}|h(x)|^p\,\dd x,
 \end{aligned}
 \label{eq:appE-positive-floor-identity}
\end{equation}
with the usual interpretation when the first term is infinite. Since
\(h\) is bounded and \(J\) has finite length, the second term converges
to zero.

For \(p=\infty\), put
\[
 M:=\|h\|_{L^\infty(J)}.
\]
Pointwise,
\[
 |h|\starprod v
 \le
 |h|\starprod v_m
 \le
 |h|\starprod v+\frac{M}{m}.
\]
Taking essential suprema gives
\[
 \|h\|_{L^\infty(v;J)}
 \le
 \|h\|_{L^\infty(v_m;J)}
 \le
 \|h\|_{L^\infty(v;J)}+\frac{M}{m}.
\]
This proves \eqref{eq:appE-positive-floor-infinity}.
\end{proof}

\begin{remark}
\label{rem:appE-unbounded-positive-floor}
For an arbitrary unbounded function \(h\), direct convergence
\(\|h\|_{L^p(v_m;J)}\to\|h\|_{L^p(v;J)}\) need not hold. What is used below
is the iterated truncation identity
\[
 \|h\|_{L^p(v;J)}
 =
 \sup_{N\ge1}\lim_{m\to\infty}
 \|h\wedge N\|_{L^p(v_m;J)},
 \qquad 1\le p\le\infty,
\]
with extended values allowed. For each fixed \(N\), the inner limit follows
from Lemma~\ref{lem:appE-positive-floor-convergence}.
The outer supremum is
Lemma~\ref{lem:appE-function-truncation}. No unweighted integrability
assumption on \(h\) is required.
\end{remark}

\subsection{Simultaneous regularisation of the bilinear inequality}
\label{subsec:appE-simultaneous-regularisation}

Let
\[
 u,v_1,v_2:J\longrightarrow[0,\infty]
\]
be measurable. For \(m\ge1\), define
\begin{equation}
 u_m:=u\wedge m,
 \qquad
 v_{i,m}:=v_i+\frac1m,
 \qquad
 i=1,2.
 \label{eq:appE-simultaneous-weights}
\end{equation}

Define \(C_J\) as the infimum of all finite constants \(C\ge0\) such that
\begin{equation}
 \|\mathcal B_J(f,g)\|_{L^q(u;J)}
 \le
 C
 \|f\|_{L^{p_1}(v_1;J)}
 \|g\|_{L^{p_2}(v_2;J)}
 \label{eq:appE-original-bilinear-inequality}
\end{equation}
for all nonnegative \(f,g\) having finite input functionals, with
\(\inf\varnothing:=+\infty\).

Likewise, let \(C_{J,m}\) be the least constant in
\begin{equation}
 \|\mathcal B_J(f,g)\|_{L^q(u_m;J)}
 \le
 C_{J,m}
 \|f\|_{L^{p_1}(v_{1,m};J)}
 \|g\|_{L^{p_2}(v_{2,m};J)}.
 \label{eq:appE-regularised-bilinear-inequality}
\end{equation}

\begin{lemma}[Monotonicity of the regularised constants]
\label{lem:appE-monotone-constants}
For every \(m\ge1\),
\begin{equation}
 C_{J,m}\le C_{J,m+1}\le C_J.
 \label{eq:appE-constant-monotonicity}
\end{equation}
Consequently, the limit
\begin{equation}
 C_{J,\infty}
 :=
 \lim_{m\to\infty}C_{J,m}
 =
 \sup_{m\ge1}C_{J,m}
 \label{eq:appE-regularised-limit-constant}
\end{equation}
exists in \([0,\infty]\).
\end{lemma}

\begin{proof}
Since
\[
 u_m\le u_{m+1}
\]
and
\[
 v_{i,m}\ge v_{i,m+1},
\]
we have
\[
 \|\mathcal B_J(f,g)\|_{L^q(u_m;J)}
 \le
 \|\mathcal B_J(f,g)\|_{L^q(u_{m+1};J)}
\]
and
\[
 \|f\|_{L^{p_i}(v_{i,m+1};J)}
 \le
 \|f\|_{L^{p_i}(v_{i,m};J)}.
\]
Thus every admissible constant for level \(m+1\) is admissible for
level \(m\), which proves
\[
 C_{J,m}\le C_{J,m+1}.
\]

Moreover,
\[
 u_m\le u
 \qquad\text{and}\qquad
 v_i\le v_{i,m}.
\]
Every admissible constant for
\eqref{eq:appE-original-bilinear-inequality} is therefore admissible
for \eqref{eq:appE-regularised-bilinear-inequality}, and hence
\[
 C_{J,m}\le C_J.
\]
\end{proof}

\begin{theorem}[Recovery of the extended-weight norm]
\label{thm:appE-regularisation-limit}
One has
\begin{equation}
 C_{J,m}\uparrow C_J.
 \label{eq:appE-Cm-to-C}
\end{equation}
Equivalently,
\begin{equation}
 C_J=\sup_{m\ge1}C_{J,m}.
 \label{eq:appE-C-sup-Cm}
\end{equation}
\end{theorem}

\begin{proof}
By Lemma~\ref{lem:appE-monotone-constants},
\[
 C_{J,\infty}\le C_J.
\]
It remains to prove the reverse inequality. If
\(C_{J,\infty}=+\infty\), then
\(C_J\le C_{J,\infty}\) is automatic. Hence assume
\(C_{J,\infty}<\infty\).

Let \(f,g\ge0\) satisfy
\[
 \|f\|_{L^{p_1}(v_1;J)}<\infty,
 \qquad
 \|g\|_{L^{p_2}(v_2;J)}<\infty.
\]
For \(N\ge1\), set
\[
 f_N:=f\wedge N,
 \qquad
 g_N:=g\wedge N.
\]
The functions \(f_N\) and \(g_N\) are bounded, and their regularised
input functionals are finite.

For every \(m\),
\begin{equation}
 \begin{aligned}
 &\|\mathcal B_J(f_N,g_N)\|_{L^q(u_m;J)}
 \\
 &\qquad\le
 C_{J,m}
 \|f_N\|_{L^{p_1}(v_{1,m};J)}
 \|g_N\|_{L^{p_2}(v_{2,m};J)}.
 \end{aligned}
 \label{eq:appE-truncated-level-m}
\end{equation}
As \(m\to\infty\),
\[
 u_m\uparrow u,
\]
and therefore
\begin{equation}
 \|\mathcal B_J(f_N,g_N)\|_{L^q(u_m;J)}
 \uparrow
 \|\mathcal B_J(f_N,g_N)\|_{L^q(u;J)}.
 \label{eq:appE-output-m-limit}
\end{equation}
By Lemma~\ref{lem:appE-positive-floor-convergence},
\[
 \|f_N\|_{L^{p_1}(v_{1,m};J)}
 \downarrow
 \|f_N\|_{L^{p_1}(v_1;J)}
\]
and similarly for \(g_N\). Since
\[
 C_{J,m}\uparrow C_{J,\infty},
\]
passage to the limit in \eqref{eq:appE-truncated-level-m} gives
\begin{equation}
 \begin{aligned}
 &\|\mathcal B_J(f_N,g_N)\|_{L^q(u;J)}
 \\
 &\qquad\le
 C_{J,\infty}
 \|f_N\|_{L^{p_1}(v_1;J)}
 \|g_N\|_{L^{p_2}(v_2;J)}.
 \end{aligned}
 \label{eq:appE-truncated-original-weights}
\end{equation}

Now
\[
 H_Jf_N(x)\uparrow H_Jf(x),
 \qquad
 H_Jg_N(x)\uparrow H_Jg(x),
\]
and consequently
\[
 \mathcal B_J(f_N,g_N)(x)
 \uparrow
 \mathcal B_J(f,g)(x).
\]
The monotone convergence theorem gives
\[
 \|\mathcal B_J(f_N,g_N)\|_{L^q(u;J)}
 \uparrow
 \|\mathcal B_J(f,g)\|_{L^q(u;J)}.
\]
By Lemma~\ref{lem:appE-function-truncation},
\[
 \|f_N\|_{L^{p_1}(v_1;J)}
 \uparrow
 \|f\|_{L^{p_1}(v_1;J)}
\]
and similarly for \(g\). Passing to \(N\to\infty\) in
\eqref{eq:appE-truncated-original-weights} yields
\[
 \|\mathcal B_J(f,g)\|_{L^q(u;J)}
 \le
 C_{J,\infty}
 \|f\|_{L^{p_1}(v_1;J)}
 \|g\|_{L^{p_2}(v_2;J)}.
\]
Thus \(C_J\le C_{J,\infty}\), completing the proof.
\end{proof}

\begin{remark}[Why both truncations are required]
\label{rem:appE-simultaneous-necessity}
Truncating only the output weight does not detect zero-cost input
directions. Flooring only the input weights does not control an
unbounded output weight. The simultaneous choice
\[
 u_m=u\wedge m,
 \qquad
 v_{i,m}=v_i+\frac1m
\]
gives
\[
 C_{J,m}\uparrow C_J
\]
without requiring separate convergence assumptions on the associated
profiles or Stieltjes measures.
\end{remark}

\subsection{Operational characteristics}
\label{subsec:appE-operational-characteristics}

For each \(m\), form the local profiles, output tails and Stieltjes
measures from the regularised triple
\[
 (u_m,v_{1,m},v_{2,m}).
\]
Let
\[
 \mathfrak A_J^{(m)}
\]
denote the complete local characteristic prescribed by the relevant
parameter regime in \LocalCharacteristicsSectionReference. Thus
\(\mathfrak A_J^{(m)}\) denotes the entire regime expression, including
all of its summands.

\begin{definition}[Operational local characteristic]
\label{def:appE-operational-characteristic}
Define
\begin{equation}
 \mathfrak A_J^{\mathrm{op}}
 :=
 \sup_{m\ge1}\mathfrak A_J^{(m)}.
 \label{eq:appE-operational-characteristic}
\end{equation}
\end{definition}

\begin{theorem}[Operational extension principle]
\label{thm:appE-operational-extension}
Suppose that, for every \(m\),
\begin{equation}
 c_{p_1,p_2,q}\,
 \mathfrak A_J^{(m)}
 \le
 C_{J,m}
 \le
 C_{p_1,p_2,q}\,
 \mathfrak A_J^{(m)},
 \label{eq:appE-uniform-regular-level-comparison}
\end{equation}
where the comparison constants are independent of \(m\). Then
\begin{equation}
 c_{p_1,p_2,q}\,
 \mathfrak A_J^{\mathrm{op}}
 \le
 C_J
 \le
 C_{p_1,p_2,q}\,
 \mathfrak A_J^{\mathrm{op}}.
 \label{eq:appE-operational-comparison}
\end{equation}

If the regular-level formula is exact, then
\begin{equation}
 C_J=\mathfrak A_J^{\mathrm{op}}.
 \label{eq:appE-operational-exact}
\end{equation}
\end{theorem}

\begin{proof}
Take the supremum over \(m\) in
\eqref{eq:appE-uniform-regular-level-comparison} and use
Theorem~\ref{thm:appE-regularisation-limit}.
\end{proof}

\begin{remark}[Complete regime expressions]
\label{rem:appE-complete-regime-expression}
Suppose that the regular-level condition has two nonnegative
components, so that
\[
 \mathfrak A_J^{(m)}
 =
 X_J^{(m)}+Y_J^{(m)}.
\]
The operational characteristic is
\[
 \sup_m
 \left(
   X_J^{(m)}+Y_J^{(m)}
 \right),
\]
not the termwise expression
\[
 \sup_mX_J^{(m)}+\sup_mY_J^{(m)}.
\]
The two expressions are nevertheless comparable, with
\begin{equation}
 \begin{aligned}
 \sup_m
 \left(
   X_J^{(m)}+Y_J^{(m)}
 \right)
 &\le
 \sup_mX_J^{(m)}
 +
 \sup_mY_J^{(m)}
 \\
 &\le
 2\sup_m
 \left(
   X_J^{(m)}+Y_J^{(m)}
 \right).
 \end{aligned}
 \label{eq:appE-sum-sup-comparison}
\end{equation}
The complete-expression definition is preferable because it preserves
the regime structure and does not combine components attained at
different regularisation levels.
\end{remark}

\begin{remark}[No profile-measure limit is required]
\label{rem:appE-no-profile-limit}
The proof of
Theorem~\ref{thm:appE-operational-extension} does not require
pointwise convergence of the profiles or weak convergence of their
Stieltjes measures. In particular, no assertion is made that
\[
 \dd\!\left(\Phi_{i,J}^{(m)}\right)^{r_i}
\]
converges to a formally defined measure associated with an
unregularised extended-valued profile. The extension is made entirely
at the level of the operator constants and the uniformly equivalent
regularised characteristics.
\end{remark}

\subsection{Infinite input exponents}
\label{subsec:appE-infinite-inputs}

When \(p_i=\infty\), positive-floor regularisation is interpreted
through
\[
 \|h\|_{L^\infty(v_{i,m};J)}
 =
 \operatorname*{ess\,sup}_{x\in J}
 \left(
   |h(x)|\starprod
   \left(v_i(x)+\frac1m\right)
 \right).
\]
Lemma~\ref{lem:appE-positive-floor-convergence} applies without change
to bounded truncations.

The corresponding integration profile is
\begin{equation}
 \Phi_{\infty,v_{i,m};J}(x)
 :=
 \int_c^x
 \frac{\dd t}{v_i(t)+m^{-1}},
 \label{eq:appE-regularised-infinity-profile}
\end{equation}
where
\[
 \frac1\infty:=0.
\]
Thus regions on which \(v_i=\infty\) contribute neither admissible
input mass nor profile growth.

\begin{proposition}[Exact doubly infinite operational formula]
\label{prop:appE-doubly-infinite}
Suppose that
\[
 p_1=p_2=\infty.
\]
Then
\begin{equation}
 \begin{aligned}
 C_J
 &=
 \sup_{m\ge1}
 \Bigg[
   \int_J
   \left(
     \Phi_{\infty,v_{1,m};J}(x)
     \Phi_{\infty,v_{2,m};J}(x)
   \right)^q
 \\
 &\hspace{34mm}\times
   u_m(x)\,\dd x
 \Bigg]^{1/q}.
 \end{aligned}
 \label{eq:appE-doubly-infinite-operational}
\end{equation}
\end{proposition}

\begin{proof}
At every regularisation level, the doubly infinite branch is exact.
Apply Theorem~\ref{thm:appE-operational-extension}.
\end{proof}

The same argument applies to either one-infinite-input branch.
The
finite-input factor is characterised by the compact linear theorem,
whereas the infinite-input factor enters through its exact integration
profile. The comparison constants remain uniform in \(m\).

\subsection{Removal of the bounded Copson-weight assumption}
\label{subsec:appE-unbounded-Copson-weight}

Let \(\mu\) and \(\nu\) be finite nonnegative Borel measures on \(J\),
let
\[
 1<P,Q<\infty,
\]
and let
\[
 \omega:J\longrightarrow[0,\infty]
\]
be measurable. Define
\begin{equation}
 C_{\mu,\omega}h(x)
 :=
 \int_{(x,d]}
 \omega(t)h(t)\,\dd\mu(t).
 \label{eq:appE-general-weighted-Copson}
\end{equation}
For \(n\ge1\), put
\begin{equation}
 \omega_n:=\omega\wedge n.
 \label{eq:appE-Copson-weight-truncation}
\end{equation}
Let \(\mathcal C(\omega)\) and \(\mathcal C(\omega_n)\) denote the
corresponding operator norms from \(L^P(\mu)\) to \(L^Q(\nu)\).

\begin{theorem}[Monotone removal of the Copson truncation]
\label{thm:appE-Copson-truncation}
One has
\begin{equation}
 \mathcal C(\omega_n)
 \uparrow
 \mathcal C(\omega).
 \label{eq:appE-Copson-norm-limit}
\end{equation}
\end{theorem}

\begin{proof}
Since
\[
 \omega_n\uparrow\omega,
\]
we have, for every \(h\ge0\),
\[
 C_{\mu,\omega_n}h(x)
 \uparrow
 C_{\mu,\omega}h(x).
\]
It follows immediately that
\[
 \mathcal C(\omega_n)
 \le
 \mathcal C(\omega_{n+1})
 \le
 \mathcal C(\omega).
\]
Let
\[
 L:=\sup_n\mathcal C(\omega_n).
\]
For every \(h\ge0\),
\[
 \|C_{\mu,\omega_n}h\|_{L^Q(\nu)}
 \le
 L\|h\|_{L^P(\mu)}.
\]
Passing to the limit by monotone convergence gives
\[
 \|C_{\mu,\omega}h\|_{L^Q(\nu)}
 \le
 L\|h\|_{L^P(\mu)}.
\]
Hence
\[
 \mathcal C(\omega)\le L,
\]
which proves \eqref{eq:appE-Copson-norm-limit}.
\end{proof}

For \(x\in J\), define
\begin{equation}
 K(x):=\nu([c,x])
 \label{eq:appE-Copson-K}
\end{equation}
and
\begin{equation}
 B_\omega(x)
 :=
 \int_{(x,d]}
 \omega(t)^{P'}\,\dd\mu(t).
 \label{eq:appE-Copson-B}
\end{equation}
Likewise,
\begin{equation}
 B_{\omega_n}(x)
 :=
 \int_{(x,d]}
 \omega_n(t)^{P'}\,\dd\mu(t).
 \label{eq:appE-Copson-Bn}
\end{equation}
Then
\begin{equation}
 B_{\omega_n}(x)\uparrow B_\omega(x)
 \qquad
 \text{for every }x\in J.
 \label{eq:appE-Copson-tail-limit}
\end{equation}

\begin{proposition}[Convergence of the Copson characteristics]
\label{prop:appE-Copson-characteristic-limit}
If \(P\le Q\), then
\begin{equation}
 \begin{aligned}
 &\sup_{x\in J}
 K(x)^{1/Q}
 B_{\omega_n}(x)^{1/P'}
 \\
 &\qquad\uparrow
 \sup_{x\in J}
 K(x)^{1/Q}
 B_\omega(x)^{1/P'}.
 \end{aligned}
 \label{eq:appE-Copson-convex-limit}
\end{equation}

If \(Q<P\), define \(R\) by
\[
 \frac1R=\frac1Q-\frac1P.
\]
Then
\begin{equation}
 \begin{aligned}
 &\int_J
 B_{\omega_n}(x)^{R/P'}\,
 \dd\!\left(K(x)^{R/Q}\right)
 \\
 &\qquad\uparrow
 \int_J
 B_\omega(x)^{R/P'}\,
 \dd\!\left(K(x)^{R/Q}\right).
 \end{aligned}
 \label{eq:appE-Copson-nonconvex-limit}
\end{equation}
\end{proposition}

\begin{proof}
In the case \(P\le Q\), the functions
\[
 K^{1/Q}B_{\omega_n}^{1/P'}
\]
increase pointwise to
\[
 K^{1/Q}B_\omega^{1/P'}.
\]
For an increasing family of nonnegative functions,
\[
 \sup_n\sup_xF_n(x)
 =
 \sup_x\sup_nF_n(x),
\]
which proves \eqref{eq:appE-Copson-convex-limit}.

If \(Q<P\), apply the monotone convergence theorem with respect to the
finite Lebesgue--Stieltjes measure
\[
 \dd\!\left(K^{R/Q}\right).
\]
\end{proof}

Combining
Theorem~\ref{thm:appE-Copson-truncation},
Proposition~\ref{prop:appE-Copson-characteristic-limit}, and
Corollary~\ref{cor:appC-compact-measure-Copson} removes the boundedness
assumption imposed temporarily in
Appendix~\ref{app:partition-approximation}.

\subsection{The global operational characteristic}
\label{subsec:appE-global-operational}

For an arbitrary interval \(I\), define
\begin{equation}
 \mathfrak A_I^{\mathrm{dir}}
 :=
 \sup_{J\Subset I}
 \mathfrak A_J^{\mathrm{op}}
 =
 \sup_{J\Subset I}
 \sup_{m\ge1}
 \mathfrak A_J^{(m)}.
 \label{eq:appE-global-operational-characteristic}
\end{equation}
Since
\[
 C_I=\sup_{J\Subset I}C_J
\]
and
\[
 C_J
 \asymp_{p_1,p_2,q}
 \mathfrak A_J^{\mathrm{op}}
\]
uniformly in \(J\), it follows that
\begin{equation}
 C_I
 \asymp_{p_1,p_2,q}
 \mathfrak A_I^{\mathrm{dir}}.
 \label{eq:appE-global-operational-comparison}
\end{equation}

The construction must therefore be performed in the order
\[
 \text{regularise the weights}
 \longrightarrow
 \text{form the complete local characteristic}
 \longrightarrow
 \sup_{m\ge1}
 \longrightarrow
 \sup_{J\Subset I}.
\]
This order preserves the endpoint atoms, the strict/closed tail
allocation, and the complete parameter-regime structure.

\section*{Use of generative AI tools}

Generative AI tools, including ChatGPT (OpenAI), were used as assistive tools during manuscript preparation for mathematical consistency checking, language refinement, literature organisation, and LaTeX preparation. The author independently reviewed the mathematical arguments, references, and final manuscript and takes full responsibility for all statements, proofs, and conclusions.

\bibliographystyle{plain-doi}
\bibliography{references}

\end{document}